\documentclass[a4paper,10pt]{article}
\title{$A_{\infty}$-algebras, spectral sequences and exact couples}
\author{Estanislao Herscovich
\footnote{Departamento de Matem\'atica, FCEyN, Universidad de Buenos Aires, Argentina. 
The author is also a research member of CONICET (Argentina). 
On leave of absence from Institut Joseph Fourier, Universit\'e Grenoble I, France. 
The author would also like to thank the Alexander von Humboldt Foundation and the Bielefeld University, for their support during part of this work. This work was also partially supported by UBACYT 20020130200169BA and PICT 2012-1186.}}
\date{}

\usepackage{amsmath,amsthm,amsfonts,amssymb,stmaryrd}
\usepackage{latexsym}
\usepackage{fullpage}
\usepackage{a4,color,palatino,fancyhdr}
\usepackage{graphicx}
\usepackage{float}  
\usepackage[small, bf, margin=90pt, tableposition=bottom]{caption}
\RequirePackage[T1]{fontenc}
\usepackage[breaklinks, naturalnames]{hyperref}
\usepackage{amsrefs}

\input{xy}
\xyoption{all}
\xyoption{poly}
\usepackage[all]{xy}

\newtheorem{theorem}{Theorem}[section]
\newtheorem{proposition}[theorem]{Proposition}

\newtheorem{lemma}[theorem]{Lemma}

\newtheorem{remark}[theorem]{Remark}
\newtheorem{example}[theorem]{Example}

\numberwithin{equation}{section}

\def\Ker{\mathop{\rm Ker}\nolimits}

\newcommand\ZZ{{\mathbb{Z}}}
\newcommand\NN{{\mathbb{N}}}

\begin{document}

\maketitle
                                                     
\hrulefill
\begin{abstract}
We study in this article a possible further structure of homotopic nature on multiplicative spectral sequences. 
More precisely, since Kadeishvili's theorem asserts that, given a dg (or $A_{\infty}$-)algebra, its cohomology has also a structure of $A_{\infty}$-algebra such that both become quasi-isomorphic, and in a multiplicative spectral sequence one considers the cohomology of dg algebras when moving from a term to the next one, a natural problem that arises is to study how this two possible structures intertwine. 
We give such a homotopic structure proposal, called \emph{$A_{\infty}$-enhancement} of multiplicative spectral sequences, which could be of interest in our opinion. 
As far we know, this construction was studied only recently by S. Lapin, even though he did not state any definition. 
Seeing that the procedure considered by Lapin is rather complicated to handle, we propose an equivalent but in our opinion easier approach. 
In particular, from our definition we show that the canonical multiplicative spectral sequence obtained from a filtered dg (or $A_{\infty}$-)algebra, which could be viewed as the main example, has such an $A_{\infty}$-enhancement. 
\end{abstract}

\textbf{Mathematics subject classification 2010:} 14D15, 16E45, 16S80, 16W70, 16W80, 18G40, 18G55.

\textbf{Keywords:} homological algebra, spectral sequences, dg algebras, $A_{\infty}$-algebras.

\hrulefill
\section{Introduction}

The aim of this article is to shed some light on the (strongly) homotopy algebraic structure on spectral sequences. 
It does not aim to give a full answer but to humbly provide some ingredients, which may be of interest. 
We first describe where our motivation to this work comes from. 
Spectral sequences were first defined by J. Leray in his seminal paper \cite{Ler46}, as a clever procedure to deal with cohomological computations in a clean manner. 
According to \cite{McC99}, Leray's work was closer to what we would now call multiplicative spectral sequences, which are spectral sequences provided with a product turning each term 
of the spectral sequence into a dg algebra. 
We recall that the underlying algebra of each successive term of a multiplicative spectral sequence is obtained by taking cohomology of the dg algebra given by the previous term. 
On the other hand, as first noticed by T. Kadeishvili in \cite{K82}, the cohomology of an $A_{\infty}$-algebra has a minimal $A_{\infty}$-algebra structure (\textit{i.e.} with vanishing differential) 
extending the canonical product on cohomology such that both $A_{\infty}$-algebras become quasi-isomorphic. 
This in particular applies to dg algebras, which are just a particular case of $A_{\infty}$-algebras. 
The natural question to us would be thus how these two operations can be put together, or in other words, how we could enhance the dg algebras structures of each term of a multiplicative spectral sequence 
in order to obtain certain strongly homotopy algebraic compatibility on it. 
To the best of our knowledge this question has only deserved little attention. 
In fact, despite it was not considered in these terms as far as we understand, a situation handling both structures (\textit{i.e.} multiplicative spectral sequences and $A_{\infty}$-algebra structures) 
has already been considered by S. Lapin in \cite{La02a}, and with several variations in some other of his successive articles, even though he does not seem to explicitly state any definition of a so-called $A_{\infty}$-enhancement, nor to study many of its properties.  
He did however considered a procedure applied to (deformations of) $A_{\infty}$-algebras which is of key importance in our opinion, although it does not appear to have been studied in detail 
(for instance, there does not seem to be any discussion if the canonical multiplicative spectral sequence associated to a filtration of a dg algebra obtained 
by the usual theory coincides with his construction at \cite{La08}). 
This procedure allow us to propose what an $A_{\infty}$-enhancement of multiplicative spectral sequences should be. 
The proposed name pours from the tradition of enhancements of triangulated categories, due to the intention of remarking the addition of an extra structure, 
though probably the situation with multiplicative spectral sequences is by far not as complicated as the one with triangulated categories. 
We moreover propose an equivalent but easier definition which is simpler to handle, and from which one can prove several interesting properties. 
In particular, one gets from it that the (canonical) multiplicative spectral sequence associated to a filtered dg (or $A_{\infty}$-)algebra, which is the main example one may think of, has an $A_{\infty}$-enhancement 
(see Propositions \ref{proposition:compatres} and \ref{proposition:final}). 
Another interesting feature of our proposal is that it stays in the world of dg algebras if that is the departure point, in open contrast to the procedure by Lapin. 
We do not claim that this definition is the best suited or that it depicts perfectly how these two structures of $A_{\infty}$-algebras and spectral sequences may interact, but the nice properties we have obtained 
make us believe it should deserve some attention for it may help understand how these structures are dealt with together.
 
We would like to remark that we also became interested in the problem for its possible applications to persistent homology, which originated from the work of H. Edelsbrunner, D. Letscher and A. Zamorodian in \cite{ELZ}. 
In persistent homology one studies particular structures coming from exact couples 
associated to a filtration of a complex (see \cite{BP}, and references therein). 
As noted by F. Belch\'{\i} and A. Murillo in \cite{BM}, one may gain some profit from the $A_{\infty}$-(co)algebra structure on cohomology. 
Their point of view is however completely different from ours (regardless of the fact they consider $A_{\infty}$-coalgebras, as pointed out in the last paragraph of this section), since they produce bars codes by means of the higher structure maps on the cohomology of the entire complex in a similar fashion to those produced in plain persistent cohomology theory. 
In our case, we try to provide the complete $A_{\infty}$-algebraic structure on the cohomology groups of each term of the corresponding exact couples. 
Taking into account that the $A_{\infty}$-algebra structures considered in Kadeishvili's theorem are only unique up (noncanonical) quasi-isomorphism, further work is necessary in order to apply our general results to actual computations in persistent homology, 
as far as we understand, for one should need in principle to build a metric (and then a probabilistic measure) on the space of these structures which is invariant under quasi-isomorphisms, 
in a similar manner to the bottleneck metric in the space of barcodes (see \cite{BGMP} and references therein). 

The article is organized as follows. 
In Section \ref{section:prel} we provide the basic theory about the two algebraic structures we will be dealing with: spectral sequences and $A_{\infty}$-algebras. 
In fact, in Subsection \ref{subsection:spe} we recall the main definitions and properties on the theory of spectral sequences and exact couples, which are well-known among the experts.  
Subsection \ref{subsection:kad} deals with the mentioned result by Kadeishvili, giving an $A_{\infty}$-algebra structure on the cohomology of an $A_{\infty}$-algebra structure such that both become 
quasi-isomorphic. 
We provide a complete proof for convenience of the reader. 
This section contains some definitions which are not completely standard, but which are more or less the canonical ones if one is intending to harmonize the bidegree in the definition of spectral sequences 
with the exact couples and the $A_{\infty}$-algebra structures. 

In Section \ref{section:defo} we recall the basic definitions and fact about the deformation theory of $A_{\infty}$-algebras. 
This has been studied by A. Fialowski and M. Penkava in \cite{FP}, and by E. Wu in his Master thesis \cite{Wu}. 
We adapt however the theory to our purposes of considering bigradings. 
The main result of that section, namely Theorem \ref{thm:quasiisofor}, is proved in detail. 
This was proved by completely different methods by Lapin in \cite{La02a}, Thm. 3.1 and Cor. 3.1, since ours uses obstruction theory, instead of homological perturbation theory. 
Our manner to proceed is more general for we do not need $k$ to be a field (see also Remark \ref{remark:dife}), and we take into account the bigrading. 
Section \ref{section:fildefo} includes how filtrations on $A_{\infty}$-algebras give rise to formal (bigraded) deformations in the previous sense, by adapting the well-known Rees algebra constructions, 
and we show that any formal (bigraded) deformations in fact appears as the Rees $A_{\infty}$-algebra of a filtered $A_{\infty}$-algebra. 

In the last section we recall the procedure studied by Lapin in \cite{La02a}, and we use it to define an $A_{\infty}$-enhancement of multiplicative spectral sequences.   
This construction is hard to handle, so we provide a much more manageable way (in our opinion) to deal with it. 
It involves the functor $\operatorname{D}$ we define there, and show that our procedure is equivalent to the one studied by Lapin. 
In any case, any of the definitions involve the notion of formal (bigraded) deformation of $A_{\infty}$-algebras, which is analyzed in the previous two sections. 
Moreover, we derive from our alternative construction several interesting results. 
In particular, we prove that one may associate an exact couple (or even a collection of them) to an $A_{\infty}$-enhancement of multiplicative spectral sequences, or more precisely, 
to a formal bigraded deformation of an $A_{\infty}$-algebra, and the multiplicative spectral sequence it determines is in fact isomorphic to the one coming 
from the formal bigraded deformation of an $A_{\infty}$-algebra (see Proposition \ref{proposition:compatres}).
We also show that the canonical multiplicative spectral sequence associated to a filtered $A_{\infty}$-algebra, which is the main example in the literature, has an $A_{\infty}$-enhancement 
(see Proposition \ref{proposition:final}). 
One interesting consequence of our detour was the corollary that any spectral sequence over a field comes from a filtration on a complex. 

Finally we would like to remark that, for $k$ a field, all our statements hold if one takes $A_{\infty}$-coalgebras and deformations of $A_{\infty}$-coalgebras instead 
(by following the lines of the definition of deformations of coalgebras in \cite{GS92}, p. 54. See also \cite{SS93}, Chapter 9, Section 1). 
The proof are somehow the duals of our proof, but the maps of the dual statements of Kadeishvili's theorem and of Theorem \ref{thm:quasiisofor} should always 
have as domain the cohomology, as in the algebraic case we present, in order to properly deal with cohomological obstructions. 
We leave these proofs to the interested reader but they are in our opinion \emph{mutatis mutandi}. 

\section{Preliminaries on basic algebraic structures}
\label{section:prel}

We follow the convention on gradings and definitions given in \cite{Herrec}, Sections $2$ and $5$. 
We recall that, as in the previously mentioned article, $k$ will denote a commutative ring with unit (which we also consider as a unitary graded ring concentrated in degree zero), 
the term \emph{module} (sometimes decorated by adjectives such as graded, or dg) will denote a symmetric bimodule over $k$ (correspondingly decorated), 
and morphisms between modules will always be $k$-linear (and satisfying further requirements if the modules are decorated as before), unless otherwise stated. 
The only difference with the previously mentioned article is that the Adams grading of the (differential) graded modules will be considered to be trivial (\textit{i.e.} all elements have zero Adams degree), 
so they are provided \textit{a priori} with only one grading, which is called \emph{cohomological}. 
The reason for such convention is just to simplify the exposition: since the Adams grading will play no role in this article (in open contrast to the case dealt with in \cite{Herrec}), we prefer avoiding it, 
even though it could be incorporated \textit{verbatim}. 
We also recall that, if $V = \oplus_{n \in \NN} V^{n}$ is a graded $k$-module, $V[1]$ is the graded module over $k$ whose $n$-th homogeneous component $V[1]^{n}$ is given by $V^{n+1}$, for all $n \in \ZZ$, and it is called the \emph{shift} of $V$. 

We shall use the following convention throughout this article. 
Given a collection of elements $\{ v_{i} : i \in I \}$ in a $k$-module, a sum of the form $\sum_{(i \in I)} v_{i}$ will always mean that the previous collection is finitely supported, 
\textit{i.e.} there exist a finite subset $I' \subseteq I$ such that $a_{i} = 0$, for all $i \in I \subseteq I'$. 
This notation is used in order to distinguish the previous situation from the usual convergent infinite sums on a topological vector space or module. 
As usual, all unadorned tensor products $\otimes$ would be over $k$. 

Finally, $\NN$ will denote the set of (strictly) positive integers, whereas $\NN_{0}$ 
will be the set of nonnegative integers. 
Similarly, for $N \in \NN$ (resp., $N \in \NN_{0}$), we denote by $\NN_{\leq N}$ (resp., $\NN_{0,\leq N}$) the set of positive (resp., nonnegative) integers less than or equal to $N$. 
Of course, similar notation could be used for other inequality signs.   

\subsection{\texorpdfstring{Spectral sequences and exact couples of dg algebras}{Spectral sequences and exact couples of dg algebras}} 
\label{subsection:spe}

In this section we briefly recall the definition of spectral sequences and exact couples, and we study the particular case in which they are provided with further algebraic structure. 
The notion of spectral sequence is due to J. Leray in his article \cite{Ler46}, and it was presented in more algebraic terms by J.-L. Koszul in \cite{Kos47} (see also \cite{Kos47b}).  
Though the study of the plain algebraic structure on spectral sequences already appeared in the literature from the very beginning in the articles of Leray and Koszul, it was J.-P. Serre in his article \cite{Se50I} 
(see also \cites{Se50II, Se50III, Se50a}) who introduced the generality we use today (see \cite{McC99}, Sections 3 to 5).  
Later on, W. S. Massey introduced the definition of exact couples (see \cite{Ma52} and \cite{Ma53}), as a more systematic manner to produce spectral sequences, 
and studied further in \cite{Ma54} the algebraic structure required on an exact couple in order to obtain the corresponding algebraic structure on the associated spectral sequence. 
A standard exposition on these subjects can be found in \cite{McC01}, Part I, or \cite{W}, Ch. 5, which we will partly follow.
We are interested in the strongly homotopy algebraic (sha) structure one may provide to spectral sequences and exact couples. 
Concerning a sha structure version on the former, it does not seem to have deserved much attention up to the present time, and as far as we know, the only near references we may provide are some articles by 
S. V. Lapin 
(\textit{cf.} \cites{La02a, La02b, La08}), where the author worked with a (in appearance) generalization of the notion of $A_{\infty}$-algebra (as remarked by V. Lyubaschenko in the AMS reference MR{1921810 (2003i:55023)}). 
In fact, the precise notion of ``$A_{\infty}$-enhancement'' of (multiplicative) spectral sequences is actually lacking in the previous articles, though several different occurrences of $A_{\infty}$-algebra structures appear on the pages of the spectral sequences considered in the different previously mentioned articles.  

A \emph{(cohomological) spectral sequence (of $k$-modules)} is given by the following data $(E^{p,q}_{r},d^{p,q}_{r}, i^{p,q}_{r})_{p,q \in \ZZ, r \in \NN}$, where $E^{p,q}_{r}$ is a $k$-module, 
$d_{r}^{p,q} : E^{p,q}_{r} \rightarrow E^{p+r,q-r+1}_{r}$ is a $k$-linear map satisfying that $d_{r}^{p+r,q-r+1} \circ d_{r}^{p,q} = 0$, for all $p,q \in \ZZ$ and $r \in \NN$, 
and $i^{p,q}_{r} : H^{p,q}_{r}(E) \rightarrow E^{p,q}_{r+1}$ is an isomorphism of $k$-modules, where $H^{p,q}_{r}(E)$ is defined as the quotient
\[     \mathrm{Ker}(d^{p,q}_{r})/\mathrm{Im}(d^{p-r,q+r-1}_{r}),     \]
for each $p, q \in \ZZ$ and $r \in \NN$. 
We may also allow that the indices $r$ belong to a set of the form $\NN_{\geq s}$, for some $s \in \NN$. 
We will say in this case that we have a \emph{(cohomological) spectral sequence starting at $s$}. 
Given $s, s' \in \NN$ with $s < s'$ any spectral sequence $(E^{p,q}_{r},d^{p,q}_{r}, i^{p,q}_{r})_{p,q \in \ZZ, r \in \NN_{\geq s}}$ starting at $s$ defines a spectral sequence 
$(E^{p,q}_{r},d^{p,q}_{r}, i^{p,q}_{r})_{p,q \in \ZZ, r \in \NN_{\geq s'}}$ starting at $s'$ 
by forgetting the terms of the spectral sequence indexed by $s, \dots, s'-1$, and it will be called the underlying spectral sequence starting at $s'$. 
Most of what we say below for usual spectral sequences clearly apply in the corresponding obvious form by restricting the indices in the appropriate manner to spectral sequences starting at some $s \in \NN$. 
Given a spectral sequence as before, we consider for each $r \in \NN$ the bigraded module 
\[     E_{r} = \bigoplus_{p,q \in \ZZ} E_{r}^{p,q},     \]
which is provided with a morphism $d_{r} : E_{r} \rightarrow E_{r}$ whose restriction to $E_{r}^{p,q}$ is $d^{p,q}_{r}$. 
We will sometimes regard $E_{r}$ as a graded module over $\ZZ$ with the \emph{total degree} defined by setting $E^{p,q}_{r}$ in total degree $p+q$, 
and in this case $d_{r}$ becomes a morphism of total degree $1$ satisfying that $d_{r} \circ d_{r} = 0$. 
Hence $(E_{r},d_{r})$ becomes cochain complex of $k$-modules and its cohomology $H(E_{r},d_{r})$ in fact coincides with $\oplus_{p, q \in \ZZ} H^{p,q}_{r}(E)$. 
On the other hand, the maps $i^{p,q}_{r}$ induce an isomorphism of bigraded modules $i_{r} : H(E_{r},d_{r}) \rightarrow E_{r+1}$.  
We will sometimes denote the spectral sequence only by $(E_{r},d_{r},i_{r})_{r \in \NN}$, or even by $(E_{r},d_{r})_{r \in \NN}$ to simplify our notation. 

For $(E_{r},d_{r},i_{r})_{r \in \NN}$ and $(E'_{r},d'_{r},i'_{r})_{r \in \NN}$ two spectral sequences, a \emph{morphism} from the former to the latter is a collection of morphisms 
$f_{r} : E_{r} \rightarrow E'_{r}$ of bigraded $k$-modules of bidegree $(0,0)$ for each $r \in \NN$, such that $f_{r} \circ d_{r} = d'_{r} \circ f_{r}$, and $f_{r+1} \circ i_{r} = i'_{r} \circ H(f_{r})$, 
for all $r \in \NN$, where $H(f_{r}) : H(E_{r},d_{r}) \rightarrow H(E'_{r},d'_{r})$ is the map induced by $f_{r}$ at the level of cohomology. 
Given $s, s' \in \NN$ with $s < s'$, two spectral sequences $(E^{p,q}_{r},d^{p,q}_{r}, i^{p,q}_{r})_{p,q \in \ZZ, r \in \NN_{\geq s}}$ and $(F^{p,q}_{r},\delta^{p,q}_{r}, j^{p,q}_{r})_{p,q \in \ZZ, r \in \NN_{\geq s'}}$ 
starting at $s$ and at $s'$, resp., are said to be \emph{compatible}, if the underlying spectral sequence starting at $s'$ of $(E^{p,q}_{r},d^{p,q}_{r}, i^{p,q}_{r})_{p,q \in \ZZ, r \in \NN_{\geq s}}$ 
is isomorphic to $(F^{p,q}_{r},\delta^{p,q}_{r}, j^{p,q}_{r})_{p,q \in \ZZ, r \in \NN_{\geq s'}}$.

Given a spectral sequence $(E_{r},d_{r},i_{r})_{r \in \NN}$, since each $E_{r+1}^{p,q}$ is a subquotient of $E^{p,q}_{r}$, for each $p, q \in \ZZ$ there exists a (unique) filtration of $k$-modules of $E^{p,q}_{1}$  
defined recursively
\begin{equation}
\label{eq:filt}
     0 = B^{p,q}_{1} \subseteq \dots \subseteq  B^{p,q}_{r} \subseteq B^{p,q}_{r+1} \subseteq \dots \subseteq Z^{p,q}_{r+1} \subseteq Z^{p,q}_{r} \subseteq \dots \subseteq Z^{p,q}_{1} = E^{p,q}_{1}     
\end{equation}
such that $i^{p,q}_{r}$ induces an isomorphism $j^{p, q}_{r} : Z^{p,q}_{r+1}/B^{p,q}_{r+1} \rightarrow E^{p,q}_{r+1}$ for all $r \in \NN$. 
Indeed, we set $j^{p, q}_{0}$ to be the identity morphism, and suppose moreover we have defined $B^{p,q}_{s}$ and $Z^{p,q}_{s}$ for $1 \leq s \leq r$ forming a nested chain as before together with a fixed isomorphism $j^{p, q}_{s} : Z^{p,q}_{s+1}/B^{p,q}_{s+1} \rightarrow E^{p,q}_{s+1}$ for $1 \leq s \leq r-1$. 
By taking inverse image under $j^{p,q}_{r-1}$ of the kernel of $d^{p,q}_{r}$ and the image of $d^{p-r,q+r-1}_{r}$ together with the previous isomorphism we get two (uniquely defined) submodules $Z^{p,q}_{r+1}$ and $B^{p,q}_{r+1}$ of $Z^{p,q}_{r}$ containing $B^{p,q}_{r}$, respectively. 
In particular, $j^{p,q}_{r-1}$ induces an isomorphism 
\[     Z^{p,q}_{r+1}/B^{p,q}_{r+1} \simeq \mathrm{Ker}(d^{p,q}_{r})/\mathrm{Im}(d^{p-r,q+r-1}_{r}),     \]
which together with $i^{p,q}_{r}$ yields the isomorphism $j^{p,q}_{r} : Z^{p,q}_{r+1}/B^{p,q}_{r+1} \rightarrow E^{p,q}_{r+1}$. 
We may form the bigraded submodules $Z_{r} = \oplus_{p, q \in \ZZ} Z^{p,q}_{r}$ and $B_{r} = \oplus_{p, q \in \ZZ} B^{p,q}_{r}$ of $E_{1}$ and the isomorphism $j_{r} : Z_{r+1}/B_{r+1} \rightarrow E_{r+1}$ whose components are the maps $j^{p,q}_{r}$. 
Consider the following $k$-modules
\[     Z^{p,q}_{\infty} = \bigcap_{r \in \NN} Z^{p,q}_{r}, \hskip 1cm B^{p,q}_{\infty} = \bigcup_{r \in \NN} B^{p,q}_{r},     \]
and the quotient $E^{p,q}_{\infty} = Z^{p,q}_{\infty}/B^{p,q}_{\infty}$. 
As before, this gives bigraded submodules $Z_{\infty}$ and $B_{\infty}$ of $E_{1}$ and $E_{\infty} = Z_{\infty}/B_{\infty}$. 

A \emph{filtered graded module} will be a graded $k$-module  $H = \oplus_{n \in \ZZ} H^{n}$ provided with a (decreasing) filtration $\{ F^{p}H \}_{p \in \ZZ}$ of graded modules, 
\textit{i.e.} $F^{p}H = \oplus_{n \in \ZZ} F^{p}H^{n}$ is a graded module whose $n$-th homogeneous component is $F^{p}H^{n}$, and $F^{p}H \supseteq F^{p+1}H$, for all $p \in \ZZ$. 
We say that the filtration of $H$ is \emph{exhaustive} if $\cup_{p \in \ZZ} F^{p}H = H$, 
and that it is \emph{Hausdorff} if $\cap_{p \in \ZZ} F^{p}H = \{ 0 \}$. 
Forgetting about the grading restriction on the members of the filtration, we see that the notion of a (Hausdorff) filtered module $H$ coincides with that of a (Hausdorff) separable topological $k$-module  
(\textit{i.e.} a $k$-module provided with a topology such that the sum $H \times H \rightarrow H$ and the multiplication maps $k \times H \rightarrow H$ are continuous), 
where $k$ is regarded as a topological ring with the discrete topology (see \cite{Bour2}, Chap. III, \S 2, n$^{0}$ 5). 
Indeed, the filtration is just a base of a system of neighbourhoods of zero. 

A filtered graded module $H$ is said to be \emph{complete} if the canonical map (in the category of $k$-modules)
\begin{equation}
\label{eq:compl}
     H \rightarrow \underset{\leftarrow q}{\lim} \hskip 0.6mm H/F^{q}H     
\end{equation}
is an isomorphism. 
We remark that the previous inverse limit is computed in the category of $k$-modules, not of graded ones. 
Moreover, the codomain object of the previous morphism is called the \emph{completion} of $H$ 
with respect to the filtration $\{ F^{p}H \}_{p \in \ZZ}$, and it is usually denoted by $\hat{H}$. 
We recall that the completion $\hat{H}$ may be provided with the canonical decreasing filtration $\{ F^{p}\hat{H} \}_{p \in \ZZ}$ of $k$-modules given by 
\begin{equation}
\label{eq:canfil}
     F^{p}\hat{H} = \underset{\leftarrow q}{\lim} \hskip 0.6mm F^{p}H/F^{q}H.     
\end{equation}
By regarding the filtration $\{ F^{p}H \}_{p \in \ZZ}$ as a family of neighbourhoods of zero for a linear topology on $H$, $\hat{H}$ is exactly the topological completion of $H$, so the name is justified.  
Note however that the previous completion $\hat{H}$ will not be in general a graded module over $k$. 
Since we will need to consider in the sequel the graded versions of the previous definitions, we give them quickly. 
A filtered graded module $H$ is said to be \emph{graded complete} if the canonical map (in the category of graded $k$-modules)
\begin{equation}
\label{eq:complgr}
     H \rightarrow \underset{\leftarrow q}{\lim}^{\mathrm{gr}} \hskip 0.6mm H/F^{q}H     
\end{equation}
is an isomorphism, where we stress that the latter inverse limit is computed in the category of graded $k$-modules. 
As before, the codomain object of the previous morphism is called the \emph{graded completion} of $H$ 
with respect to the filtration $\{ F^{p}H \}_{p \in \ZZ}$, and it is usually denoted by $\hat{H}^{\mathrm{gr}}$. 
Furthermore, the graded completion $\hat{H}^{\mathrm{gr}}$ can also be provided with a canonical decreasing filtration $\{ F^{p}\hat{H}^{\mathrm{gr}} \}_{p \in \ZZ}$ of graded $k$-modules given by 
\begin{equation}
\label{eq:canfilgr}
     F^{p}\hat{H}^{\mathrm{gr}} = \underset{\leftarrow q}{\lim}^{\mathrm{gr}} \hskip 0.6mm F^{p}H/F^{q}H.     
\end{equation}
Notice that the kernels of the canonical maps \eqref{eq:compl} and \eqref{eq:complgr} coincide and it is given by $\cap_{p \in \ZZ} F^{p}H$. 
Thus, if the filtration is either complete or graded complete, then it is Hausdorff. 
If $k$ is a field, then, as a graded $k$-module, the graded completion $\hat{H}^{\mathrm{gr}}$ of $H$ can be identified with the product in the category of graded $k$-modules given by
\[     \prod_{p \in \ZZ} \hskip -0.8mm {}^{\mathrm{gr}} \hskip 0.8mm F^{p}H/F^{p+1}H.     \]

The \emph{associated graded object} to a filtered graded $k$-module $H$ is given by the direct sum 
\[     \bigoplus_{p \in \ZZ} F^{p}H/F^{p+1}H     \]
of graded modules, and is denoted by $\mathrm{Gr}_{F^{\bullet}H}(H)$. 
We see that in fact is a bigraded object, where the homogeneous component of degree $(p,q) \in \ZZ$ is given by $F^{p}H^{p+q}/F^{p+1}H^{p+q}$. 
The total degree of this bigraded module coincides with the degree induced by that of $H$. 

We say that a spectral sequence $(E_{r},d_{r},i_{r})_{r \in \NN}$ \emph{weakly converges} to a filtered graded module $H$ if there exists an isomorphism of bigraded modules 
$\beta : E_{\infty} \rightarrow \mathrm{Gr}_{F^{\bullet}H}(H)$. 
In other words, for all $p, q \in \ZZ$ there exists an isomorphism $\beta^{p,q} : E^{p,q}_{\infty} \rightarrow F^{p}H^{p+q}/F^{p+1}H^{p+q}$ of $k$-modules. 
If furthermore the filtration of $H$ is exhaustive and Hausdorff, one says that 
the spectral sequence \emph{approaches $H$} (or \emph{abuts to $H$}). 
Finally, the spectral sequence \emph{converges to $H$} if it approaches it, it is regular and the filtration of $H$ is complete.     

As explained by Massey in \cite{Ma52} and \cite{Ma53}, a typical manner to produce a spectral sequence is from an \emph{exact couple}. 
Even though exact couple need not lie in the category of bigraded objects, for they can be considered in more general contexts, we shall only recall the definition in the situation 
which fits our specific setting of spectral sequences. 
Moreover, we shall allow more general bigradings than those considered in the standard references of the literature \cite{McC01}, Chapter 2, Section 2, pp. 37--42, or \cite{W}, Chapter 5, Section 9. 
Given $r, s \in \NN_{0}$, an \emph{exact couple of bigraded $k$-modules of $(r,s)$-th type} is given by the data $(E,D,\mathrm{i},\mathrm{j},\mathrm{k})$ where $E = \oplus_{p,q \in \ZZ} E^{p,q}$ 
and $D =  \oplus_{p,q \in \ZZ} D^{p,q}$ are bigraded modules over $k$, $\mathrm{i} : D \rightarrow D$ is a homogeneous morphism of bidegree $(-1,1)$, $\mathrm{j} : D \rightarrow E$ is a morphism of bidegree 
$(r,-r)$ and $\mathrm{k} : E \rightarrow D$ has bidegree $(1+s,-s)$ such that $\mathrm{Im}(\mathrm{k}) = \Ker(\mathrm{i})$, $\mathrm{Im}(\mathrm{j}) = \Ker(\mathrm{k})$ and $\mathrm{Im}(\mathrm{i}) = \Ker(\mathrm{j})$. 
An exact couple of bigraded $k$-modules of $(r,0)$-th type will be called an \emph{exact couple of bigraded $k$-modules of $r$-th type}. 
These are the ``usual'' exact couples of the literature. 
An exact couple is usually depicted as a triangle 
\[
\xymatrix
{
D 
\ar[rr]^{\mathrm{i}}
&
&
D
\ar[dl]^{\mathrm{j}}
\\
&
E
\ar[ul]^{\mathrm{k}}
&
}
\]
which is exact at each vertex. 
Note that the morphism $d = \mathrm{j} \circ \mathrm{k}$ is a differential, \textit{i.e.} it satisfies that $d \circ d = 0$. 
If $(E,D,\mathrm{i},\mathrm{j},\mathrm{k})$ is an exact couple of $(r,s)$-th type, one may produce an exact couple $(E',D',\mathrm{i}',\mathrm{j}',\mathrm{k}')$ of $(r+1,s)$-th type, 
which is called the \emph{derived exact couple}, as follows. 
Set $E' = \Ker(d)/\mathrm{Im}(d)$, $D' = \mathrm{Im}(\mathrm{i})$, $\mathrm{i}'$ to be the restriction of $\mathrm{i}$ to $D'$, $\mathrm{j}'$ given by the formula $\mathrm{j}'(i(x)) = [j(x)]$, for $x \in D$ and where the brackets denote the cohomology class, and $\mathrm{k}'$ to be the canonical maps induced by $\mathrm{k}$. 
Inductively, for $n \in \NN$, take $(E^{(n)},D^{(n)},\mathrm{i}^{(n)},\mathrm{j}^{(n)},\mathrm{k}^{(n)})$ to be the derived couple of $(E^{(n-1)},D^{(n-1)},\mathrm{i}^{(n-1)},\mathrm{j}^{(n-1)},\mathrm{k}^{(n-1)})$, and 
$(E^{(0)},D^{(0)},\mathrm{i}^{(0)},\mathrm{j}^{(0)},\mathrm{k}^{(0)}) = (E,D,\mathrm{i},\mathrm{j},\mathrm{k})$.  
Now, given an exact couple of $(r,s)$-th type $(E,D,\mathrm{i},\mathrm{j},\mathrm{k})$ it defines an spectral sequence $(E_{r'},d_{r'}, i_{r'})_{r' \in \NN_{> (r+s)}}$ starting at $(r+s+1)$ 
by setting $E_{r'} = E^{(r'-r-s-1)}$, $d_{r'} = \mathrm{j}^{(r'-r-s-1)} \circ \mathrm{k}^{(r'-r-s-1)}$, and $i_{r'}$ is the identity map. 
It will be called the \emph{spectral sequence (starting at $(r+s+1)$) associated to the exact couple (of $(r,s)$-th type)}. 
One of the main examples of exact couples are obtained from filtrations of dg modules over $k$ as we now recall. 
Let $(M,d_{M})$ be a dg module over $k$ provided with a decreasing filtration $\{ F^{p}M \}_{p \in \ZZ}$ of the underlying graded module of $M$ such that $d_{M}(F^{p} M) \subseteq F^{p}M$, for all $p \in \ZZ$. 
For $s \in \NN_{0}$, we define the dg $k$-modules 
\begin{equation}
\label{eq:reescomplex}
     {}^{s}\mathcal{D} = \bigoplus_{p \in \ZZ} \Big(\big(F^{p}M \cap d_{M}^{-1}(F^{p+s}M)\big) \otimes k.\hbar^{-p}\Big)  \subseteq M \otimes k[\hbar^{\pm 1}],     
\end{equation}
provided with the differential given by (the restriction of) $d_{M} \otimes \mathrm{id}_{k[\hbar^{\pm 1}]} \hbar^{-s}$, and 
\begin{equation}
\label{eq:gr}  
   {}^{s}\mathcal{E} = \bigoplus_{p \in \ZZ} \Big( \frac{F^{p}M \cap d_{M}^{-1}(F^{p+s}M)}{F^{p+1}M \cap d_{M}^{-1}(F^{p+1+s}M)} \Big),     
\end{equation} 
with the induced differential by $d_{M}$. 
It sends the class $z + (F^{p+1}M \cap d_{M}^{-1}(F^{p+1+s}M))$ to $d_{M}(z) + (F^{p+s+1}M \cap d_{M}^{-1}(F^{p+2 s+1}M))$, 
for $z \in F^{p}M \cap d_{M}^{-1}(F^{p+s}M)$. 
Moreover, these dg modules are bigraded by setting 
\[     {}^{s}\mathcal{D}^{p,q} = (F^{p}M^{p+q} \cap d_{M}^{-1}(F^{p+s}M^{p+q+1})) \otimes k.\hbar^{-p}     \] 
and 
\[     {}^{s}\mathcal{E}^{p,q} = (F^{p}M^{p+q} \cap d_{M}^{-1}(F^{p+s}M^{p+q+1}))/(F^{p+1}M^{p+q} \cap d_{M}^{-1}(F^{p+1+s}M^{p+q+1})).     \] 
Consider now the short exact sequence of complexes of modules over $k$ of the form 
\begin{equation}
\label{eq:seqfil}
     0 \rightarrow {}^{s}\mathcal{D} \overset{{}^{s}\tilde{i}}{\rightarrow} {}^{s}\mathcal{D} \overset{{}^{s}\tilde{j}}{\rightarrow} {}^{s}\mathcal{E} \rightarrow 0,     
\end{equation}
where ${}^{s}\tilde{i}$ is the morphism given by multiplication by $\hbar$, and ${}^{s}\tilde{j}$ is the canonical projection. 
Note that ${}^{s}\tilde{i}$ has bidegree $(-1,1)$ and ${}^{s}\tilde{j}$ has bidegree $(0,0)$
By considering the long exact sequence of cohomology groups of the previous short exact sequence, which may be rearranged as a triangle 
\[
\xymatrix
{
H^{\bullet}({}^{s}\mathcal{D}) 
\ar[rr]^{H^{\bullet}({}^{s}\tilde{i})}
&
&
H^{\bullet}({}^{s}\mathcal{D})
\ar[dl]^{H^{\bullet}({}^{s}\tilde{j})}
\\
&
H^{\bullet}({}^{s}\mathcal{E})
\ar[ul]^{{}^{s}\delta^{\bullet}}
&
}
\]
we get an exact couple of $(0,s)$-th type, where we recall that ${}^{s}\delta^{\bullet}$ is the delta morphism obtained by means of the Snake lemma (see \cite{W}, Thm. 1.3.1, Example 5.9.3). 
It will be called the \emph{exact couple of $(0,s)$-th type associated to the filtration of the complex}, and the corresponding spectral sequence starting at $(s+1)$ will be called 
the \emph{spectral sequence starting at $(s+1)$ associated to the filtration of the complex}. 
By a standard construction on spectral sequences coming from filtrations, one sees that for any $s , s' \in \NN_{0}$, the previously constructed spectral sequence starting at $(s+1)$ 
is compatible with the one starting at $(s'+1)$ (see \cite{W}, Chapter 5, Section 4). 

From the previously observed property, let us consider a collection of short exact sequences 
\[     0 \rightarrow {}^{s}A \overset{{}^{s}i}{\rightarrow} {}^{s}B \overset{{}^{s}p}{\rightarrow} {}^{s}C \rightarrow 0     \] 
of complexes of $k$-modules indexed by $s \in \NN_{0}$, where each complex ${}^{s}A$, ${}^{s}B$ and ${}^{s}C$ is in fact the total complex of a bigraded $k$-module, 
with differentials of bidegree $(s,1-s)$, and where ${}^{s}i$ and ${}^{s}p$ are homogeneous morphism of bidegree $(-1,1)$ and $(0,0)$, respectively. 
For each $s \in \NN_{0}$, the $s$-th complex defines an exact couple of $(0,s)$ type by following the procedure recalled in the previous paragraph by taking cohomology. 
We say that such a collection of complexes defines a \emph{compatible spectral sequence}, if the members of the family of spectral sequences starting at $(s+1)$, associated 
to the exact couples referred before, are compatible with each other. 
The main example of collection of complexes defining a compatible spectral sequence is the given by the short exact sequences \eqref{eq:seqfil}, coming from a filtration of a complex. 

As we stated before the following definitions are standard (\textit{cf.} \cite{McC01}, Section 2.3, or \cite{W}, 5.4.8). 
A \emph{multiplicative (cohomological) spectral sequence} (sometimes called a \emph{spectral sequence of algebras}) is a (cohomological) spectral sequence 
$(E_{r},d_{r}, i_{r})_{r \in \NN}$, provided with a collection of maps $(\mu_{r})_{r \in \NN}$ where  
$\mu_{r} : E_{r} \otimes E_{r} \rightarrow E_{r}$ is a bigraded associative product on $E_{r}$ (\textit{i.e.} $\mu_{r}$ is associative and 
$\mu_{r}(E^{p,q}_{r} \otimes E^{p',q'}_{r}) \subseteq E^{p+p',q+q'}_{r}$, for all $p, p', q, q' \in \ZZ$), making $(E_{r},d_{r})$ into a dg algebra (for the grading given by the total degree), 
\textit{i.e.} such that $d_{r} \circ \mu_{r} = \mu_{r} \circ (\mathrm{id}_{E_{r}} \otimes d_{r} + d_{r} \otimes \mathrm{id}_{E_{r}})$, where we use the Koszul sign rule for the total degree, 
and $i_{r}$ is an isomorphism of graded algebras, \textit{i.e.} $\mu_{r+1}$ is given by the composition 
\begin{multline*}
     E_{r+1} \otimes E_{r+1} \overset{i_{r}^{-1} \otimes i_{r}^{-1}}{\longrightarrow} H(E_{r},d_{r}) \otimes H(E_{r},d_{r}) \longrightarrow 
     \\ 
     H(E_{r} \otimes E_{r}, \mathrm{id}_{E_{r}} \otimes d_{r} + d_{r} \otimes \mathrm{id}_{E_{r}}) \overset{H(\mu_{r})}{\longrightarrow} H(E_{r},d_{r}) \overset{i_{r}}{\longrightarrow} E_{r+1}     
\end{multline*}
where the unnamed middle map is the canonical morphism given by $[e] \otimes [e'] \rightarrow [e \otimes e']$, where $e, e' \in \mathrm{Ker}(d_{r})$ and the brackets denote the corresponding cohomology classes. 
It is easy to see that in this situation the cocycle submodules $Z_{r}$ appearing in the sequence \eqref{eq:filt} are in fact bigraded subalgebras of $E_{1}$, each $B_{r}$ is an ideal of $Z_{r}$, and the map 
$j_{r} : Z_{r+1}/B_{r+1} \rightarrow E_{r+1}$ is an isomorphism of bigraded algebras. 
As a direct consequence, we see that $Z_{\infty}$ is a bigraded subalgebra of $E_{1}$ with ideal $B_{\infty}$, so $E_{\infty}$ has a canonical structure of bigraded algebra. 
For convenience, unless explicitly stated it will be usually considered to be bigraded by the indices $p$ and $q$, or as graded module by the total degree $p+q$. 
A \emph{morphism} between two multiplicative spectral sequences is a morphism$\{ f_{r} \}_{r \in \NN}$ between the underlying spectral sequences such that $f_{r}$ is also a morphism of algebras. 
The notion of multiplicative spectral sequences starting at $s$, for $s \in \NN$, together with their morphisms, and the compatibility of collections of multiplicative spectral sequences, 
are defined analogously as in the case of plain spectral sequences.

Suppose we are given a \emph{filtered graded algebra} $H$, \textit{i.e.} $H$ is a filtered graded module with a graded algebra structure such that $F^{p}H \cdot F^{q}H \subseteq F^{p+q}H$, for all $p, q \in \ZZ$ 
(and $1_{H} \in F^{0}H$ if $H$ is unitary). 
We see that $\mathrm{Gr}_{F^{\bullet}H}(H)$ is a bigraded algebra. 
In this case, we say that the multiplicative spectral sequence \emph{weakly converges} to $H$ if there exists an isomorphism of bigraded algebras $\beta : E_{\infty} \rightarrow \mathrm{Gr}_{F^{\bullet}H}(H)$. 
The definitions when the multiplicative spectral sequence \emph{approaches} or \emph{converges} to a filtered graded algebra are as previously, 
but always considering that the map $\beta$ is an isomorphism of bigraded algebras. 

A typical manner to produce a multiplicative spectral sequence is when we consider a \emph{filtered dg algebra}. 
We recall that filtered dg algebras is a dg algebra $(A,d_{A})$ provided with a (decreasing) filtration $\{ F^{p}A \}_{p \in \ZZ}$ 
of the underlying graded module of $A$ satisfying the compatibility conditions 
\begin{equation}
\label{eq:compfiltdg}
     d_{A}(F^{p}A) \subseteq F^{p}A, \hskip 1cm \textit{ and } \hskip 1cm F^{q}A \cdot F^{q'}A \subseteq F^{q+q'}A,     
\end{equation}
for all $p, q, q' \in \ZZ$. 
If $A$ has a unit we further assume that $1_{A} \in F^{0}A$. 
In this case, it is easy to verify that the spectral sequence associated to the filtration of the underlying complex of $A$ is indeed multiplicative. 
Furthermore, the collection of multiplicative spectral sequences starting at $(s+1)$, for $s \in \NN_{0}$, given by the short exact sequences \eqref{eq:seqfil}, coming from a filtration of a dg algebra is compatible 
(see \cite{W}, Chapter 5, Section 4). 
The theory of multiplicative exact couples was developed in \cite{Ma54} but it does not seem to this author to be completely straightforward (paraphrasing what D. Benson states in his book \cite{Ben}, Ch. 3, Section 9) and in some sense satisfactory as one may hope. 

\subsection{\texorpdfstring{$A_{\infty}$-algebras and Kadeishvili's theorem}{A-infinity-algebras and Kadeishvili's theorem}} 
\label{subsection:kad}

The notion of $A_{\infty}$-algebra was introduced by J. Stasheff in \cite{Sta} in his study of homotopy theory of loop spaces. 
We refer the reader to \cite{Prou}, Chapitre 3, or also \cite{LH}, Chapitre 1, for standard references. 
As stated previously, our particular sign and grading conventions are explained in detail in \cite{Herrec}, Sections 2 and 5. 
Though they do not coincide with the aforementioned references, they agree with several others in the literature (see \cites{LPWZ04, LPWZ09} and references therein). 

An \emph{$A_{\infty}$-algebra} structure on a cohomological graded $k$-module $A$ is a collection of maps 
$m_{i} : A^{\otimes i} \rightarrow A$ for $i \in \NN$ of cohomological degree $2-i$ satisfying the \emph{Stasheff identities} $\mathrm{SI}(n)$ given by   
\begin{equation}
\label{eq:ainftyalgebra}
   \sum_{(r,s,t) \in \mathcal{I}_{n}} (-1)^{r + s t}  m_{r + 1 + t} \circ (\mathrm{id}_{A}^{\otimes r} \otimes m_{s} \otimes \mathrm{id}_{A}^{\otimes t}) = 0,
\end{equation} 
for $n \in \NN$, where $\mathcal{I}_{n} = \{ (r,s,t) \in \NN_{0} \times \NN \times \NN_{0} : r + s + t = n \}$. 
Given $N \in \NN$, if $A$ is provided only with the morphisms $m_{i}$ for $i \in \NN_{\leq N}$ and satisfy the Stasheff identities $\mathrm{SI}(n)$ for $n \in \NN_{\leq N}$, we say that it is an \emph{$A_{N}$-algebra}.
Note that the first Stasheff identity $\mathrm{SI}(1)$ means that $m_{1}$ is a differential of $A$, 
so we may consider the cohomology $k$-module $H^{\bullet}(A)$ given by the quotient $\mathrm{Ker}(m_{1})/\mathrm{Im}(m_{1})$. 
Following \cite{Lun}, we will say that $A$ is \emph{flat} if $H^{\bullet}(A)$ is a (graded) projective $k$-module. 

Let $N \in \NN \cup \{ \infty \}$, and let $A$ be an $A_{N}$-algebra. 
If $N = \infty$, we define $\NN_{\leq \infty} = \NN$. 
There exists a (not necessarily counitary) dg coalgebra $B_{N}(A)$, called the \emph{bar construction of $A$}. 
If $N = \infty$, it is usually denoted just by $B(A)$.  
Its underlying graded coalgebra is given by the truncated tensor coalgebra $\oplus_{i \in \NN_{\leq N}} A[1]^{\otimes i}$, where $A[1]$ denotes the shift of $A$. 
As usual, if $n \in \NN_{\leq N}$ we will typically denote an element $s_{A}(a_{1}) \otimes \dots \otimes s_{A}(a_{n}) \in A[1]^{\otimes n}$ in the form $[a_{1} | \dots | a_{n}]$, 
where $a_{1}, \dots, a_{n} \in A$, and $s_{A} : A \rightarrow A[1]$ is the canonical morphism of degree $- 1$ whose underlying map of $k$-modules is the identity.  
The coproduct is given by the usual deconcatenation 
\[     \Delta ([a_{1}|\dots |a_{n}]) = \sum_{i=1}^{n-1} [a_{1}|\dots | a_{i}] \otimes [a_{i+1} | \dots | a_{n}].     \]
Since $B_{N}(A)$ is a truncated tensor graded coalgebra, its differential $B_{N}$ can be defined as follows. 
It is the unique coderivation determined by $\pi_{1} \circ B_{N}$, where $\pi_{1} : B_{N}(A) \rightarrow A[1]$ is the canonical projection 
(\textit{cf.} \cite{LH}, Lemme 1.1.2.2, and see \cite{LH}, Section 1.2.2, pp. 29--30), such that this composition map is given by the sum $b = \sum_{i \in \NN_{\leq N}} b_{i}$, 
where $b_{i} : A[1]^{\otimes i} \rightarrow A[1]$ is defined as $b_{i} = - s_{A} \circ m_{i} \circ (s_{A}^{\otimes i})^{-1}$. 
In fact, equation \eqref{eq:ainftyalgebra} is precisely the condition for this coderivation to be a differential (\textit{cf.} \cite{LH}, Lemme 1.2.2.1, and see \cite{LH}, Section 1.2.2, pp. 29--30).  

An $A_{\infty}$-algebra is called \emph{(strictly) unitary} if there is a map $\eta_{A} : k \rightarrow A$ of complete degree zero such that $m_{i} \circ (\mathrm{id}_{A}^{\otimes r} \otimes \eta_{A} \otimes \mathrm{id}_{A}^{\otimes t})$
vanishes for all $i \neq 2$ and all $r, t \geq 0$ such that $r+1+t = i$. 
We shall usually denote the image of $1_{k}$ under $\eta_{A}$ by $1_{A}$, and call it the \emph{(strict) unit} of $A$. 
Clearly, the flatness property may also be stated for strictly unitary $A_{\infty}$-algebras.  
We say that a unitary or nonunitary $A_{\infty}$-algebra is called \emph{minimal} if $m_{1}$ vanishes. 
Note that all the previous definitions can be applied as well to $A_{N}$-algebras, for $N \in \NN$. 
We see that a (unitary) dg algebra $(A,d_{A},\mu_{A})$ is a particular case of (unitary) $A_{\infty}$-algebra, where $m_{1} = d_{A}$ is the differential and $m_{2} = \mu_{A}$ is the product. 

Given $s \in \NN_{0}$, we say that an $A_{\infty}$-algebra $A$ has a \emph{compatible bigrading of $s$-th type} if $A$ is provided with a bigrading $A = \oplus_{p, q  \in \ZZ} A^{p,q}$ such that 
$A^{n} = \oplus_{p \in \ZZ} A^{p,n-p}$, and $m_{n}$ is a homogeneous morphism of bidegree $(2-n) (s,-s+1)$, for all $n \in \NN$. 
If $A$ is unitary we further assume that $1_{A} \in A^{0,0}$. 
Note that if $A$ has a compatible bigrading (of $s$-th type), then the cohomology $k$-module $H^{\bullet}(A)$ is provided with a canonically induced bigrading. 

A \emph{morphism of $A_{\infty}$-algebras} $f_{\bullet} : A \rightarrow B$ between two $A_{\infty}$-algebras $A$ and $B$ is a collection of morphisms 
of the underlying graded $k$-modules $f_{i} : A^{\otimes n} \rightarrow B$ of cohomological degree $1-i$ for $i \in \NN$ 
satisfying the \emph{Stasheff identities on morphisms} $\mathrm{MI}(n)$ given by  
\begin{equation}
\label{eq:ainftyalgebramor}
   \sum_{(r,s,t) \in \mathcal{I}_{n}} (-1)^{r + s t}  f_{r + 1 + t} \circ (\mathrm{id}_{A}^{\otimes r} \otimes m_{s}^{A} \otimes \mathrm{id}_{A}^{\otimes t}) 
   = \sum_{q \in \NN} \sum_{\bar{i} \in \NN^{q, n}} (-1)^{w} m_{q}^{B} \circ (f_{i_{1}} \otimes \dots \otimes f_{i_{q}}),
\end{equation} 
for $n \in \NN$, where $w = \sum_{j=1}^{q} (q-j) (i_{j} - 1)$ and $\NN^{q,n}$ is the subset of $\NN^{q}$ of elements $\bar{i} = (i_{1},\dots,i_{q})$ such that $|\bar{i}| = i_{1} + \dots + i_{q} = n$. 
Given $N \in \NN$, if $A$ and $B$ are only $A_{N}$-algebras, a \emph{morphism of $A_{N}$-algebras} $f_{\bullet} : A \rightarrow B$ is a collection of morphisms 
of the underlying graded $k$-modules $f_{i} : A^{\otimes n} \rightarrow B$ of cohomological degree $1-i$ for $i \in \NN_{\leq N}$ satisfying the previous identities $\mathrm{MI}(n)$ for $n \in \NN_{\leq N}$.  
If $A$ and $B$ are unitary $A_{\infty}$-algebras, the morphism $f_{\bullet}$ is called \emph{(strictly) unitary} if $f_{1}(1_{A}) = 1_{B}$, and for all $i \geq 2$ we have that $f_{i}(a_{1}, \dots, a_{i})$ vanishes 
if there exists $j \in \{1, \dots, i \}$ such that $a_{j} = 1_{A}$. 
Notice that $f_{1}$ is a morphism of dg $k$-modules for the underlying structures on $A$ and $B$. 
We say that a morphism of (resp., unitary) $A_{\infty}$-algebras $f_{\bullet} : A \rightarrow B$ is a \emph{quasi-isomorphism} if $f_{1}$ is a quasi-isomorphism of the underlying complexes. 
We say that a morphism $f_{\bullet}$ is \emph{strict} if $f_{i}$ vanishes for $i \geq 2$. 
Note that all these definitions also apply to morphisms of $A_{N}$-algebras. 

Given two (resp., unitary) $A_{\infty}$-algebras $A$ and $B$ provided with compatibles bigradings of $s$-th type, for some $s \in \NN_{0}$, and a morphism $f_{\bullet} : A \rightarrow B$ of (resp., unitary) $A_{\infty}$-algebras, 
we say that $f$ is \emph{compatible with the bigradings} if $f_{n}$ is a homogeneous morphism of bidegree $(1-n) (s,-s+1)$, for all $n \in \NN$.  

Let $N \in \NN \cup \{ \infty \}$, and let $A$ and $A'$ be two $A_{N}$-algebras. 
Given $f_{\bullet} : A \rightarrow A'$ a morphism of $A_{N}$-algebras, it induces a morphism of dg coalgebras $B_{N}(f_{\bullet}) : B_{N}(A) \rightarrow B_{N}(A')$ 
between the bar constructions as follows. 
Taking into account that $B_{N}(A')$ is a truncated tensor coalgebra, such a morphism of graded coalgebras $F_{N}$ is completely determined by the composition 
$\pi'_{1} \circ B_{N}(f_{\bullet}) : B_{N}(A) \rightarrow A'[1]$, 
where $\pi'_{1} : B_{N}(A') \rightarrow A'[1]$ denotes the canonical projection. 
The latter composition is given by a sum $\sum_{i \in \NN_{\leq N}} F_{N,i}$, where $F_{N,i} : A[1]^{\otimes i} \rightarrow A'[1]$, which we define to be $F_{N,i} = s_{A'} \circ f_{i} \circ (s_{A}^{\otimes i})^{-1}$, for 
$i \in \NN_{\leq N}$. 
In fact, \eqref{eq:ainftyalgebramor} is precisely the condition for this morphism to commute with the differentials (see \cite{LH}, Section 1.2.2). 

We state the following result which should be well-known among the experts, though it does not seem to be very present in the available literature. 
It is a key piece in the proof of Kadeishvili's theorem, which we later recall (in the articles \cites{K80, K82} by Kadeishvili, the author just states it is a ``direct calculation''). 
We give the proof for completeness and for it will be our guide for the forthcoming proof in the context of deformation theory. 
\begin{lemma}
\label{lemma:ind}
Let $N \in \NN$, $(A,\{ m_{i} \}_{i \in \NN_{\leq N+1}})$ be an $A_{N+1}$-algebra and $(A',\{ m'_{i} \}_{i \in \NN_{\leq N}})$ an $A_{N}$-algebra satisfying that $m'_{1} = 0$. 
Suppose we have a morphism $f_{\bullet} : A' \rightarrow A$ of $A_{N}$-algebras from $A'$ to the underlying $A_{N}$-algebra of $A$ such that $f_{1}$ induces an injective map between the cohomology groups. 
Define the map from $(A')^{\otimes (N+1)}$ to $A$ given by 
\begin{equation*}
\begin{split}
   U_{N+1} = &\sum_{(r,s,t) \in \mathcal{I}_{N+1}^{*}} (-1)^{r + s t}  f_{r + 1 + t} \circ (\mathrm{id}_{A'}^{\otimes r} \otimes m'_{s} \otimes \mathrm{id}_{A'}^{\otimes t}) 
   \\
   &- \sum_{q \in \NN_{\geq 2}} \sum_{\bar{i} \in \NN^{q, N+1}} (-1)^{w} m_{q} \circ (f_{i_{1}} \otimes \dots \otimes f_{i_{q}}),
\end{split}
\end{equation*} 
where $\mathcal{I}_{m}^{*} = \{ (r,s,t) \in \NN_{0} \times \NN_{\geq 2} \times \NN_{0} : r + s + t = m,  0 < r + t <  m - 1 \}$, $w = \sum_{j=1}^{q} (q-j) (i_{j} - 1)$ and $\NN^{q,m}$ is the subset of $\NN^{q}$ of elements $\bar{i} = (i_{1},\dots,i_{q})$ such that $|\bar{i}| = i_{1} + \dots + i_{q} = m$. 
Then $m_{1} \circ U_{N+1}$ vanishes. 
\end{lemma}
\begin{proof}
Choose any homogeneous maps $m_{N+1}' : (A')^{\otimes (N+1)} \rightarrow A'$ of degree $1- N$ and $f_{N+1} : (A')^{\otimes (N+1)} \rightarrow A$ of degree $-N$. 
By following the recipe of the bar construction, the maps $\{ m'_{i} \}_{i \in \NN_{\leq N+1}}$ induce a coderivation $B'_{N+1}$ on the tensor coalgebra $B_{N+1}(A')$. 
As $A'$ is not a $A_{N+1}$-algebra this is in fact an abuse of notation, for we should better write something like $\bar{T}^{c, \leq (N+1)}(A'[1])$, but we will allow our notation for the sake of simplicity, 
and because it is completely clear that we are not assuming that $B'_{N+1} \circ B'_{N+1}$ vanishes. 
However, since $A'$ is an $A_{N}$-algebra, $\pi'_{1} \circ B_{N+1} \circ B_{N+1} \circ \iota'_{p}$ vanishes for $p = 1, \dots, N$, 
where $\iota'_{p} : (A'[1])^{\otimes p} \rightarrow B_{N+1}(A')$ denotes the canonical inclusion, and $\pi'_{1} : B_{N+1}(A') \rightarrow A'[1]$ is the canonical projection. 
Accordingly, $\iota_{p} : (A[1])^{\otimes p} \rightarrow B_{N+1}(A)$ denotes the canonical inclusion, and $\pi_{1} : B_{N+1}(A) \rightarrow A[1]$ is the canonical projection. 
Moreover, by the previous explanations on the cobar construction of morphisms of $A_{N}$-algebras, the maps $\{ f_{i} \}_{i \in \NN_{\leq N+1}}$ induce a morphism of graded coalgebras $F_{N+1}$ 
from $B_{N+1}(A')$ to $B_{N+1}(A)$. 
Define $\tilde{U}_{N+1}$ given by $s_{A} \circ U_{N+1} \circ (s_{A'}^{\otimes (N+1)})^{-1}$. 
It is easy to show that 
\begin{equation*}
\begin{split}
   \tilde{U}_{N+1} = &\sum_{(r,s,t) \in \mathcal{I}_{N+1}^{*}} F_{N,r + 1 + t} \circ (\mathrm{id}_{A'}^{\otimes r} \otimes b'_{s} \otimes \mathrm{id}_{A'}^{\otimes t}) 
   \\
   &- \sum_{q \in \NN_{\geq 2}} \sum_{\bar{i} \in \NN^{q, N+1}} b_{q} \circ (F_{N,i_{1}} \otimes \dots \otimes F_{N,i_{q}}).
\end{split}
\end{equation*} 
It is clear that to prove that $m_{1} \circ U_{N+1}$ vanishes is tantamount to the vanishing of $b_{1} \circ \tilde{U}_{N+1}$. 
Furthermore, it is direct to see that 
\[     \tilde{U}_{N+1} = \pi_{1} \circ (F_{N+1} \circ B'_{N+1} - B_{N+1} \circ F_{N+1}) \circ \iota'_{N+1} - F_{N+1,1} \circ b'_{N+1} + b_{1} \circ F_{N+1,N+1}.     \]
Then 
\begin{align*}     
b_{1} \circ \tilde{U}_{N+1} &= b_{1} \circ \pi_{1} \circ (F_{N+1} \circ B'_{N+1} - B_{N+1} \circ F_{N+1}) \circ \iota'_{N+1}     
\\
&= B_{N+1} \circ (F_{N+1} \circ B'_{N+1} - B_{N+1} \circ F_{N+1}) \circ \iota'_{N+1}
\\
&= B_{N+1} \circ F_{N+1} \circ B'_{N+1}  \circ \iota'_{N+1},
\end{align*}
where we have used in the first equality that $b_{1} \circ b_{1} = 0$ and $b_{1} \circ F_{N+1,1} = 0$, for $m_{1} \circ m_{1} = 0$ and $m_{1} \circ f_{1} = 0$, 
and in the second equality that the image of $(F_{N+1} \circ B'_{N+1} - B_{N+1} \circ F_{N+1}) \circ \iota'_{N+1}$ is included in the image of $\iota_{1}$, 
by the hypothesis that $\{ f_{i} \}_{i \in \NN_{\leq N}}$ is a morphism of $A_{N}$-algebras. 
In the last equality we have used that $B_{N+1} \circ B_{N+1} = 0$, because $A$ is an $A_{N+1}$-algebra. 
As the image of $B'_{N+1}  \circ \iota'_{N+1}$ is included in $B_{N}(A')$, because $b'_{1}$ vanishes, $F_{N+1}|_{B_{N}(A')} = F_{N}$ and 
$F_{N} \circ B'_{N} = B_{N} \circ F_{N}$, we see that $b_{1} \circ \tilde{U}_{N+1}$ coincides with 
\[      F_{N} \circ B'_{N} \circ B'_{N+1}  \circ \iota'_{N+1} = F_{N+1} \circ B'_{N+1} \circ B'_{N+1}  \circ \iota'_{N+1}.     \]
Note that $B'_{N+1} \circ B'_{N+1}  \circ \iota'_{N+1} = \pi_{1} \circ B'_{N+1} \circ B'_{N+1}  \circ \iota'_{N+1}$, because $A'$ is an $A_{N}$-algebra (see \cite{LH}, Lemme B.1.1, a), 
and thus $b_{1} \circ \tilde{U}_{N+1}$ coincides with 
\[      F_{N+1,1} \circ (\pi'_{1} \circ B'_{N+1} \circ B'_{N+1}  \circ \iota'_{N+1}).     \]
Choose any $\omega \in (A'[1])^{\otimes (N+1)}$, and consider $a' = (\pi'_{1} \circ B'_{N+1} \circ B'_{N+1}  \circ \iota'_{N+1}) (\omega)$. 
It is a cocycle in $A'[1]$, for $b'_{1}$ vanishes. 
Moreover, as $F_{N+1,1}$ induces an injection between the cohomology groups, and the image of $a'$ under $F_{N+1,1}$ is the coboundary $(b_{1} \circ \tilde{U}_{N+1})(\omega)$, 
$a'$ has to be itself a coboundary, but since $b'_{1}$ is zero, this means that $a'$ vanishes, so the same occurs to $(b_{1} \circ \tilde{U}_{N+1})(\omega)$, for $\omega$ arbitrary. 
This proves the lemma. 
\end{proof}

The following result is well-known and due to T. Kadeishvili (see \cite{K80}, Thm. 1, for the case of dg algebras, or \cite{K82}, Thm., for the case of plain $A_{\infty}$-algebras). 
We shall however provide the complete proof in the general case of $A_{\infty}$-algebras given by the previously mentioned author, which follows the usual patterns of the so called \emph{obstruction theory}, 
for we shall make use of it in the sequel. 
We also consider some extra conditions about the bigradings which are useful for our study of spectral sequences. 
\begin{theorem}
\label{theorem:kadeish} 
Let $(A,m_{\bullet})$ be a flat (resp., a flat unitary) $A_{\infty}$-algebra and $f_{1} : H^{\bullet}(A) \rightarrow A$ be the composition of a section of the canonical projection $\Ker(m_{1}) \rightarrow H^{\bullet}(A)$ and the inclusion $\Ker(m_{1}) \subseteq A$ (resp., satisfying that $f_{1} \circ \eta_{H^{\bullet}(A)} = \eta_{A}$). 
Then there exists a structure of (resp., unitary) $A_{\infty}$-algebra on $H^{\bullet}(E)$ given by $\{ \bar{m}_{n} \}_{n \in \NN}$ and a quasi-isomorphism of (resp., unitary) $A_{\infty}$-algebras 
$f_{\bullet} : H^{\bullet}(A) \rightarrow A$ whose first component is $f_{1}$, such that $\bar{m}_{1} = 0$ and $\bar{m}_{2}$ is the multiplication induced by $m_{2}$. 
Moreover, all these possible structures of (resp., unitary) $A_{\infty}$-algebras on $H^{\bullet}(A)$ are quasi-isomorphic. 
Any of these quasi-isomorphic (resp., unitary) $A_{\infty}$-structures will be called a \emph{model}. 
Furthermore, if $A$ has a compatible bigrading of $s$-th type, for some $s \in \NN_{0}$, then we may choose among these models one that is compatible with the induced bigrading on the cohomology $H^{\bullet}(A)$, 
such that the quasi-isomorphism of (resp., unitary) $A_{\infty}$-algebras $f_{\bullet} : H^{\bullet}(A) \rightarrow A$ can be chosen to be compatible with the bigradings. 
Any of these models on $H^{\bullet}(A)$ will be called \emph{compatible with the bigrading of $A$}. 
\end{theorem}
\begin{proof}
The result follows from an typical recurrence argument, which we now recall. 
First note that, since $f_{1}$ is included in the kernel of $m_{1}$ by construction, it satisfies that $m_{1} \circ f_{1} = f_{1} \circ \bar{m}_{1} = 0$, which is the first Stasheff identity on the morphisms $\mathrm{MI}(1)$. 
If $A$ is unitary, for $m_{1}(1_{A})$ vanishes, we choose $1_{H^{\bullet}(A)}$ to be the cohomology class of $1_{A}$, so it is possible to choose $f_{1}$ such that 
$f_{1} \circ \eta_{H^{\bullet}(A)} = \eta_{A}$. 
Suppose that for a given $N \in \NN$ we have constructed morphisms $\bar{m}_{n} : H^{\bullet}(A)^{\otimes n} \rightarrow H^{\bullet}(A)$ of degree $2-n$, for all $1 \leq n \leq N$
and $f_{n} : H^{\bullet}(A)^{\otimes n} \rightarrow A$ of cohomological degree $1- n$, for all $1 \leq n \leq N$, such that $\bar{m}_{1} = 0$, 
$f_{1}$ coincides with the morphism given in the statement of the theorem, and they satisfy the Stasheff identities $\mathrm{SI}(n)$ and $\mathrm{MI}(n)$, for $1 \leq n \leq N$. 
In the case we assume $A$ is unitary, we also suppose that $\bar{m}_{n}(\bar{a}_{1}, \dots, \bar{a}_{n})$ vanishes 
if there exists $j \in \{1, \dots, n \}$ such that $\bar{a}_{j} = 1_{H^{\bullet}(A)}$, for $3 \leq n \leq N$, and that  $f_{n}(\bar{a}_{1}, \dots, \bar{a}_{n})$ vanishes 
if there exists $j \in \{1, \dots, n \}$ such that $\bar{a}_{j} = 1_{H^{\bullet}(A)}$, for $2 \leq n \leq N$. 
This only means that $H^{\bullet}(A)$ has a structure of (resp., unitary) $A_{N}$-algebra and there is a quasi-isomorphism of (resp., unitary) $A_{N}$-algebras from $H^{\bullet}(A)$ to 
(the underlying (resp., unitary) $A_{N}$-algebra of) $A$.  
Define the morphism from $H^{\bullet}(A)^{\otimes (N+1)}$ to $A$ of degree $1-N$ given by 
\begin{equation}
\label{eq:ainftyalgebraobst}
\begin{split}
   U_{N+1} = &\sum_{(r,s,t) \in \mathcal{I}_{N+1}^{*}} (-1)^{r + s t}  f_{r + 1 + t} \circ (\mathrm{id}_{H^{\bullet}(A)}^{\otimes r} \otimes \bar{m}_{s} \otimes \mathrm{id}_{H^{\bullet}(A)}^{\otimes t}) 
   \\
   &- \sum_{q \in \NN_{\geq 2}} \sum_{\bar{i} \in \NN^{q, N+1}} (-1)^{w} m_{q} \circ (f_{i_{1}} \otimes \dots \otimes f_{i_{q}}),
\end{split}
\end{equation} 
where we recall that $\mathcal{I}_{m}^{*} = \{ (r,s,t) \in \NN_{0} \times \NN_{\geq 2} \times \NN_{0} : r + s + t = m,  0 < r + t <  m - 1 \}$, $w = \sum_{j=1}^{q} (q-j) (i_{j} - 1)$ and $\NN^{q,m}$ is the subset of $\NN^{q}$ of elements $\bar{i} = (i_{1},\dots,i_{q})$ such that $|\bar{i}| = i_{1} + \dots + i_{q} = m$. 
By Lemma \ref{lemma:ind} we have that $m_{1} \circ U_{N+1}$ vanishes. 
We define $\bar{m}_{N+1}$ as minus the composition of $U_{N+1}$ and the canonical projection $\pi : \mathrm{Ker}(m_{1}) \rightarrow H^{\bullet}(A)$. 
By the very definition of $U_{N+1}$ we see that $U_{N+1} + f_{1} \circ \bar{m}_{N+1} = (\mathrm{id}_{\mathrm{Ker}(m_{1})} - f_{1} \circ \pi) \circ U_{N+1}$ lies in the image of $m_{1}$, 
so there exists a morphism $f_{N+1} : H^{\bullet}(A)^{\otimes (N+1)} \rightarrow A$ (necessarily of degree $-N$) such that 
\begin{equation}
\label{eq:ainfdos}
    U_{N+1} + f_{1} \circ \bar{m}_{N+1} = m_{1} \circ f_{N+1},
\end{equation} 
which is just the Stasheff identity on morphisms $\mathrm{MI}(N+1)$. 
If we assumed $A$ is unitary, then a trivial verification shows that  we may choose $f_{N+1}(\bar{a}_{1}, \dots, \bar{a}_{N+1})$ such that it vanishes 
if there exists $j \in \{1, \dots, N+1 \}$ such that $\bar{a}_{j} = 1_{H^{\bullet}(A)}$. 
We see that $\bar{m}_{N+1}$ together with the previously considered multiplications $\bar{m}_{2}, \dots, \bar{m}_{N}$ satisfy the Stasheff identity $\mathrm{SI}(N+1)$. 
Indeed, these Stasheff identities can be easily proved by considering the bar constructions of $H^{\bullet}(A)$ and of $A$ up to tensor degree $(N + 1)$ with their corresponding coderivations $\bar{B}_{N+1}$ and $B_{N+1}$, constructed from $\bar{m}_{2}, \dots, \bar{m}_{N+1}$ and from $m_{1}, \dots, m_{N+1}$ respectively. 
The morphism identities up to degree $(N+1)$ are tantamount to show that the induced morphism $F_{N+1}$ by $f_{1}, \dots, f_{N+1}$ 
between the bar constructions of $H^{\bullet}(A)$ and of $A$ up to tensor degree $(N + 1)$ satisfies that $B_{N+1} \circ F_{N+1} = F_{N+1} \circ \bar{B}_{N+1}$. 
The fact that $A$ satisfies the $A_{N+1}$-algebra axioms tells us that $B_{N+1} \circ B_{N+1} = 0$. 
Taking into account that $F_{1}$ is injective, for $f_{1}$ is also by definition (see \cite{HR}, Prop. 2.4.2), we get that $\bar{B}_{N+1} \circ \bar{B}_{N+1} = 0$, which in turn implies the Stasheff identity $\mathrm{SI}(N+1)$ for $H^{\bullet}(A)$. 
If we assumed $A$ is unitary, then a trivial verification shows that $\bar{m}_{N+1}(\bar{a}_{1}, \dots, \bar{a}_{N+1})$ vanishes 
if there exists $j \in \{1, \dots, N+1 \}$ such that $\bar{a}_{j} = 1_{H^{\bullet}(A)}$. 
The last statement about the bigradings follows easily by considering the corresponding bidegrees in equations \eqref{eq:ainftyalgebraobst} and \eqref{eq:ainfdos}.  
\end{proof}

\begin{remark}
\label{remark:merkulov} 
In the case $A$ is a (unitary) dg algebra, S. Merkulov has given in \cite{Mer}, Thm. 3.4, a fairly explicit procedure to construct possible higher multiplications of the (unitary) 
$A_{\infty}$-structure on the cohomology space $H^{\bullet}(A)$. 
The corresponding quasi-isomorphism from the cohomology space $H^{\bullet}(A)$ to $A$ can be explicitly constructed (see \cite{LPWZ09}, Prop. 2.3). 
In the general case of an $A_{\infty}$-algebra, completely explicit expressions for the model and the quasi-isomorphism can be systematically given using \emph{homological perturbation theory} 
(see for instance \cite{GLS}). 
\end{remark}

\begin{remark}
\label{remark:impkade}
Suppose $k$ is a field. 
A far more general version, also proved using obstruction theory and the closed model category structure (with only some limits) is stated in \cite{LH}, Th\'eor\`eme 1.4.1.1. 
In fact, the first part of its proof shows that, if $(M,d_{M})$ is a complex of $k$-modules, provided with a surjective quasi-isomorphism $g : M \rightarrow A$ from $M$ to the underlying complex 
of an $A_{\infty}$-algebra $A$, then, since the kernel $\Ker(g)$ is contractible, the complex $M$ is isomorphic to the direct sum complex $K \oplus A$ such that $g$ is identified to the canonical projection 
$K \oplus A \rightarrow A$. 
This in turn implies that the product $A_{\infty}$-algebra structure on $K \oplus A$, where we regard $K$ as an $A_{\infty}$-algebra with $m_{i}^{K} = 0$, for $i \in \NN_{\geq 2}$, 
gives the desired $A_{\infty}$-algebra structure on $M$, and the canonical morphism of $A_{\infty}$-algebras from the product $K \prod A$ to $A$ extends (and in fact coincides with) 
the canonical projection of complexes $K \oplus A \rightarrow A$. 
We recall that the product $A_{\infty}$-algebra $A \prod A'$ of two $A_{\infty}$-algebras $A$ and $A'$ exists by \cite{LH}, Th\'eor\`eme 1.3.3.1, (A), but in this case a simple computation suffices: if $B(A)$ and $B(A')$ 
are the bar constructions of two $A_{\infty}$-algebras $A$ and $A'$, the product in the category of cocomplete dg coalgebras of $B(A)$ and $B(A')$ exists and it is given by the graded tensor coalgebra 
$\bar{T}^{c}(A[1] \oplus A'[1])$ with the induced differential by that of $B$ and $B'$. 
This is the bar construction of the product $A \prod A'$ (the universal property follows from the well-known fact stated in \cite{LH}, Lemme 1.1.2.2). 
Using this result together with the previous theorem (by writing out a quasi-isomorphism of complexes from $M$ to $A$ factorized through $H^{\bullet}(A)$) we see that, 
if there is a quasi-isomorphism of complexes between a complex $(M,d_{M})$ and the underlying complex of an $A_{\infty}$-algebra $A$, 
there exists an $A_{\infty}$-algebra structure on $M$ together with a quasi-isomorphism of $A_{\infty}$-algebras to $A$. 
\end{remark}

\section{\texorpdfstring{Deformations of $A_{\infty}$-algebras}{Deformations of A-infinity-algebras}} 
\label{section:defo}

The algebraic version of deformation theory (for associative algebras) was initiated by M. Gerstenhaber in \cite{Ger64}, and was further extended for dg algebras by the mentioned author and C. Wilkerson in \cite{GW96}. 
The theory for $A_{\infty}$-algebras is generalized straightforward, and was considered by A. Fialowski and M. Penkava in \cite{FP}, but was also in the unpublished article \cite{Wu} by E. Wu. 

In this section we will let $k$ be a commutative ring with unit and let $R$ be commutative unitary $k$-algebra $R$, considered to be concentrated in zero cohomological degree, 
provided with an augmentation morphism of $k$-algebras $\epsilon_{R} : R \rightarrow k$. 
We further suppose that $R$ is finitely generated as $k$-module. 
Such algebras are usually called \emph{test algebras}. 
This is not the usual setting of deformation theory, because one additionally imposes that $k$ is a field and $R$ is local, but this is more appropriate for our purposes. 
In this brief section, a graded $R$-module will denote a symmetric graded $R$-bimodule. 
Given an $A_{\infty}$-algebra $(A, m_{\bullet})$ over $k$, an \emph{$R$-deformation of $A$} consists in the structure of $A_{\infty}$-algebra over $R$ on the (cohomologically) graded $R$-module 
$A \otimes R = \oplus_{n \in \ZZ} (A^{n} \otimes R)$ given by morphisms of $R$-modules
\[     m_{n}^{R} : (A \otimes R)^{\otimes_{R} n} \rightarrow A \otimes R     \]
of degree $2-n$, for all $n \in \NN$, satisfying that $(\mathrm{id}_{A} \otimes \epsilon_{R}) \circ m_{n}^{R}$ restricted to $A^{\otimes n}$ coincides with $m_{n}$, for all $n \in \NN$. 
We remark that $A^{\otimes n}$ is canonically identified with a $k$-submodule of $(A \otimes R)^{\otimes_{R} n}$ by means of the map 
$a_{1} \otimes \dots \otimes a_{n} \mapsto (a_{1} \otimes 1_{R}) \otimes_{R} \dots \otimes_{R} (a_{n} \otimes 1_{R})$. 
Taking into account that $(A \otimes R)^{\otimes_{R} n} \simeq A^{\otimes n} \otimes R$, using the previous identification we may rephrase the last condition for the definition of deformation 
as $m_{n}^{R} \otimes_{R} \mathrm{id}_{k} = m_{n}$, for all $n \in \NN$. 
If $A$ has a unit $1_{A}$, we say that the deformation \emph{preserves the unit} if $1_{A} \otimes 1_{R} \in A \otimes R$ is a strict unit for the structure maps $\{m_{n}^{R}\}_{n \in \NN}$. 
Given two $R$-deformations $A'$ and $A''$ of the $A_{\infty}$-algebra $A$, a \emph{(resp., strict) morphism of $R$-deformations} from $A'$ to $A''$ is a (resp., strict) morphism $f_{\bullet}$ of $A_{\infty}$-algebras over $R$ from $A'$ to $A''$ such that 
$f_{\bullet} \otimes_{R} \mathrm{id}_{k}$ is the identity of $A$. 
If $A$ has a unit, and $A'$ and $A''$ preserve the unit, we shall further impose that $f_{\bullet}$ is a morphism of unitary $A_{\infty}$-algebras over $R$. 
Moreover, the $R$-deformation $A'$ and $A''$ are said to be \emph{equivalent} if there exists a morphism 
$f_{\bullet} : A' \rightarrow A''$ of $R$-deformations of $A$ such that the induced morphism $F$ between the bar constructions $B(A')$  of $A'$ and $B(A'')$ of $A''$ 
(where we remark that $A'$ and $A''$ are (resp., unitary) $A_{\infty}$-algebras over $R$) is an isomorphism of dg coalgebras over $R$.   

The typical examples we may be interested in are when $R = k[\hbar]/(\hbar^{(N+1)})$, for $N \in \NN$, and most particularly, when $k[\hbar]/(\hbar^{(N+1)})$ is regarded as a bigraded algebra with 
$\hbar$ of bidegree $(-1,1)$. 
In this case we will focus on $k[\hbar]/(\hbar^{(N+1)})$-deformations of (resp., unitary) $A_{\infty}$-algebras with a compatible bigrading of $s$-th type, for some $s \in \NN_{0}$ and $N \in \NN$, 
such that the bigrading of the previous deformations are compatible of $s$-th type when they are regarded as (resp., unitary) $A_{\infty}$-algebras over $k$. 
These deformations will be called \emph{bigraded}. 

If $A$ and $B$ are two (resp., unitary) $A_{\infty}$-algebras over $k$, and $A_{R}$ and $B_{R}$ are two $R$-deformations of $A$ and $B$, respectively, 
a \emph{morphism} from $A_{R}$ to $B_{R}$ is the data of a morphism $f_{\bullet}$ of (resp., unitary) $A_{\infty}$-algebras from $A$ to $B$, together with a morphism of 
(resp., unitary) $A_{\infty}$-algebras over $R$ given by a collection of maps 
\[     \tilde{f}_{n} : A_{R}^{\otimes_{R} n} \rightarrow B_{R},     \]
for $n \in \NN$, such that $\tilde{f}_{n} \otimes_{R} \mathrm{id}_{k}$ coincides with $f_{n}$, for all $n \in \NN$, 
where we are making use of the isomorphisms $(A \otimes R)^{\otimes_{R} n} \simeq A^{\otimes n} \otimes R$ explained at the previous paragraph. 
We shall say that a morphism $(f_{\bullet}, \tilde{f}_{\bullet})$ of $R$-deformations is a \emph{quasi-isomorphism} if both $f_{\bullet}$ and $\tilde{f}_{\bullet}$ are quasi-isomorphisms. 
If the (resp., unitary) $A_{\infty}$-algebras are provided with a compatible bigrading of $s$-th type, for some $s \in \NN_{0}$, $R = k[\hbar]/(\hbar^{(N+1)})$, for some $N \in \NN$, 
and the considered deformations are bigraded, we shall further assume that the morphisms $f_{\bullet}$ and $\tilde{f}_{\bullet}$ are compatible with the bigradings.

Note that, if $A'$ is an $R$-deformation of the (resp., unitary) $A_{\infty}$-algebra $A$ and let $T$ be a quotient of $R$. 
Then $A' \otimes_{R} T$ together with the maps $m_{n}^{R} \otimes_{R} \mathrm{id}_{T}$ is a $T$-deformation of $A$, called the \emph{deformation reduction to $T$ of $A'$}. 
Given a family of test algebras $\{ R_{i} \}_{i \in \NN}$ together with quotient maps $R_{i+1} \rightarrow R_{i}$ for $i \in \NN$ 
and a collection of (resp., unitary) $A_{\infty}$-algebras $\{A_{i}\}_{i \in \NN}$ such that $A_{i}$ is an 
$R_{i}$-deformation of $A$, we say that this family is \emph{compatible} if the deformation reduction of $A_{i+1}$ to $R_{i}$ given by the (resp., unitary) $A_{\infty}$-algebra 
$A_{i+1} \otimes_{R_{i+1}} R_{i}$ over $R_{i}$ and the $R_{i}$-deformation $A_{i}$ of $A$ are strictly isomorphic. 
Given a family of test algebras $\{ R_{i} \}_{i \in \NN}$ together with quotient maps $R_{i+1} \rightarrow R_{i}$ for $i \in \NN$, a \emph{morphism} from a compatible collection of 
deformations of (resp., unitary) $A_{\infty}$-algebras $\{ A_{i} \}_{i \in \NN}$ of $A$ to another compatible collection of deformations of (resp., unitary) $A_{\infty}$-algebras $\{ B_{i} \}_{i \in \NN}$ of $B$ 
is the data of a morphism $f_{\bullet} : A \rightarrow B$ of (resp., unitary) $A_{\infty}$-algebras together with collection of morphisms $\{ (f_{\bullet})_{i} \}_ {i \in \NN}$, where $(f_{\bullet})_{i}$ is a morphism of (resp., unitary) $A_{\infty}$-algebras over $R_{i}$ from $A_{i}$ to $B_{i}$, such that $(f_{\bullet}, (f_{\bullet})_{i})$ is a morphism of $R$-deformations from $A_{i}$ to $B_{i}$, for all $i \in \NN$, 
satisfying that $(f_{\bullet})_{i+1} \otimes_{R_{i+1}} \mathrm{id}_{R_{i}}$ coincides with $(f_{\bullet})_{i}$, for all $i \in \NN$. 
In this case we say that a morphism $(f_{\bullet}, \{ (f_{\bullet})_{i} \}_ {i \in \NN})$ of compatible families of deformation is a \emph{quasi-isomorphism} if each 
$(f_{\bullet}, (f_{\bullet})_{i})$ is a quasi-isomorphism of $R_{i}$-deformations, for all $i \in \NN$. 
The main example we will have in mind is for the family of test algebras $R_{i} = k[\hbar]/(\hbar^{i+1})$, for $i \in \NN$. 
In this case, we will follow all the grading assumption introduced in the two previous paragraphs. 

We would like to provide the notion of formal deformation of a (unitary) $A_{\infty}$-algebra. 
Since our interest comes from spectral sequences where there is a bigrading, we shall deal with the case that our (unitary) $A_{\infty}$-algebra has a compatible bigrading. 
Most of the results we present here are analogous to the standard ones on formal deformation theory, which can be found for instance in \cite{Ka}, Ch. XVI, Sections 1-4. 
Some of the proofs are completely parallel to ones given there, and in that cases we only give the corresponding exact reference.  
Our difference relies that we also want to deal with the corresponding bigrading. 
For this reason, we will briefly explain the corresponding analogous setting of our interest.  
We will thus consider the commutative bigraded complete $k$-algebra in one indeterminate $k[\hbar]$, which is supposed to be bigraded by setting $\hbar$ in bidegree $(-1,1)$ (and $k$ with trivial bidegree).  
As before, we have the obvious augmentation map sending $\hbar$ to zero which will be denoted by $\epsilon_{k[\hbar]}$. 
What we say holds if the bidegree of $\hbar$ is much more general but we will not need that. 
Given any bigraded module $M$ over $k[\hbar]$, the inverse limit in the category of bigraded $k[\hbar]$-modules 
\[     \underset{\leftarrow N}{\lim}^{\mathrm{bgr}} \hskip 0.6mm M \otimes_{k[\hbar]} k[\hbar]/(\hbar^{N}).     \]
will be called the \emph{bigraded completion} of $M$, and will be denoted by $\hat{M}^{\mathrm{bgr}}$. 
There is a canonical morphism 
\[     \tau_{M}^{\mathrm{bgr}} : M \rightarrow \hat{M}^{\mathrm{bgr}}     \]
of bigraded $k[\hbar]$-modules, and we will say that $M$ is \emph{bigraded complete} if this latter map is an isomorphism. 
Also note that, given two bigraded $k[\hbar]$-modules $M$ and $N$, any homogeneous morphism (of bigraded $k[\hbar]$-modules) $f$ from $M$ to $N$ 
of bidegree $(p,q)$ automatically induces a homogeneous morphism (of bigraded $k[\hbar]$-modules) $\hat{f}^{\mathrm{bgr}}$ from $\hat{M}^{\mathrm{bgr}}$ to $\hat{N}^{\mathrm{bgr}}$ of the same bidegree, 
and it is usually called the \emph{bigraded completion (morphism) of $f$}. 
Given two bigraded modules $M$ and $N$ over $k[\hbar]$, the \emph{bigraded completed tensor product} is defined as the bigraded completion 
of the usual tensor product $M \otimes_{k[\hbar]} N$, and it is denoted by 
\[      M \hat{\otimes}_{k[\hbar]}^{\mathrm{bgr}} N.     \] 
Moreover, if $V = \oplus_{p,q \in \ZZ} V^{p,q}$ is a bigraded module over $k$, we may consider the bigraded $k[\hbar]$-module given by the tensor product $V \otimes k[\hbar]$, and we regard $V$ inside it via 
$v \mapsto v \otimes 1_{k[\hbar]}$. 
The bigraded completion of the previous $k[\hbar]$-module would be denoted by $V[[\hbar]]^{\mathrm{bgr}}$ (which does not coincide in general with the previous usual tensor product). 
Notice that $V[[\hbar]]^{\mathrm{bgr}}$ coincides with the inverse limit in the category of bigraded $k[\hbar]$ given by 
\begin{equation}
\label{eq:grcompl}
     \underset{\leftarrow N}{\lim}^{\mathrm{bgr}} \hskip 0.6mm V \otimes k[\hbar]/(\hbar^{N}).     
\end{equation}
We see that $V[[\hbar]]^{\mathrm{bgr}}$ is given in more in more explicit terms by the bigraded $k$-module 
\begin{equation}
\label{eq:defo}
     \bigoplus_{p,q \in \ZZ} \left(\prod_{r \in \NN_{0}} (V^{p+r,q-r} \otimes k.\hbar^{r})\right),     
\end{equation}
together with the obvious action of $k[\hbar]$. 
Indeed, the inverse limit \eqref{eq:grcompl} is by very definition the bigraded $k$-module with homogeneous $(p,q)$-th component 
\[     \underset{\leftarrow N}{\lim} \hskip 0.6mm (V \otimes k[\hbar]/(\hbar^{N}))^{p,q} = \underset{\leftarrow N}{\lim} \hskip 0.6mm \Big(\prod_{j=0}^{N} V^{p+j,q-j} \otimes k.\hbar^{j}\Big),      \]
which yields 
\[     \prod_{r \in \NN_{0}} (V^{p+r,q-r} \otimes k.\hbar^{r}).     \]

If $f : M' \rightarrow M$ and $f : N' \rightarrow N$ are two homogeneous morphisms of bigraded $k[\hbar]$-modules of bidegrees $(p,q)$ and $(p',q')$, respectively, 
the bigraded completion of the morphism of $k[\hbar]$-modules 
\[     f \otimes_{k[\hbar]} g : M' \otimes_{k[\hbar]} N' \rightarrow M \otimes_{k[\hbar]} N     \] 
will be denoted by $f \hat{\otimes}_{k[\hbar]}^{\mathrm{bgr}} g$.  
Any bigraded $k[\hbar]$-module $M$ isomorphic to $V[[\hbar]]^{\mathrm{bgr}}$ for some free bigraded $k$-module $V$ will be called \emph{$\hbar$-topologically free}. 
We remark that we have chosen to add an $\hbar$ to the usual terminology (\textit{cf.} \cite{Ka}, p. 388), because in our (more general) situation an $\hbar$-topologically free module 
is not necessarily the completion of a free $k[\hbar]$-module. 
However, as we may easily see, all the ``non-freeness'' comes from the structure as module over $k$, and not from the action of $\hbar$. 
That is the reason of our terminology. 
Of course, if $k$ is a field, an $\hbar$-topologically free module coincides with the usual notion of topologically free module. 
Furthermore, it is easy to show that any bigraded $k[\hbar]$-module $M$ is $\hbar$-topologically free if and only if the canonical projection morphism $M \rightarrow M/\hbar.M$ of bigraded $k$-modules has a section, 
it is bigraded complete and \emph{$\hbar$-torsion-free}, \textit{i.e.} if $\hbar . m = 0$ for some $m \in M$, then $m = 0$ 
(the exact same proof as in \cite{Ka}, Prop. XVI.2.4, applies here as well). 
Given two free bigraded modules $V$ and $W$ over $k$, we have a canonical (homogeneous) isomorphism 
\[     (V \otimes W)[[\hbar]]^{\mathrm{bgr}} \simeq V[[\hbar]]^{\mathrm{bgr}} \hat{\otimes}_{k[\hbar]}^{\mathrm{bgr}} W[[\hbar]]^{\mathrm{bgr}}     \] 
of bigraded $k[\hbar]$-modules of trivial bidegree  (\textit{cf.} \cite{Ka}, Prop. XVI.3.2). 
Indeed, by the description of $\hbar$-topologically free modules given in \eqref{eq:defo}, we get that 
\[       (V \otimes W)[[\hbar]]^{\mathrm{bgr}} \simeq \bigoplus_{p,q \in \ZZ} \left(\prod_{t \in \NN_{0}} \Big(\bigoplus_{p',q' \in \ZZ} (V^{p',q'} \otimes W^{p+t-p',q-t-q'} \otimes k.\hbar^{t})\Big)\right),     \]
whereas 
\begin{align*}
       &V[[\hbar]]^{\mathrm{bgr}} \otimes_{k[\hbar]} W[[\hbar]]^{\mathrm{bgr}} 
         \\
    &\simeq 
\bigoplus_{p',p'',q',q'' \in \ZZ} \left(\Big(\prod_{r \in \NN_{0}} (V^{p'+r,q'-r} \otimes k.\hbar^{r})\Big) \otimes_{k[\hbar]}
\Big(\prod_{s \in \NN_{0}} W^{p''+s,q''-s} \otimes k.\hbar^{s})\Big)\right) 
\end{align*}
so 
\begin{multline*}
   (V[[\hbar]]^{\mathrm{bgr}} \otimes_{k[\hbar]} W[[\hbar]]^{\mathrm{bgr}}) \otimes_{k[\hbar]} k[\hbar]/(\hbar^{(N+1)})
   \\
   \simeq \bigoplus_{p,q \in \ZZ} \left(\prod_{t \in \NN_{0, \leq N}} \Big(\bigoplus_{p',q' \in \ZZ} (V^{p',q'} \otimes W^{p+t-p',q-t-q'} \otimes k.\hbar^{t})\Big)\right), 
\end{multline*}
and the claim follows by taking inverse limit. 
Note that, if $M$ is a bigraded complete $k[\hbar]$-module the canonical map $\mathcal{H}om_{k[\hbar]}(V[[\hbar]]^{\mathrm{bgr}},M) \rightarrow \mathcal{H}om_{k}(V,M)$ given by restriction is an isomorphism of $k[\hbar]$-modules, 
where we recall that $\mathcal{H}om$ denotes the corresponding internal space of morphisms in the category of bigraded $k$-modules (\textit{cf.} \cite{Ka}, Prop. XVI.2.3. The exact same proof works here as well). 
Finally, we remark that, given a free bigraded $k$-module $M$, the canonical map $M \otimes k[\hbar]/(\hbar^{N}) \rightarrow M[[\hbar]]^{\mathrm{bgr}} \otimes_{k[\hbar]} k[\hbar]/(\hbar^{N})$ of right $k[\hbar]/(\hbar^{N})$-modules 
induced by the morphism $m \otimes 1_{k[\hbar]/(\hbar^{N})} \mapsto (m \otimes 1_{k[\hbar]}) \otimes_{k[\hbar]} 1_{k[\hbar]/(\hbar^{N})}$ is an isomorphism. 

A \emph{formal bigraded deformation} $A_{\hbar}$ of an (resp., a unitary) $A_{\infty}$-algebra $(A, m_{\bullet})$ over $k$ provided with a compatible bigrading of $s$-th type, for some $s \in \NN_{0}$, 
is given by the $\hbar$-topologically free bigraded $k[\hbar]$-module 
$A[[\hbar]]^{\mathrm{bgr}}$, where we recall that $\hbar$ has bidegree $(-1,1)$, 
provided with a collection of morphisms  
\[     m_{n}^{\hbar} : A_{\hbar}^{\hat{\otimes}_{k[\hbar]}^{\mathrm{bgr}} n} \rightarrow A_{\hbar},     \]
for $n \in \NN$, of bidegree $(2-n) (s,-s+1)$ for $n \in \NN$, such that the collection
\begin{equation}
\label{eq:intdefo}
     \{ (A_{\hbar} \otimes_{k[\hbar]} k[\hbar]/(\hbar^{(N+1)}), m_{\bullet}^{\hbar} \otimes_{k[\hbar]} \mathrm{id}_{k[\hbar]/(\hbar^{(N+1)})}) \}_{N \in \NN}     
\end{equation}
forms a compatible family of (bigraded) deformations of the (resp., unitary) $A_{\infty}$-algebra $A$ with respect to the family of test bigraded algebras $\{ k[\hbar]/(\hbar^{(N+1)}) \}_{N \in \NN}$, considered with the canonical projections. 
This implies that the restriction of the composition of $m_{n}^{\hbar}$ with $(\mathrm{id}_{A} \otimes \epsilon_{k[\hbar]})$ to $A^{\otimes n}$ coincides with $m_{n}$, for all $n \in \NN$. 
We remark that there is a canonical injection  
\[     A^{\otimes n} \rightarrow A_{\hbar}^{\hat{\otimes}_{k[\hbar]}^{\mathrm{bgr}} n}     \]
of bigraded $k$-submodule by means of the map $a_{1} \otimes \dots \otimes a_{n} \mapsto (a_{1} \otimes 1_{k[\hbar]}) \otimes_{k[\hbar]} \dots \otimes_{k[\hbar]} (a_{n} \otimes 1_{k[\hbar]})$. 
Note that by definition each of the deformations \eqref{eq:intdefo} are in fact (resp., unitary) $A_{\infty}$-algebras provided with a compatible bigrading of $s$-th type (when regarded over $k$). 

\begin{remark}
\label{remark:defofor}
We could have defined a formal bigraded deformation $A_{\hbar}$ of an (resp., a unitary) $A_{\infty}$-algebra $A$ as an $\hbar$-topologically free bigraded $k[\hbar]$-module provided with a structure of 
(resp., unitary) $A_{\infty}$-algebra over $k[\hbar]$, such that the underlying (resp., unitary) $A_{\infty}$-algebra over $k$ has a compatible bigrading of $s$-th type, 
and the reduction $A_{\hbar} \otimes_{k[\hbar]} k$ is strictly isomorphic to $A$. 
We note that in this case the structure maps $\tilde{m}_{i}$ of $A_{\hbar}$ should be given by maps
\[      \tilde{m}_{i} : A_{\hbar}^{\otimes_{k[\hbar]} i} \rightarrow A_{\hbar}     \] 
instead of having as domain the completed tensor products appearing in the definition of formal bigraded deformation. 
However, since $A_{\hbar}$ is complete and $\tilde{m}_{i}$ is $k[\hbar]$-linear it induces a structure map $m_{i}^{\hbar}$ as in the definition of the previous paragraph. 
A similar remark can be stated for the definition of morphisms in the next paragraph. 
We believe however that the proper definition should always rely on an inverse system of the type we considered, which is the usual manner one follows to handle the general case (\textit{e.g.} if $\hbar$ is ungraded, etc).  
\end{remark}

If $A$ and $B$ are two (resp., unitary) $A_{\infty}$-algebras over $k$ provided with compatible bigradings of $s$-th type, for some $s \in \NN_{0}$, 
and $A_{\hbar}$ and $B_{\hbar}$ are two formal bigraded deformations of $A$ and $B$, respectively, 
a \emph{morphism} from $A_{\hbar}$ to $B_{\hbar}$ is the data of a morphism $f_{\bullet}$ of (resp., unitary) $A_{\infty}$-algebras from $A$ to $B$ compatible with the bigradings, together with a collection of maps 
\[     \tilde{f}_{n} : (A_{\hbar}^{\mathrm{bgr}})^{\hat{\otimes}_{k[\hbar]}^{\mathrm{bgr}} n} \rightarrow B_{\hbar}^{\mathrm{bgr}},     \]
for $n \in \NN$, of bidegree $(1-n) (s,-s+1)$ for $n \in \NN$, such that the collection
\begin{equation}
\label{eq:intdefobis}
     \{ (\tilde{f}_{i} \otimes_{k[\hbar]} k[\hbar]/(\hbar^{(N+1)}))_{i \in \NN}  \}_{N \in \NN}     
\end{equation}
forms a morphism of compatible collections a bigraded deformations of (resp., unitary) $A_{\infty}$-algebras, with respect to the family of test bigraded algebras given by $\{ k[\hbar]/(\hbar^{(N+1)}) \}_{N \in \NN}$, 
considered with the canonical projections. 

\begin{remark}
\label{remark:quasi-isodefbis}
Let $A_{\hbar}$ and $B_{\hbar}$ be formal bigraded deformations of the (resp., unitary) $A_{\infty}$-algebra $A$ and $B$, resp., where the two latter are provided with a compatible bigrading of $s$-th type, 
for some $s \in \NN_{0}$. 
Let $(f_{\bullet}, \tilde{f}_{\bullet})$ be a morphism of formal bigraded deformations from the deformation $A_{\hbar}$ of $A$ to the deformation $B_{\hbar}$ of $B$. 
It is a quasi-isomorphism if and only if the morphism $\tilde{f}_{1}$ is a quasi-isomorphism between the corresponding underlying complexes over $k[\hbar]$. 
This follows from the easy fact, when writing out the definition that the induced morphism between the cohomology groups is an isomorphism, that the term corresponding to the zeroth power of 
$\hbar$ gives exactly the definition for $f_{1} = f_{1}^{0}$ to be a quasi-isomorphism. 
\end{remark}

The following result is similar to the one proved by Lapin in \cite{La02a}, Thm. 3.1 and Cor. 3.1, but our proof follows the original Kadeishvili idea of obstruction theory, instead that of homological perturbation theory. 
As a detail, we do not need $k$ to be a field and we also take into account the bigrading we introduced previously. 
\begin{theorem}
\label{thm:quasiisofor}
Let $A$ be an (resp., a unitary) $A_{\infty}$-algebra over $k$ provided with a compatible bigrading of $s$-th type, for some $s \in \NN_{0}$, and let $A_{\hbar}$ be a formal bigraded deformation of $A$. 
Let us consider $H^{\bullet}(A)$ provided with a model compatible with the bigrading of $A$ and $f_{\bullet} : H^{\bullet}(A) \rightarrow A$ a quasi-isomorphism of (resp., strictly unitary) 
$A_{\infty}$-algebras compatible with the corresponding bigradings. 
Then there exists a formal bigraded deformation $H^{\bullet}(A)_{\hbar}$ of $H^{\bullet}(A)$ and a quasi-isomorphism $\tilde{f}_{\bullet} : H^{\bullet}(A)_{\hbar} \rightarrow A_{\hbar}$ 
of formal bigraded deformations of (resp., strictly unitary) $A_{\infty}$-algebras, \textit{i.e.} such that the underlying morphism of (resp., strictly unitary) $A_{\infty}$-algebras over $k$ is compatible with the bigradings, 
and we have the commutative diagram 
\[     
\xymatrix
{
H^{\bullet}(A)_{\hbar} 
\ar[r]^(.6){\tilde{f}_{\bullet}}
\ar[d]
&
A_{\hbar}
\ar[d]
\\
H^{\bullet}(A)
\ar[r]^(.6){f_{\bullet}}
&
A
}
\]
where the vertical maps are the canonical projections. 
Moreover, all these possible structures of formal bigraded deformations on $H^{\bullet}(A)$ are quasi-isomorphic.
\end{theorem}
\begin{proof}
The proof will follow a similar pattern to the one of Kadeishvili's theorem we recalled at the end of the previous section, but with some subtleties coming from the deformation part. 

Let us consider the following setting. 
We fix $N \in \NN$, and we shall only consider the underlying $A_{N}$-algebra structures of $A$ and $H^{\bullet}(A)$. 
Moreover, given $N' \in \NN$, by taking a quotient modulo $\hbar^{(N'+1)}$, we consider $A_{\hbar^{(N'+1)}} = A_{\hbar} \otimes_{k[\hbar]} k[\hbar]/(\hbar^{(N'+1)})$ a $k[\hbar]/(\hbar^{(N'+1)})$-deformation of $A$. 
We shall assume we have defined a bigraded $k[\hbar]/(\hbar^{N'})$-deformation $H^{\bullet}(A)_{\hbar^{N'}}$ of $H^{\bullet}(A)$ and a morphism $\{ \tilde{f}_{i} \}_{i \in \NN_{\leq N}}$ of $A_{N}$-algebras 
over $k[\hbar]/(\hbar^{N'})$ such that the diagram 
\[     
\xymatrix
{
H^{\bullet}(A)_{\hbar^{N'}} 
\ar[r]^(.6){(\tilde{f}_{\bullet})_{\bullet \in \NN_{\leq N}}}
\ar[d]
&
A_{\hbar^{N'}}
\ar[d]
\\
H^{\bullet}(A)
\ar[r]^(.6){(f_{\bullet})_{\bullet \in \NN_{\leq N}}}
&
A
}
\]
commutes. 
The bigrading assumption on the deformation $H^{\bullet}(A)_{\hbar^{N'}}$ means that we may write $\tilde{f}_{i}|_{H^{\bullet}(A)^{\otimes i}} = \sum_{j = 0}^{N'-1} \tilde{f}_{i}^{j} \hbar^{j}$, 
where $\tilde{f}_{i}^{j} : H^{\bullet}(A)^{\otimes i} \rightarrow A$  is a homogeneous map of bidegree $(1-i)(s,-s+1) + (j,-j)$, for $i \in \NN_{\leq N}$ and $j \in \NN_{0, \leq (N'-1)}$, 
where we recall that $\NN_{0,\leq p}$ denotes the subset of $\NN_{0}$ of nonnegative integers less than or equal to $p$. 
The previous commutativity means that $\tilde{f}_{i}^{0} = f_{i}$, for all $i \in \NN_{\leq N}$. 
If $\{ \tilde{m}_{i} \}_{i \in \NN_{\leq N}}$ and $\{ \tilde{\bar{m}}_{i} \}_{i \in \NN_{\leq N}}$ denote the structure maps of the $A_{N}$-algebras over $k[\hbar]/(\hbar^{(N'+1)})$ and over $k[\hbar]/(\hbar^{N'})$ given by 
$A_{\hbar^{(N'+1)}}$ and $H^{\bullet}(A)_{\hbar^{N'}}$, resp., we will write $\tilde{m}_{i}|_{A^{\otimes i}} = \sum_{j=0}^{N'} m_{i}^{j} \hbar^{j}$ 
and $\tilde{\bar{m}}_{i}|_{H^{\bullet}(A)^{\otimes i}} = \sum_{j=0}^{N'-1} \bar{m}_{i}^{j} \hbar^{j}$, for $i \in \NN_{\leq N}$, where $m_{i}^{j} : A^{\otimes i} \rightarrow A$ and 
$\bar{m}_{i}^{j} : H^{\bullet}(A)^{\otimes i} \rightarrow H^{\bullet}(A)$ are morphisms of bidegree 
$(2-i)(s,-s+1) + (j,-j)$, for $i \in \NN_{\leq N}$, and $j \in \NN_{0, \leq N'}$ or $j \in \NN_{0, \leq (N'-1)}$, respectively. 
Suppose moreover that we have defined homogeneous maps $\bar{m}_{i}^{N'} : H^{\bullet}(A)^{\otimes i} \rightarrow H^{\bullet}(A)$ of bidegree 
$(2-i)(s,-s+1) + (N',-N')$, for $i \in \NN_{\leq (N-1)}$, together with morphisms $\tilde{f}_{i}^{N'} : H^{\bullet}(A)^{\otimes i} \rightarrow A$ of bidegree $(1-i)(s,-s+1) + (N',-N')$, for $i = 1, \dots, N-1$ 
such that they provide a morphism of $A_{N-1}$-algebras between the reductions $H^{\bullet}(A)_{\hbar^{(N'+1)}}$ and $A_{\hbar^{(N'+1)}}$. 
In order to prove the theorem it suffices to show the following statement: we may extend the previously defined morphism from a (also to be defined) $k[\hbar]/(\hbar^{(N'+1)})$-deformation 
$H^{\bullet}(A)_{\hbar^{(N'+1)}}$ (extending $H^{\bullet}(A)_{\hbar^{N'}}$) to $A_{\hbar^{(N'+1)}}$, regarded as $A_{N}$-algebras over $k[\hbar]/(\hbar^{(N'+1)})$. 
Indeed, suppose that we proved the previous statement. 
We apply this procedure inductively as follows. 
First, for fixed $N = 1$, we prove it for any $N' \in \NN$ by induction. 
Then, each time we increase the value of $N$ in one, we use only part of the already constructed structure maps $\{ \tilde{\bar{m}}_{i} \}_{i \in \NN_{\leq N}}$ 
and morphisms $\{ \tilde{f}_{i} \}_{i \in \NN_{\leq N}}$: those involving terms $\bar{m}_{i}^{j}$ and  $\tilde{f}_{i}^{j}$, with $j \leq (N'-1)$, and apply the previous statement. 

Let us now prove the statement. 
We consider the map from $H^{\bullet}(A)^{\otimes N}$ to $A$ of bidegree $(2-N)(s,-s+1)+(N',-N')$ given by 
\begin{equation}
\label{eq:ainftyalgebraobstdefo}
\begin{split}
   U_{N}^{N'} = &\sum_{(r,s,t) \in \mathcal{I}_{N}} \sum_{j \in \NN_{0,\leq N'}^{s}} (-1)^{r + s t}  \tilde{f}_{r + 1 + t}^{j} \circ 
  (\mathrm{id}_{H^{\bullet}(A)}^{\otimes r} \otimes \bar{m}_{s}^{N'-j} \otimes \mathrm{id}_{H^{\bullet}(A)}^{\otimes t}) 
   \\
   &- \sum_{q \in \NN} \sum_{\bar{i} \in \NN^{q, N}} \sum_{\bar{j} \in (\NN^{q+1, N'})^{*}} (-1)^{w} m_{q}^{j_{0}} \circ (\tilde{f}_{i_{1}}^{j_{1}} \otimes \dots \otimes \tilde{f}_{i_{q}}^{j_{q}}),
\end{split}
\end{equation} 
where $\bar{j} = (j_{0}, \dots, j_{q})$, 
\[    \NN_{0,\leq N'}^{s} = \begin{cases} \NN_{0, \leq N'}, &\text{if $s \neq N$},
                                            \\ 
                                            \NN_{\leq N'}, &\textit{else},
                                           \end{cases} 
\] 
and
\[    (\NN^{q+1, N'})^{*} = \begin{cases} \NN^{q+1, N'}, &\text{if $q \neq 1$},
                                            \\ 
                                            \NN^{2, N'} \setminus \{ (0,N') \}, &\textit{else}.
                                           \end{cases} 
\] 
By Lemma \ref{lemma:ind2} we have that $m_{1}^{0} \circ U_{N}^{N'}$ vanishes. 
We define $\bar{m}_{N}^{N'}$ as minus the composition of $U_{N}^{N'}$ and the canonical projection $\pi : \mathrm{Ker}(m_{1}) \rightarrow H^{\bullet}(A)$. 
By the very definition of $U_{N}^{N'}$ we see that $U_{N}^{N'} + f_{1}^{0} \circ \bar{m}_{N}^{N'} = (\mathrm{id}_{\mathrm{Ker}(m_{1})} - f_{1}^{0} \circ \pi) \circ U_{N}^{N'}$ lies in the image of $m_{1}^{0}$, 
so there exists a morphism $f_{N}^{N'} : H^{\bullet}(A)^{\otimes N} \rightarrow A$ (necessarily of bidegree $(1-N)(s,-s+1)+(N',-N')$) such that 
\begin{equation}
\label{eq:ainfdosdos}
    U_{N}^{N'} + f_{1}^{0} \circ \bar{m}_{N}^{N'} = m_{1}^{0} \circ f_{N}^{N'},
\end{equation} 
which is just the term of the restriction to $H^{\bullet}(A)^{\otimes N}$ of the Stasheff identity on morphisms $\mathrm{MI}(N)$ of $A_{N}$-algebras over $k[\hbar]/(\hbar^{(N'+1)})$ which is multiplied by $\hbar^{N'}$. 

If we assumed $A$ is unitary, then a trivial verification shows that we may choose $\tilde{f}_{N}(\bar{a}_{1}, \dots, \bar{a}_{N})$ satisfying that it vanishes 
if there exists $j \in \{1, \dots, N \}$ such that $\bar{a}_{j} = 1_{H^{\bullet}(A)}$. 
We have that $\tilde{\bar{m}}_{N}$ together with the previously considered multiplications $\tilde{\bar{m}}_{1}, \dots, \tilde{\bar{m}}_{N-1}$ satisfy the Stasheff identity $\mathrm{SI}(N-1)$ 
as $A_{(N-1)}$-algebras over $k[\hbar]/(\hbar^{(N'+1)})$. 
Indeed, these Stasheff identities can be easily proved by considering the bar constructions of $H^{\bullet}(A)_{\hbar^{(N'+1)}}$ and of $A_{\hbar^{(N'+1)}}$ up to tensor degree $N$ with their corresponding coderivations $\tilde{\bar{B}}_{N}$ and $\tilde{B}_{N}$, constructed from $\tilde{\bar{m}}_{1}, \dots, \tilde{\bar{m}}_{N}$ and from $\tilde{m}_{1}, \dots, \tilde{m}_{N}$ respectively. 
The morphism identities up to tensor degree $N$ are tantamount to show that the morphism $\tilde{F}_{N}$ induced by $\tilde{f}_{1}, \dots, \tilde{f}_{N}$ 
between the bar constructions of $H^{\bullet}(A)_{\hbar^{(N'+1)}}$ and of $A_{\hbar^{(N'+1)}}$ up to tensor degree $N$ satisfies that $\tilde{B}_{N} \circ \tilde{F}_{N} = \tilde{F}_{N} \circ \tilde{\bar{B}}_{N}$. 
The fact that $A_{\hbar^{(N'+1)}}$ is an $A_{N}$-algebra over $k[\hbar]/(\hbar^{(N'+1)})$ tells us that $\tilde{B}_{N} \circ \tilde{B}_{N} = 0$. 
Since $\tilde{f}_{1}$ is injective, for $\tilde{f}_{1}^{0}$ is also by definition (this follows from a straightforward computation), which in turn implies that $\tilde{F}_{N}$ is injective (see \cite{HR}, Prop. 2.4.2), 
we get that $\tilde{\bar{B}}_{N} \circ \tilde{\bar{B}}_{N} = 0$, which in turn implies the Stasheff identity 
$\mathrm{SI}(N)$ for $H^{\bullet}(A)_{\hbar^{(N'+1)}}$. 

If $A$ is assumed to be unitary, a trivial calculation shows that $\tilde{\bar{m}}_{N}(\bar{a}_{1}, \dots, \bar{a}_{N})$ vanishes 
if there exists $j \in \{1, \dots, N \}$ such that $\bar{a}_{j} = 1_{H^{\bullet}(A)}$. 
Moreover, by Lemma \ref{lemma:ind3quasiiso}, the map $\tilde{f}_{1}$ is a quasi-isomorphism. 
The last statement of the theorem also follows from the mentioned Lemma and the theorem is thus proved. 
\end{proof}

We now state the lemmas required in the proof of the previous theorem. 
The proof of the first of those follows a similar spirit to the one of Lemma \ref{lemma:ind}, though there are some subtleties and differences, and we find it is not immediate, so we give it completely. 
\begin{lemma}
\label{lemma:ind2}
Let $N \in \NN$, and let $(A,\{ m_{i} \}_{i \in \NN_{\leq N}})$ and $(A',\{ m'_{i} \}_{i \in \NN_{\leq N}})$ be two $A_{N}$-algebras such that $m'_{1} = 0$. 
Let $N' \in \NN$, and let $A_{\hbar^{(N'+1)}}$ be a bigraded $k[\hbar]/(\hbar^{(N'+1)})$-deformation of $A$, 
and $A'_{\hbar^{N'}}$ a bigraded $k[\hbar]/(\hbar^{N'})$-deformation of $A'$. 
Suppose given a morphism $\{ \tilde{f}_{i} \}_{i \in \NN_{\leq N}}$ of $A_{N}$-algebras over $k[\hbar]/(\hbar^{N'})$ such that the diagram 
\[     
\xymatrix
{
A'_{\hbar^{N'}} 
\ar[r]^{(\tilde{f}_{\bullet})_{\bullet \in \NN_{\leq N}}}
\ar[d]
&
A_{\hbar^{N'}}
\ar[d]
\\
A'
\ar[r]^{(f_{\bullet})_{\bullet \in \NN_{\leq N}}}
&
A
}
\]
commutes.
Suppose $f_{1}$ induces an injective map between the corresponding cohomology groups.  
Define the map from $(A')^{\otimes N}$ to $A$ given by 
\begin{equation*}
\begin{split}
   U_{N}^{N'} = &\sum_{(r,s,t) \in \mathcal{I}_{N}} \sum_{j \in \NN_{0,\leq (N')}^{s}} (-1)^{r + s t}  \tilde{f}_{r + 1 + t}^{j} \circ 
  (\mathrm{id}_{A'}^{\otimes r} \otimes (m')_{s}^{N'-j} \otimes \mathrm{id}_{A'}^{\otimes t}) 
   \\
   &- \sum_{q \in \NN} \sum_{\bar{i} \in \NN^{q, N}} \sum_{\bar{j} \in (\NN^{q+1, N'})^{*}} (-1)^{w} m_{q}^{j_{0}} \circ (\tilde{f}_{i_{1}}^{j_{1}} \otimes \dots \otimes \tilde{f}_{i_{q}}^{j_{q}}),
\end{split}
\end{equation*} 
where $\bar{j} = (j_{0}, \dots, j_{q})$, 
\[    \NN_{0,\leq N'}^{s} = \begin{cases} \NN_{0, \leq N'}, &\text{if $s \neq N$},
                                            \\ 
                                            \NN_{\leq N'}, &\textit{else},
                                           \end{cases} 
\] 
and
\[    (\NN^{q+1, N'})^{*} = \begin{cases} \NN^{q+1, N'}, &\text{if $q \neq 1$},
                                            \\ 
                                            \NN^{2, N'} \setminus \{ (0,N') \}, &\textit{else}.
                                           \end{cases} 
\] 
Then $m_{1}^{0} \circ U_{N}^{N'}$ vanishes. 
\end{lemma}
\begin{proof}
Choose any homogeneous maps $(m')_{N}^{N'} : (A')^{\otimes N} \rightarrow A'$ of bidegree $(2- N) (s,-s+1) + (N',-N')$ and $\tilde{f}_{N}^{N'} : (A')^{\otimes N} \rightarrow A$ of bidegree 
$(1- N) (s,-s+1) + (N',-N')$. 
By the explanations on the bar construction of $A_{N}$-algebras given at the third paragraph of Subsection \ref{subsection:kad}, the maps $\{ \tilde{m}'_{i} \}_{i \in \NN_{\leq N}}$ given by $\tilde{m}'_{i}|_{A'} = \sum_{j=0}^{N} (m')_{i}^{j} \hbar^{j}$ induce a coderivation $\tilde{B}'_{N}$ on the tensor coalgebra $B_{N}(A'_{\hbar^{(N'+1)}})$. 
As $A_{\hbar^{(N'+1)}}'$ is not a $A_{N}$-algebra over $k[\hbar]/(\hbar^{(N'+1)})$ this is in fact an abuse of notation, ans we should probably write instead something like 
$\bar{T}^{c,\leq N}_{k[\hbar]/(\hbar^{(N'+1)})}(A'_{\hbar^{(N'+1)}}[1])$, 
but we will allow our notation for the sake of simplicity, and because it is completely clear that we are not assuming that $\tilde{B}'_{N} \circ \tilde{B}'_{N}$ vanishes. 
However, since $A'_{\hbar^{(N'+1)}}$ is an $A_{N}$-algebra over $k[\hbar]/(\hbar^{N'})$, $(\pi')_{1}^{j} \circ \tilde{B}'_{N} \circ \tilde{B}'_{N} \circ (\iota')_{p}^{0}$ vanishes for $p = 1, \dots, N-1$ and $j = 0, \dots, N'-1$, 
where $(\iota')_{p}^{0} : (A'[1])^{\otimes p} \rightarrow B_{N}(A'_{\hbar^{(N'+1)}})$ denotes the canonical inclusion, and $(\pi')_{1}^{j}$, for $j = 0, \dots, N'$, denotes the composition of the canonical projections 
$B_{N}(A'_{\hbar^{(N'+1)}}) \rightarrow A'_{\hbar^{(N'+1)}}$ and $A'_{\hbar^{(N'+1)}} \rightarrow A'$, where the latter map is given by $\sum_{l=0}^{N'} a_{l} \hbar^{l} \mapsto a_{j}$. 
We use the analogous notation without primes for $A$. 
By also following the recipe on the bar construction of a morphism of $A_{N}$-algebras given at the second paragraph before Lemma \ref{lemma:ind}, the maps $\{ \tilde{f}_{i} \}_{i \in \NN_{\leq N}}$ induce a morphism of graded coalgebras $\tilde{F}_{N}$ 
from $B_{N}(A'_{\hbar^{(N'+1)}})$ to $B_{N}(A_{\hbar^{(N'+1)}})$. 
Define $\tilde{U}_{N}^{N'}$ as $s_{A} \circ U_{N}^{N'} \circ (s_{A'}^{\otimes N})^{-1}$. 
It is direct to show that 
\begin{equation*}
\begin{split}
  \tilde{U}_{N}^{N'} = &\sum_{(r,s,t) \in \mathcal{I}_{N}} \sum_{j \in \NN_{0,\leq (N')}^{s}} \tilde{F}_{N,r + 1 + t}^{j} \circ 
  (\mathrm{id}_{A'}^{\otimes r} \otimes (b')_{s}^{N'-j} \otimes \mathrm{id}_{A'}^{\otimes t}) 
   \\
   &- \sum_{q \in \NN} \sum_{\bar{i} \in \NN^{q, N}} \sum_{\bar{j} \in (\NN^{q+1, N'})^{*}} b_{q}^{j_{0}} \circ (\tilde{F}_{N,i_{1}}^{j_{1}} \otimes \dots \otimes \tilde{F}_{N,i_{q}}^{j_{q}}).
\end{split}
\end{equation*} 
It is clear that that the vanishing of $m_{1}^{0} \circ U_{N}^{N'}$ is tantamount to the vanishing of $b_{1}^{0} \circ \tilde{U}_{N}^{N'}$. 
Moreover, it is easy to see that 
\[     \tilde{U}_{N}^{N'} = \pi_{1}^{N'} \circ (\tilde{F}_{N+1} \circ \tilde{B}'_{N+1} - \tilde{B}_{N+1} \circ \tilde{F}_{N+1}) \circ (\iota')_{N}^{0} - \tilde{F}_{N,1}^{0} \circ (b')_{N}^{N'} + b_{1}^{0} \circ \tilde{F}_{N,N}^{N'}.     \]
Then 
\begin{align*}     
b_{1}^{0} \circ \tilde{U}_{N}^{N'} &= b_{1}^{0} \circ \pi_{1}^{N'} \circ (\tilde{F}_{N} \circ \tilde{B}'_{N} - \tilde{B}_{N} \circ \tilde{F}_{N}) \circ (\iota')_{N}^{0}     
\\
&= \tilde{B}_{N} \circ (\tilde{F}_{N} \circ \tilde{B}'_{N} - \tilde{B}_{N} \circ \tilde{F}_{N}) \circ (\iota')_{N}^{0}
\\
&= \tilde{B}_{N} \circ \tilde{F}_{N} \circ \tilde{B}'_{N}  \circ (\iota')_{N}^{0},
\end{align*}
where we have used in the first equality that $b_{1}^{0} \circ b_{1}^{0} = 0$ and $b_{1}^{0} \circ \tilde{F}_{N,1}^{0} = 0$, for $m_{1}^{0} \circ m_{1}^{0} = 0$ and $m_{1}^{0} \circ \tilde{f}_{1}^{0} = 0$, 
and in the second equality that the image of $(\tilde{F}_{N} \circ \tilde{B}'_{N} - \tilde{B}_{N} \circ \tilde{F}_{N}) \circ (\iota')_{N}^{0}$ is included in the image of $\iota_{1}^{0} \hbar^{N'}$, 
by the hypothesis that $\{ \tilde{f}_{i} \}_{i \in \NN_{\leq N}}$ is a morphism of $A_{N}$-algebras over $k[\hbar]/(\hbar^{N'})$. 
In the last equality we have used that $\tilde{B}_{N} \circ \tilde{B}_{N} = 0$, for $A_{\hbar^{(N'+1)}}$ is an $A_{N}$-algebra. 
Since the image of $\tilde{B}'_{N}  \circ (\iota')_{N}^{0}$ is included in $B_{N}(A')$, because $(b')_{1}^{0}$ vanishes, the restriction to $B_{N}(A'_{\hbar^{N'}})$ of the morphism $\tilde{F}_{N}$ from $B_{N}(A'_{\hbar^{(N'+1)}})$ to $B_{N}(A_{\hbar^{(N'+1)}})$ coincides with the morphism $\tilde{F}_{N}$ from $B_{N}(A'_{\hbar^{N'}})$ to $B_{N}(A_{\hbar^{N'}})$, and 
$\tilde{F}_{N} \circ \tilde{B}'_{N} = \tilde{B}_{N} \circ \tilde{F}_{N}$, we see that $b_{1}^{0} \circ \tilde{U}_{N}^{N'}$ coincides with 
\[      \tilde{B}_{N} \circ \tilde{F}_{N} \circ \tilde{B}'_{N}  \circ (\iota')_{N}^{0} = \tilde{F}_{N} \circ \tilde{B}'_{N} \circ \tilde{B}'_{N}  \circ (\iota')_{N}^{0}.     \]
Note that $\tilde{B}'_{N} \circ \tilde{B}'_{N}  \circ (\iota')_{N}^{0} = \pi_{1}^{N'} \circ \tilde{B}'_{N} \circ \tilde{B}'_{N}  \circ (\iota')_{N}^{0}$, 
because $A'_{\hbar^{N'}}$ is an $A_{N}$-algebra over $k[\hbar]/(\hbar^{N'})$ (see \cite{LH}, Lemme B.1.1, a), 
and thus $b_{1}^{0} \circ \tilde{U}_{N}^{N'}$ coincides with 
\[      \tilde{F}_{N,1}^{N'} \circ (\pi_{1}^{N'} \circ \tilde{B}'_{N} \circ \tilde{B}'_{N}  \circ (\iota')_{N}^{0}).     \]
Choose any $\omega \in (A'[1])^{\otimes N}$, and consider $a' = (\pi_{1}^{N'} \circ \tilde{B}'_{N} \circ \tilde{B}'_{N}  \circ (\iota')_{N}^{0}) (\omega)$. 
It is a cocycle in $A'[1]$, because $(b')_{1}^{0}$ vanishes. 
As $\tilde{F}_{N,1}^{N'}$ induces an injection between the cohomology groups, and the image of $a'$ under $\tilde{F}_{N,1}^{N'}$ is the coboundary $(b_{1}^{0} \circ \tilde{U}_{N}^{N'})(\omega)$, 
$a'$ has to be itself a coboundary, but as $(b')_{1}^{0}$ is zero, this means that $a'$ vanishes, so the same occurs to $(b_{1}^{0} \circ \tilde{U}_{N}^{N'})(\omega)$, for $\omega$ arbitrary. 
The lemma is proved. 
\end{proof}

The following result is also required in the proof of the previous theorem. 
We believe that it should be well-known among the experts, but surprisingly we could not find any proof in the literature whatsoever. 
\begin{lemma}
\label{lemma:ind3quasiiso}
Let $(M,d_{M})$ be a dg module over $k$, considered as an $A_{1}$-algebra, and provided with a compatible bigrading of $s$-th type, for some $s \in \NN_{0}$. 
Let $M_{\hbar}$ be a formal bigraded deformation of $(M,d_{M})$, which is an $A_{1}$-algebra over $k[\hbar]$ as well. 
Consider the cohomology $H^{\bullet}(M)$ of $(M,d_{M})$, and take $f : H^{\bullet}(M) \rightarrow M$ the composition of a section of the canonical projection $\mathrm{Ker}(d_{M}) \rightarrow H^{\bullet}(M)$, which we suppose it exists,  
with the canonical inclusion $\Ker(d_{M}) \rightarrow M$. 
Suppose moreover we have $H^{\bullet}(M)_{\hbar}$ a formal bigraded deformation of $H^{\bullet}(M)$, which is just an $A_{1}$-algebra over $k[\hbar]$,  
together with a morphism $\tilde{f} : H^{\bullet}(M)_{\hbar} \rightarrow M_{\hbar}$ of complexes over $k[\hbar]$ such that 
\[     
\xymatrix
{
H^{\bullet}(M)_{\hbar} 
\ar[r]^(.6){\tilde{f}}
\ar[d]
&
M_{\hbar}
\ar[d]
\\
H^{\bullet}(M)
\ar[r]^(.6){f}
&
M
}
\]
commutes. 
Then $\tilde{f}$ is a quasi-isomorphism (of complexes over $k[\hbar]$). 
\end{lemma}
\begin{proof}
Before we begin the proof let us set some notation it will be useful. 
Let us write $\tilde{f}|_{M} = \sum_{i \in \NN_{0}} f^{i} \hbar^{i}$. 
Accordingly, if we denote by $\tilde{d}$ and $\bar{d}$ the differentials of $M_{\hbar}$ and $H^{\bullet}(M)_{\hbar}$, respectively, we have the expressions 
$\tilde{d}|_{M} = \sum_{i \in \NN_{0}} d^{i} \hbar^{i}$ and $\bar{d}|_{M} = \sum_{i \in \NN_{0}} \bar{d}^{i} \hbar^{i}$, where we recall that $\bar{d}^{0} = 0$. 
For an element $\tilde{m} \in M_{\hbar}$ of the form $\sum_{i \in \NN_{0}} m_{i} \hbar^{i}$, and $N \in \NN_{0}$, we shall denote by $\tilde{m}_{\leq N}$ the finite sum 
$\sum_{i =0}^{N} m_{i} \hbar^{i}$ in the quotient $M_{\hbar} \otimes_{k[\hbar]} k[\hbar]/(\hbar^{N+1})$. 
Analogously, we shall write $\tilde{f}^{\leq N}$ the finite sum $\sum_{i = 0}^{N} f^{i} \hbar^{i}$, which is a morphism of complexes from 
$H^{\bullet}(M)_{\hbar} \otimes_{k[\hbar]} k[\hbar]/(\hbar^{N+1})$ to $M_{\hbar} \otimes_{k[\hbar]} k[\hbar]/(\hbar^{N+1})$, and the same for $\tilde{d}^{\leq N}$ and $\bar{d}^{\leq N}$, 
as differentials of $M_{\hbar} \otimes_{k[\hbar]} k[\hbar]/(\hbar^{N+1})$ and $H^{\bullet}(M)_{\hbar} \otimes_{k[\hbar]} k[\hbar]/(\hbar^{N+1})$, respectively. 
As before, we may denote the two latter deformations by $M_{\hbar^{(N+1)}}$ and $H^{\bullet}(M)_{\hbar^{(N+1)}}$, respectively.  

In order to prove the lemma we shall prove that for $N \in \NN$, the map $\tilde{f}^{\leq N}$ induces a quasi-isomorphism between the corresponding bigraded deformations. 
We will proceed by induction. 
Note that it is true for $N = 0$, by the construction of $f$. 
We shall now assume that the statement holds for $0, \dots, N-1$, and we shall prove it for $N$, 
\textit{i.e.} we shall show that $\tilde{f}^{\leq N}$ induces a quasi-isomorphism from $H^{\bullet}(M)_{\hbar^{(N+1)}}$ to $M_{\hbar^{(N+1)}}$. 
It is easy to see that the statement implies the lemma. 

Let us first prove that the cohomology map induced by $\tilde{f}^{\leq N}$ is injective. 
We have to prove that given $h = \sum_{j=0}^{N} h_{j} \hbar^{j} \in H^{\bullet}(M)_{\hbar^{(N+1)}}$ such that $\bar{d}^{\leq N}(h) = 0$ and $\tilde{f}^{\leq N} (h) = \tilde{d}^{\leq N}(\alpha)$, 
for $\alpha = \sum_{j=0}^{N} \alpha_{j} \hbar^{j} \in M_{\hbar^{(N+1)}}$, then there exists $\beta = \sum_{j=0}^{N-1} \beta_{j} \hbar^{j} \in H^{\bullet}(M)_{\hbar^{(N+1)}}$ 
such that $h = \bar{d}^{\leq N}(\beta)$. 
However, since $\bar{d}^{\leq N}(h) = 0$ implies in particular that $\bar{d}^{\leq (N-1)}(h_{\leq (N-1)}) = 0$, and $\tilde{f}^{\leq N} (h) = \tilde{d}^{\leq N}(\alpha)$ yields 
$\tilde{f}^{\leq (N-1)} (h_{\leq (N-1)}) = \tilde{d}^{\leq (N-1)}(\alpha_{\leq N-1})$, by inductive hypothesis we have that there exists $\beta' = \sum_{j=0}^{N-2} \beta_{j} \hbar^{j} \in H^{\bullet}(M)_{\hbar^{N}}$ 
such that $h_{\leq (N-1)} = \bar{d}^{\leq (N-1)}(\beta')$. 
By regarding $\beta'$ in $H^{\bullet}(M)_{\hbar^{(N+1)}}$, we see that $h ' = h - \bar{d}^{\leq N}(\beta') = h'_{N} \hbar^{N}$, 
for some $h'_{N} \in H^{\bullet}(M)$, and 
\[     \tilde{f}^{\leq N} (h') = \tilde{f}^{\leq N}\Big(h - \bar{d}^{\leq N}(\beta')\Big) = \tilde{d}^{\leq N}\Big(\alpha - \tilde{f}^{\leq N}(\beta')\Big).     \]
Hence it suffices to prove our injectivity statement for $h' = h'_{N} \hbar^{N}$, or, equivalently, to assume that $h$ is of the form $h_{N} \hbar^{N}$, for  $h_{N} \in H^{\bullet}(M)$. 
In this case, we have that $\tilde{d}^{\leq (N-1)}(\alpha_{\leq (N-1)})$ vanishes, so by the surjectivity of the cohomology map induced by $\tilde{f}^{\leq (N-1)}$ we have 
that there exists $\beta = \sum_{j=0}^{N-1} \beta_{j} \hbar^{j} \in H^{\bullet}(M)_{\hbar^{N}}$ and $\gamma = \sum_{j=0}^{N-1} \gamma_{j} \hbar^{j} \in M_{\hbar^{N}}$ 
such that $\bar{d}^{\leq (N-1)}(\beta') = 0$ and $\tilde{f}^{\leq (N-1)} (\beta) = \alpha_{\leq (N-1)} + \tilde{d}^{\leq (N-1)} (\gamma)$. 
We shall now regard $\beta$ in $H^{\bullet}(M)_{\hbar^{(N+1)}}$ and $\gamma$ in $M_{\hbar^{(N+1)}}$. 
We claim that $\bar{d}^{\leq N} (\beta) = h$. 
The vanishing of $\bar{d}^{\leq (N-1)}(\beta')$ tells us that 
\[     \bar{d}^{\leq N} (\beta) = \sum_{j=1}^{N} \bar{d}^{j}(\beta_{N-j}) \hbar^{N},     \]
so we must show that $h_{N}$ coincides with $\sum_{j=1}^{N} \bar{d}^{j}(\beta_{N-j})$. 
Seeing that $f = f^{0}$ is injective, it suffices to prove that $f^{0}(h_{N})$ and $f^{0}(\sum_{j=1}^{N} \bar{d}^{j}(\beta_{N-j}))$ coincide. 
This follows from the following computations. 
We first consider 
\begin{align*}
    f^{0}(h_{N}) &= \sum_{i=0}^{N} d^{i} (\alpha_{N-i}) = d^{0}(\alpha_{N}) + \sum_{i=1}^{N} d^{i} (\alpha_{N-i})
    \\
    &= d^{0}(\alpha_{N}) + \sum_{i=1}^{N} d^{i} \Big(\sum_{j=0}^{N-i} f^{j}(\beta_{N-i-j}) - \sum_{j=0}^{N-i} d^{j}(\gamma_{N-i-j})\Big)
    \\
    &= d^{0}(\alpha_{N}) + \sum_{l=1}^{N} \sum_{i=1}^{l} d^{i} f^{l-i}(\beta_{N-l}) - \sum_{l=1}^{N} \sum_{i=1}^{l} d^{i} d^{l-i}(\gamma_{N-l}),
\end{align*}
where we have used $\tilde{f}^{\leq N}(h) = \tilde{d}^{\leq N}(\alpha)$ in the first equality, that  $\tilde{f}^{\leq (N-1)} (\beta) = \alpha_{\leq (N-1)} + \tilde{d}^{\leq (N-1)} (\gamma)$ in the second equality, 
and we have only reindexed the terms of the last two sums in the last member. 
Using that $\tilde{f}^{\leq N}$ is a morphism of complexes and $\tilde{d}^{\leq N}$ is a differential, the last member of the previous collection of equations coincides with 
\begin{align*}
   d^{0}&(\alpha_{N}) - \sum_{l=1}^{N} d^{0} f^{l}(\beta_{N-l}) + \sum_{l=1}^{N} \sum_{i=0}^{l-1} f^{i} \bar{d}^{l-i}(\beta_{N-l}) + \sum_{l=1}^{N} d^{0} d^{l}(\gamma_{N-l}) 
   \\
   &= f^{0}\Big(\sum_{l=1}^{N} \bar{d}^{l}(\beta_{N-l})\Big) + d^{0}\Big(\alpha_{N} - \sum_{l=1}^{N} \big(f^{l}(\beta_{N-l}) - d^{l}(\gamma_{N-l})\big)\Big) 
   \\
   &\phantom{= f^{0}(\sum_{l=1}^{N} \bar{d}^{l}(\beta_{N-l}))} + \sum_{l=1}^{N-1} f^{l} \Big(\sum_{j=1}^{N-l} \bar{d}^{j}(\beta_{N-j-l})\Big)
   \\
   &= f^{0}\Big(\sum_{l=1}^{N} \bar{d}^{l}(\beta_{N-l})\Big) + d^{0}\Big(\alpha_{N} - \sum_{l=1}^{N} \big(f^{l}(\beta_{N-l}) - d^{l}(\gamma_{N-l})\big)\Big),
\end{align*}
where the first and last terms of the second member are obtained from splitting and reindexing the third term of the first member, and we have used in the last equality 
that the last term of the second member vanishes, since, for each $l = 1, \dots, N-1$, $\sum_{j=1}^{N-l} \bar{d}^{j}(\beta_{N-j-l}) = 0$, for this latter identity is tantamount to 
$\bar{d}^{\leq (N-1)}(\beta) = 0$. 
We have that 
\[     f^{0}\Big(h_{N} - \sum_{l=1}^{N} \bar{d}^{l}(\beta_{N-l})\Big) = d^{0}\Big(\alpha_{N} - \sum_{l=1}^{N} \big(f^{l}(\beta_{N-l}) - d^{l}(\gamma_{N-l})\big)\Big).     \]
Taking into account that $f^{0}$ induces a quasi-isomorphism from $H^{\bullet}(M)$ to $M$, but the differential of $H^{\bullet}(M)$ is zero, we get that $f^{0}(h_{N} - \sum_{l=1}^{N} \bar{d}^{l}(\beta_{N-l}))$ vanishes, 
which by the injectivity of $f^{0}$ implies that $h_{N} = \sum_{l=1}^{N} \bar{d}^{l}(\beta_{N-l})$, which was to be shown. 

Let us now prove that  the cohomology map induced by $\tilde{f}^{\leq N}$ is surjective. 
Let $m = \sum_{j=0}^{N} m_{j} \hbar^{j} \in M_{\hbar^{(N+1)}}$ such that $\tilde{d}^{\leq N}(m) = 0$. 
We want to show that there exist $h = \sum_{j=0}^{N} h_{j} \hbar^{j} \in H^{\bullet}(M)_{\hbar^{(N+1)}}$ and $n = \sum_{j=0}^{N} n_{j} \hbar^{j} \in M_{\hbar^{(N+1)}}$ such that 
$\bar{d}^{\leq N}(h) = 0$ and $\tilde{f}^{\leq N}(h) = m + \tilde{d}^{\leq N}(n)$. 
Note that, from $\tilde{d}^{\leq N}(m) = 0$ we have in particular that $\tilde{d}^{\leq (N-1)}(m_{\leq (N-1)}) = 0$. 
Hence, by the inductive hypothesis there exist $h' = \sum_{j=0}^{N-1} h_{j} \hbar^{j} \in H^{\bullet}(M)_{\hbar^{N}}$ and $n' = \sum_{j=0}^{N-1} n_{j} \hbar^{j} \in M_{\hbar^{N}}$ such that 
$\bar{d}^{\leq (N-1)}(h') = 0$ and $\tilde{f}^{\leq (N-1)}(h') = m_{\leq (N-1)} + \tilde{d}^{\leq (N-1)}(n')$. 
We shall regard $h'$ as an element of  $H^{\bullet}(M)_{\hbar^{(N+1)}}$, and $n'$ as an element in $M_{\hbar^{(N+1)}}$. 
In this case, we see that $\bar{d}^{\leq N}(h') = (\sum_{j=1}^{N} \bar{d}^{j}(h_{N-j})) \hbar^{N}$. 
We shall first show that it vanishes. 
In order to do so, we consider the identity $\tilde{f}^{\leq (N-1)} (h') = m_{\leq (N-1)} + \tilde{d}^{\leq (N-1)}(n')$, which can be equivalently rewritten as a collection of identities 
\begin{equation}
\label{eq:ecj}
   \sum_{i=0}^{j} f^{i}(h_{j-i}) = m_{j} + \sum_{i=0}^{j} d^{i}(n_{j-i}),
\end{equation}
for $j = 0, \dots, N-1$. 
We apply $d^{N-j}$ to the $j$-th of the preceding equations and we add up to get  
\[     \sum_{j=0}^{N-1} \sum_{i=0}^{j} d^{N-j} f^{i}(h_{j-i}) = \sum_{j=0}^{N-1} d^{N-j}(m_{j}) + \sum_{j=0}^{N-1} \sum_{i=0}^{j} d^{N-j} d^{i} (n_{j-i}).     \]
By reindexing the sum in the first member and the last sum in the second one we get 
\[     \sum_{l=1}^{N} \sum_{i=1}^{l} d^{i} f^{l-i}(h_{N-l}) = \sum_{j=0}^{N-1} d^{N-j}(m_{j}) + \sum_{l=1}^{N} \sum_{i=1}^{l} d^{i} d^{l-i} (n_{N-l}).     \]
Using that $\tilde{f}^{\leq N}$ is a morphism of complexes and $\tilde{d}^{\leq N}$ is a differential, we rewrite the previous identity as 
\[     - \sum_{l=1}^{N} d^{0} f^{l} (h_{N-l}) + \sum_{l=1}^{N} \sum_{i=0}^{l-1} f^{i} \bar{d}^{l-i}(h_{N-l}) = \sum_{j=0}^{N-1} d^{N-j}(m_{j}) - \sum_{l=1}^{N} d^{0} d^{l} (n_{N-l}).     \]
By splitting and reindexing the second summand of the first member, we obtain the equivalent identity
\begin{multline*}
    - d^{0}\Big(\sum_{l=1}^{N} f^{l} (h_{N-l})\Big) + f^{0}\Big(\sum_{j=1}^{N} \bar{d}^{j}(h_{N-j})\Big) + \sum_{l=1}^{N-1} f^{l} \Big(\sum_{j=1}^{N-l} \bar{d}^{j}(h_{N-j-l})\Big) 
    \\
    = \sum_{j=0}^{N-1} d^{N-j}(m_{j}) - \sum_{l=1}^{N} d^{0} d^{l} (n_{N-l}).     
\end{multline*}
Since, for each $l = 1, \dots, N-1$, $\sum_{j=1}^{N-l} \bar{d}^{j}(h_{N-j-l}) = 0$, for this latter identity is tantamount to $\bar{d}^{\leq (N-1)}(h) = 0$, we get that  
\[     - d^{0}\Big(\sum_{l=1}^{N} f^{l} (h_{N-l})\Big) + f^{0}\Big(\sum_{j=1}^{N} \bar{d}^{j}(h_{N-j})\Big) = \sum_{j=0}^{N-1} d^{N-j}(m_{j}) - \sum_{l=1}^{N} d^{0} d^{l} (n_{N-l}).     \] 
Using that $\tilde{d}^{\leq N}(m)$ vanishes we conclude that 
\[      f^{0}\Big(\sum_{j=1}^{N} \bar{d}^{j}(h_{N-j})\Big) = d^{0}\Big(\sum_{l=1}^{N} f^{l} (h_{N-l}) - m_{N} - \sum_{l=1}^{N} d^{l} (n_{N-l})\Big).     \] 
As $f^{0}$ induces a quasi-isomorphism from $H^{\bullet}(M)$ to $M$, but the differential of $H^{\bullet}(M)$ is zero, we get that $f^{0}(\sum_{j=1}^{N} \bar{d}^{j}(h_{N-j}))$ vanishes, 
which by the injectivity of $f^{0}$ implies that $\sum_{j=1}^{N} \bar{d}^{j}(h_{N-j})$ also vanishes, which was to be shown. 

We will finally prove the surjectivity statement. 
We need to show that there exist $h_{N} \in H^{\bullet}(M)$ and $n_{N} \in M$ such that, for $h = h' + h_{N} \hbar^{N}$ and $n = n' + n_{N} \hbar^{N}$ elements in $H^{\bullet}(M)_{\hbar^{(N+1)}}$ and $M_{\hbar^{(N+1)}}$, 
respectively, $\bar{d}^{\leq N}(h) = 0$ and $\tilde{f}^{\leq N}(h) = m + \tilde{d}^{\leq N}(n)$. 
In view of the fact that $\bar{d}^{0}$ is zero, we see that $\bar{d}^{\leq N}(h) = 0$ holds no matter what value of $h_{N}$ we choose. 
We need only to deal with the second equation. 
Moreover, as $\tilde{f}^{\leq N-1)}(h') = m_{\leq (N-1)} + \tilde{d}^{\leq (N-1)}(n')$, the identity $\tilde{f}^{\leq N}(h) = m + \tilde{d}^{\leq N}(n)$ is tantamount to 
\[     f^{0}(h_{N}) + \sum_{i=1}^{N} f^{i}(h_{N-i}) = m_{N} + \sum_{i=1}^{N} d^{i}(n_{N-i}) + d^{0}(n_{N}).     \]
Hence, it suffices to prove that 
\[     m_{N} + \sum_{i=1}^{N} d^{i}(n_{N-i}) - \sum_{i=1}^{N} f^{i}(h_{N-i})      \]
belongs to the kernel of $d^{0}$, for in that case it can uniquely written as a sum of an element in the image of $f^{0}$ and a coboundary, \textit{i.e.} elements $f^{0}(h_{N})$ and $d^{0}(-n_{N})$, 
which in our case shows the existence of the required $h_{N}$ and $n_{N}$. 
We shall thus compute 
\begin{align*}
   d^{0}&\Big(m_{N} + \sum_{i=1}^{N} d^{i}(n_{N-i}) - \sum_{i=1}^{N} f^{i}(h_{N-i})\Big) 
  \\ 
   &= - \sum_{i=1}^{N} d^{i}(m_{N-i}) - \sum_{i=1}^{N} \sum_{j=1}^{i} d^{j} d^{i-j}(n_{N-i}) + \sum_{i=1}^{N} \sum_{j=1}^{i} d^{j} f^{i-j} (h_{N-i}) 
   \\
   &\phantom{=} - \sum_{i=1}^{N} \sum_{j=0}^{i-1} f^{j} \bar{d}^{i-j} (h_{N-i})
   \\
   &= - \sum_{i=1}^{N} d^{i}\Big(m_{N-i} + \sum_{j=0}^{N-i} d^{j}(n_{N-i-j}) - \sum_{j=0}^{N-i} f^{j} (h_{N-i-j})\Big)  
   \\
   &\phantom{=} - \sum_{j=0}^{N-1} f^{j}\Big(\sum_{i=1}^{N-j}  \bar{d}^{i} (h_{N-i-j})\Big),
\end{align*}
where we have used that $\tilde{d}^{\leq N}(m) = 0$, $\tilde{d}^{\leq N} \circ \tilde{d}^{\leq N} = 0$, and that $\tilde{f}^{\leq N}$ is a morphism of complexes in the first identity, 
and we have reindexed the sums of the second member in the last identity. 
By the inductive hypothesis we have that $\tilde{f}^{\leq (N-1)}(h') = m_{\leq (N-1)} + \tilde{d}^{\leq (N-1)}(n')$, which is tantamount to the vanishing of the arguments of $d^{i}$ (for $i = 1, \dots, N$)
appearing in the first sum of the last chain of identities. 
Furthermore, the vanishing of $\tilde{d}^{\leq (N-1)}(h')$ is equivalent to the vanishing of the arguments of $f^{j}$ (for $j = 0, \dots, N-1$) appearing in the last sum of the previous chain of identities. 
This proves that $m_{N} + \sum_{i=1}^{N} d^{i}(n_{N-i}) - \sum_{i=1}^{N} f^{i}(h_{N-i})$ belongs to the kernel of $d^{0}$, which in turn proves the surjectivity claim. 
The lemma is proved.  
\end{proof}

\begin{remark}
\label{remark:quasi-isodef}
Even though the previous lemma is formulated in our specific setting of formal bigraded deformation, the proof goes completely \textit{verbatim} if the deformations of the complexes are formal (but not necessarily graded). 
\end{remark}

\begin{remark}
\label{remark:dife}
Note that our way to proceed is a priori more general than the one used by Lapin in \cite{La02a}, \cite{La02b}, and his other articles of the kind, where he has used homological perturbation theory, 
since we showed that all of the obstructions to build a deformation of the cohomology $H^{\bullet}(A)$ together with a morphism from this deformation to the deformation of $A$ can be resolved, 
and for any way a resolving these obstructions we obtain in fact a quasi-isomorphism between the formal bigraded deformation of $H^{\bullet}(A)$ and of $A$.  
\end{remark}

\section{\texorpdfstring{Filtered deformations of $A_{\infty}$-algebras}{Filtered deformations of A-infinity-algebras}}
\label{section:fildefo}

We assume for the rest of the article that $k$ is a field, though this hypothesis is not strictly necessary for the main definitions of this section. 
A \emph{(strict and decreasing) filtration} on an $A_{\infty}$-algebra $(A, m_{\bullet})$ is a decreasing filtration $\{ F^{p}A \}_{p \in \ZZ}$ of the underlying graded module of $A$ satisfying 
the compatibility condition 
\begin{equation}
\label{eq:compfiltainf}
     m_{n}(F^{p_{1}}A \otimes \dots \otimes F^{p_{n}}A) \subseteq F^{p_{1}+\dots+p_{n}}A,     
\end{equation}
for all $n \in \NN$, and $p_{1}, \dots, p_{n} \in \ZZ$. 
If $A$ has a unit we further assume that $1_{A} \in F^{0}A$. 
We shall assume that the filtration is exhaustive and graded complete (so Hausdorff), but not necessarily complete. 
Note that, given a filtered $A_{\infty}$-algebra $A$, the compatibility condition \eqref{eq:compfiltainf} tells us that the union $\cup_{p \in \ZZ} F^{p}A$ is also an $A_{\infty}$-algebra 
(such that the inclusion in $A$ is a strict morphism of $A_{\infty}$-algebras). 
Similarly, the structure maps $\{ m_{n} \}_{n \in \NN}$ of $A$ naturally induce an $A_{\infty}$-algebra $\{ \hat{m}_{n}^{\mathrm{gr}} \}_{n \in \NN}$ on the graded completion construction of the underlying filtered graded module of $A$.    
Indeed, for any $p_{1}, \dots, p_{n}$ consider the maps 
\[     (\hat{A}^{\mathrm{gr}})^{\otimes n} \rightarrow (\bigotimes_{j=1}^{n} A/F^{p_{j}}A) \rightarrow A/F^{p_{1}+\dots+p_{n}}A     \] 
given by the composition of the tensor product of canonical projections and the morphism induced by $m_{n}$. 
It is easy to show any pair of these maps for $p_{1}, \dots, p_{n}$ and $p'_{1}, \dots, p'_{n}$ satisfying that $p_{1} + \dots + p_{n} = p'_{1} + \dots + p'_{n} = p$ coincide, 
so we may consider the collection of morphisms of graded modules  
\[     q_{p} :  (\hat{A}^{\mathrm{gr}})^{\otimes n} \rightarrow A/F^{p}A     \]
indexed by $p$. 
This collection forms a system, \textit{i.e.} given any pair of integers $p < p'$, the composition of $q_{p'}$ with the canonical projection $A/F^{p'}A \rightarrow A/F^{p}A$ coincides with $q_{p}$.  
It thus induces a morphism of graded modules $(\hat{A}^{\mathrm{gr}})^{\otimes n} \rightarrow \hat{A}^{\mathrm{gr}}$, which we denoted by $\hat{m}_{n}^{\mathrm{gr}}$. 
The Stasheff identities are immediate. 
If $A$ has a unit $1_{A}$, its image under the canonical map $A \rightarrow \hat{A}^{\mathrm{gr}}$ explained in \eqref{eq:complgr} induces a unit on $\hat{A}^{\mathrm{gr}}$. 
Note however that the completion $\hat{A}$ of the underlying filtered $k$-module of $A$ is not an $A_{\infty}$-algebra in general, because the $k$-module $\hat{A}$ need not be graded. 

Given a filtered $A_{\infty}$-algebra $(A, m_{\bullet})$, the associated bigraded module $\mathrm{Gr}_{F^{\bullet}A}(A)$ has a structure of $A_{\infty}$-algebra 
with structure maps $m_{n}^{\mathrm{gr}}$, for $n \in \NN$, induced by those of $A$. 
It will be referred as the \emph{associated graded $A_{\infty}$-algebra}. 
If $A$ has a unit $1_{A}$, its class (as an element of $F^{0}A/F^{1}A$) induces a unit on the associated graded $A_{\infty}$-algebra. 
This $A_{\infty}$-algebra structure is in fact compatible with the bigrading of $0$-th type. 
Moreover, to a filtered $A_{\infty}$-algebra we may associate the \emph{Rees $A_{\infty}$-algebra} $\mathrm{Re}_{F^{\bullet}A}(A)$, given by 
\begin{equation}
\label{eq:rees0}
     \bigoplus_{p \in \ZZ} (F^{p}A \otimes k.\hbar^{-p}) \subseteq A \otimes k[\hbar^{\pm}],    
\end{equation}
with structure maps 
\begin{equation}
\label{eq:rees} 
    m_{n}^{\mathrm{re}}(a_{1} \otimes \hbar^{-p_{1}}, \dots, a_{n} \otimes \hbar^{-p_{n}}) = m_{n}(a_{1}, \dots, a_{n}) \otimes \hbar^{-(p_{1} + \dots + p_{n})},     
\end{equation}
for $n \in \NN$. 
Note that the ``polynomial part'' is in some sense superfluous (and could be dropped from the definition) but it helps keeping track of the index $p$ coming from the filtration. 
If $A$ has a unit $1_{A}$, then $\mathrm{Re}_{F^{\bullet}A}(A)$ has the strict unit $1_{A} \otimes \hbar^{0}$. 

On the other hand, the decreasing property of the filtration of $A$ tells us that $\mathrm{Re}_{F^{\bullet}A}(A)$ has further the structure of bigradee $k[\hbar]$-module such that the inclusion \eqref{eq:rees0} turns it into a $k[\hbar]$-submodule of $A \otimes k[\hbar^{\pm}]$, provided with the regular action. 
Furthermore, note that if $A$ is unitary then $\mathrm{Re}_{F^{\bullet}A}(A)$ has a canonical inclusion of the bigraded algebra $k[\hbar]$ given by 
$c \hbar^{p} \mapsto 1_{A} \otimes c \hbar^{p}$, for $c \in k$.  
By definition of the structure maps \eqref{eq:rees}, it is direct to check that $\mathrm{Re}_{F^{\bullet}A}(A)$ is in fact an $A_{\infty}$-algebra over $k[\hbar]$. 
We can regard it as a formal bigraded deformation of $\mathrm{Gr}_{F^{\bullet}A}(A)$ as follows. 
For each $p \in \ZZ$ choose a $k$-submodule $X^{p}$ of $F^{p}A$ such that $F^{p}A = X^{p} \oplus F^{p+1}A$, for $k$ is a field. 
Since the filtration of $A$ is Hausdorff, the canonical morphism of $k$-modules 
\[     F^{p}A \rightarrow \prod_{q \in \ZZ_{\geq p}} X^{q}     \]
is injective. 
Furthermore, as we assumed that the filtered $A_{\infty}$-algebra $A$ is graded complete, then we have in fact an isomorphism of graded $k$-modules
\[     F^{p}A \rightarrow \prod_{q \in \ZZ_{\geq p}}\hskip -2.5mm {}^{\mathrm{gr}} \hskip 0.8mm X^{q},     \]
where we recall that the codomain denotes the product in the category of graded $k$-modules. 
Let us denote by $\pi^{p}_{q} : F^{p}A  \rightarrow X^{q}$, for $p \leq q$, the composition of (any of) the previous morphism(s) and the canonical projection. 
Also denote by $\mathrm{proj}_{p} : F^{p}A \rightarrow F^{p}A/F^{p+1}A$ the canonical projection which identifies the $k$-module given by the codomain of the latter map and $X^{p}$. 
We may thus consider the $k$-linear map 
\begin{equation}
\label{eq:morreesdef}
\begin{split}
   \mathrm{Re}_{F^{\bullet}A}(A) &\rightarrow \mathrm{Gr}_{F^{\bullet}A}(A)[[\hbar]]^{\mathrm{bgr}}    
   \\
   \sum_{(p \in \ZZ)} a_{p} \otimes \hbar^{-p} &\mapsto \sum_{(p \in \ZZ)} \sum_{q \in \ZZ_{\geq p}} \mathrm{proj}_{q}(\pi^{p}_{q}(a_{p})) \otimes \hbar^{q-p},
\end{split}
\end{equation}
where we recall that the parentheses for the index $p$ means that the corresponding sums are finitely supported.  
The previous map is clearly well-defined (see \eqref{eq:defo}) and homogeneous $k[\hbar]$-linear of bidegree $(0,0)$. 
By the graded completion assumption on the filtration of $A$ and \eqref{eq:defo} we see that in fact it is bijective, so an isomorphism of bigraded $k[\hbar]$-modules. 

\begin{remark}
\label{remark:impa_inf}
Let $A$ be a filtered (graded complete) $A_{\infty}$-algebra. 
By exactly the same arguments as those used for proving that the collection of multiplicative spectral sequences associated to the complexes \eqref{eq:seqfil} defined from a filtered dg algebra is compatible 
(see the last paragraph of Subsection \ref{subsection:spe}), the exactly same statement holds if we consider a filtered (graded complete) $A_{\infty}$-algebra. 
We remark that the multiplicative structure of the spectral sequence is induced by the product of $A$, after taking cohomology on the complexes \eqref{eq:seqfil}, and considering the corresponding exact couples, 
as explained in Subsection \ref{subsection:spe}. 
\end{remark}

Let $(B_{\hbar},\tilde{m}_{\bullet})$ be a formal bigraded deformation of an (resp., a unitary) $A_{\infty}$-algebra $(B,m_{\bullet})$ provided with a compatible bigrading of $0$-th type. 
Given $p \in \ZZ$, set 
\begin{equation*}
\label{eq:filhatbgr}
     F^{p}B = \bigoplus_{p' \geq p} \bigoplus_{q \in \ZZ} B^{p',q}.     
\end{equation*}
It defines a decreasing filtration on $B$, which is exhaustive and Hausdorff. 
Furthermore, if we regard $B$ as graded $k$-module with the total degree, the previous filtration is of graded $k$-modules. 
We may thus consider the graded completion $\hat{B}^{\mathrm{gr}}$ of the graded module $B$ with the previous filtration. 
By definition we see that the underlying graded $k$-module of $\hat{B}^{\mathrm{gr}}$ is isomorphic to the submodule of 
\[     \bigoplus_{n \in \ZZ} \Big(\prod_{p \in \ZZ} B^{p,n-p} \Big)     \]
formed by the elements 
\[     \hat{b} = \sum_{(n \in \ZZ)} \sum_{p \in \ZZ} b_{p,n-p},     \]
where $b_{p,n-p} \in B^{p,n-p}$, such that for each $n \in \ZZ$ there exists $P_{n} \in \ZZ$ satisfying that $b_{p,n-p}$ vanishes for $p < P_{n}$. 
Moreover, for each $n \in \ZZ$, we define the map 
\[     \operatorname{P}_{n} : \bigoplus_{n \in \ZZ} \Big(\prod_{p \in \ZZ} B^{p,n-p} \Big) \rightarrow \{ - \infty \} \cup \ZZ \cup \{ + \infty \},     \]
satisfying the following conditions. 
For all $p \in \ZZ$, we have that, if $p < \operatorname{P}_{n}(\hat{b})$ (using the convention $-\infty < p < + \infty$, for $p \in \ZZ$) then $b_{p,n-p} = 0$, and that, if $p = \operatorname{P}_{n}(\hat{b})$ 
(this condition only holds in case the latter is an integer) then $b_{p,n-p} \neq 0$. 
Note that the condition for $\hat{b}$ of being an element of $\hat{B}^{\mathrm{gr}}$ can be rephrased as stating that $\mathrm{inf}_{n \in \ZZ}(\operatorname{P}_{n}(\hat{b})) > - \infty$. 
We shall denote the previous infimum by $\operatorname{P}(\hat{b})$, which defines a map with the same domain and codomain as each $\operatorname{P}_{n}$. 
As explained in \eqref{eq:canfilgr}, $\hat{B}^{\mathrm{gr}}$ has a canonical filtration, denoted by $\{ F^{p}\hat{B}^{\mathrm{gr}} \}_{p \in \ZZ}$.  
We have an identification (as bigraded $k[\hbar]$-modules) between $B_{\hbar}$ and 
\[     \bigoplus_{p, q \in \ZZ} \Big(\prod_{r \in \NN_{0}} B^{p+r,q-r} \otimes k.\hbar^{r} \Big).     \]
Let us consider the map $\iota : \hat{B}^{\mathrm{gr}} \rightarrow B_{\hbar}$ defined as 
\[     \hat{b} = \sum_{(n \in \ZZ)} \sum_{p \in \ZZ} b_{p,n-p} \mapsto \sum_{(n \in \ZZ)} \sum_{r \in \NN_{0}} b_{\operatorname{P}_{n}(\hat{b})+r,n-\operatorname{P}_{n}(\hat{b})-r} \hbar^{r}.     \] 
Note that $\iota$ is in fact a morphism of graded $k$-modules from $\hat{B}^{\mathrm{gr}}$ to the graded module $B_{\hbar}$ with the total degree. 

On the other hand, since $\hbar$ has total degree $0$, the quotient $B_{\hbar}/(\hbar - 1)$, which we shall denote by $\bar{B}$, is a graded $k$-module. 
Note that the kernel of the canonical projection $B_{\hbar} \rightarrow \bar{B}$ is given by $B_{\hbar} . (\hbar - 1)$, formed by the set of elements given by $b. (\hbar - 1)$, for $b \in B_{\hbar}$, 
and that $\bar{B}$ is an (resp., a unitary) $A_{\infty}$-algebra over $k$ with structure morphisms induce by those of $B_{\hbar}$. 
Indeed, given $n \in \NN$, the map $\tilde{m}_{n}$ induces a map $\bar{m}_{n} : \bar{B}^{\otimes n} \rightarrow \bar{B}$, if given $b_{1}, \dots, b_{n} \in B_{\hbar}$ such that there exists 
$i \in \NN_{\leq n}$ with $b_{i} = b_{i}' (\hbar -1)$, then  $\tilde{m}_{n}(b_{1}, \dots, b_{n}) \in B_{\hbar} . (\hbar - 1)$. 
This last statement follows directly from the fact that $\tilde{m}_{n}$ is $k[\hbar]$-linear. 

The composition of the map $\iota : \hat{B}^{\mathrm{gr}} \rightarrow B_{\hbar}$ together with the canonical projection $B_{\hbar} \rightarrow \bar{B}$ define a homogeneous morphism $\iota'$ of graded $k$-modules of degree zero. 
By taking into account the specific description of elements on both spaces we see that this composition map is in fact an isomorphism of graded $k$-modules. 
We see that  $\bar{B}$ is provided with a decreasing filtration of graded $k$-modules given by the image of $\{ F^{p}\hat{B}^{\mathrm{gr}} \}_{p \in \ZZ}$ 
under $\iota'$, which we denote by $\{ F^{p}\bar{B} \}_{p \in \ZZ}$.  
It is exhaustive and complete, for the same occurs to $\hat{B}^{\mathrm{gr}}$. 
Furthermore, the filtration $\{ F^{\bullet}\bar{B} \}_{\bullet \in \ZZ}$ is respected by the (resp., unitary) $A_{\infty}$-algebra structure, so $\bar{B}$ becomes a filtered (resp., filtered unitary) $A_{\infty}$-algebra. 
In order to prove this last claim, take nonvanishing homogeneous elements $\bar{b}^{1} \in F^{p_{1}}\bar{B}$, \dots, $\bar{b}^{i} \in F^{p_{i}}\bar{B}$ of degree $n_{1}, \dots, n_{i}$, resp. 
In view of the fact that $\iota'$ is a bijection, we have the corresponding elements $\hat{b}^{1} \in F^{p_{1}}\hat{B}^{\mathrm{gr}}$, \dots, $\hat{b}^{i} \in F^{p_{i}}\hat{B}^{\mathrm{gr}}$ given by taking the inverse image 
under $\iota'$ of the collection considered in the previous statement. 
Note that $\operatorname{P}(\hat{b}^{j}) = \operatorname{P}_{n_{j}}(\hat{b}^{j}) \geq p_{i}$, for all $j = 1, \dots, i$. 
We recall that $\bar{m}_{i}(\bar{b}^{1}, \dots, \bar{b}^{i})$ is the image of $\tilde{m}_{i}(\iota(\hat{b}^{1}), \dots, \iota(\hat{b}^{i}))$ under the canonical projection $B_{\hbar} \rightarrow \bar{B}$. 
By regarding the bidegrees of the elements $\iota(\hat{b}^{1}), \dots, \iota(\hat{b}^{i})$ we see that $\tilde{m}_{i}(\iota(\hat{b}^{1}), \dots, \iota(\hat{b}^{i}))$ is an element of bidegree 
$(\sum_{j=1}^{i} \operatorname{P}(\hat{b}^{j}), \sum_{j=1}^{i} (n_{j} - \operatorname{P}(\hat{b}^{j})))$ in $B_{\hbar}$, so it belongs to the image under $\iota$ of 
$F^{\sum_{j=1}^{i} \operatorname{P}(\hat{b}^{j})}\hat{B}^{\mathrm{gr}} \subseteq F^{\sum_{j=1}^{i} p_{j}}\hat{B}^{\mathrm{gr}}$. 
Hence, we get that $\bar{m}_{i}(\bar{b}^{1}, \dots, \bar{b}^{i})$ lies in $F^{\sum_{j=1}^{i} p_{j}}\bar{B}$, as was to be shown. 

We have the following result.  
\begin{proposition}
\label{proposition:fil-defo}
Let $(B_{\hbar},\tilde{m}_{\bullet})$ be a formal bigraded deformation of an (resp., a unitary) $A_{\infty}$-algebra $(B,m_{\bullet})$ provided with a compatible bigrading of $0$-th type. 
Consider the filtered (resp., filtered unitary) $A_{\infty}$-algebra $\bar{B}$ for the filtration $\{ F^{\bullet}\bar{B} \}_{\bullet \in \ZZ}$ defined at the previous paragraph. 
Then the (resp., unitary) $A_{\infty}$-algebras over $k[\hbar]$ provided with compatible bigradings of $0$-th type given by $B_{\hbar}$ and the Rees $A_{\infty}$-algebra $\mathrm{Re}_{F^{\bullet}\bar{B}}(\bar{B})$ 
are strictly isomorphic. 
Moreover, the formal bigraded deformations over $B$ they define are strictly isomorphic, such that the induced map on $B$ is in fact the identity. 
\end{proposition}
\begin{proof}
Taking into account that $B_{\hbar}$ is a formal bigraded deformation of $B$, we shall write as usual $\tilde{m}_{i}|_{B} = \sum_{j \in \NN_{0}} m_{i}^{j} \hbar^{j}$. 
We now note that we shall utilize the identification as filtered graded $k$-modules between $\hat{B}^{\mathrm{gr}}$ and $\bar{B}$ given by $\iota'$. 
In particular we shall write an element $\bar{b} \in \bar{B}$ in the form 
\[     \bar{b} = \sum_{(n \in \ZZ)} \sum_{p \in \ZZ} b_{p,n-p},     \]
where $b_{p,n-p} \in B^{p,n-p}$, such that for each $n \in \ZZ$ there exists $P_{n} \in \ZZ$ satisfying that $b_{p,n-p}$ vanishes for $p < P_{n}$, 
and may also apply the maps $\operatorname{P}_{n}$ and $\operatorname{P}$ to elements of $\bar{B}$. 
By this identification, we shall that identify $F^{\bullet}\bar{B}$ with the respective filtration of $\hat{B}^{\mathrm{gr}}$. 
So we may identify the $(p,q)$-th homogeneous component of the bigraded $k[\hbar]$-module $\mathrm{Re}_{F^{\bullet}\bar{B}}(\bar{B})$ with $F^{p}(\hat{B}^{\mathrm{gr}})^{p+q}$. 
Under this identification, the structure map $m_{i}^{\mathrm{re}}$ of $\mathrm{Re}_{F^{\bullet}\bar{B}}(\bar{B})$ are given by composing 
\begin{equation}
\label{eq:a_inf}
     m_{i}^{\mathrm{re}}|_{\hat{B}^{\mathrm{gr}}} = \tilde{m}_{i} \circ \iota^{\otimes i}     
\end{equation}
together with  canonical projection $B_{\hbar} \rightarrow \bar{B}$. 

We define an explicit map $\chi : \mathrm{Re}_{F^{\bullet}\bar{B}}(\bar{B}) \rightarrow B_{\hbar}$ as follows. 
Let $\bar{b} \in F^{p}\bar{B}$ of degree $n$, which we regard by the identifications explained at the previous paragraph as  
\[     \sum_{p' \in \ZZ_{\geq p}} b_{p',n-p'}.     \]
The mapping $\chi$ is defined as the $k$-linear extension of the map 
\[     \bar{b} \otimes \hbar^{-p} \mapsto \sum_{r \in \NN_{0}} b_{p+r,n-p-r} \hbar^{r}.     \]
Note that it is a morphism of bigraded $k[\hbar]$-modules, and that the restriction of $\chi$ to the homogeneous component of bidegree $(p,n-p)$ of 
$\mathrm{Re}_{F^{\bullet}\bar{B}}(\bar{B})$ coincides with the $k[\hbar]$-linear map induced by the restriction to $F^{p}(\hat{B}^{\mathrm{gr}})^{n}$ of $\iota$. 
It also coincides with the map \eqref{eq:morreesdef}, after the obvious identification of (resp., unitary) $A_{\infty}$-algebras over $k$ provided with compatible bigradings between $B$ and $\mathrm{Gr}_{F^{\bullet}\hat{B}^{\mathrm{gr}}}(\hat{B}^{\mathrm{gr}})$. 
The map is clearly seen to be an isomorphism of graded $k[\hbar]$-modules, for its inverse is the $k$-linear extension of  
\[      \sum_{r \in \NN_{0}} b_{p+r,q-r} \hbar^{r} \mapsto \Big(\sum_{r \in \NN_{0}} b_{p+r,q-r}\Big) \otimes \hbar^{-p}.     \]

We claim that $\chi$ is a strict morphism of (resp., unitary) $A_{\infty}$-algebras over $k[\hbar]$, \textit{i.e.} $\tilde{m}_{i} \circ \chi^{\otimes_{k[\hbar]} i} = \chi \circ m_{i}^{\mathrm{re}}$, for all $i \in \NN$. 
This follows immediately from identity \eqref{eq:a_inf}. 
Since we regard the Rees $A_{\infty}$-algebra $\mathrm{Re}_{F^{\bullet}\bar{B}}(\bar{B})$ as a formal bigrading deformation of $B$ 
by means of the map \eqref{eq:morreesdef}, the comments at the previous paragraph tells us that $\chi : \mathrm{Re}_{F^{\bullet}\bar{B}}(\bar{B}) \rightarrow B_{\hbar}$ 
gives in fact a strict isomorphism of formal bigraded deformations inducing the identity on $B$.  
The proposition is thus proved. 
\end{proof}

\section{\texorpdfstring{Formal deformations of $A_{\infty}$-algebras and spectral se\-quences}{Formal deformations of A-infinity-algebras and spectral se\-quences}}
\label{section:final}    

We shall consider the following two definitions, which were studied by Lapin in \cite{La02a}, but were already known and had appeared in the literature (in some particular cases). 
We moreover adapt them in our context of bigraded objects. 
Since these constructions are hard to handle, we provide an equivalent but much more manageable construction, from which we derive several interesting properties. 

Let $(A,m_{\bullet})$ be a minimal (resp., minimal unitary) $A_{\infty}$-algebra provided with a compatible bigrading of $s$-th type, for some $s \in \NN_{0}$, and let $(A_{\hbar},\tilde{m}_{\bullet})$ be a formal bigraded deformation of $A$. 
The minimality assumption means that $m_{1}^{0}$ vanishes. 
We first define the \emph{projected dg algebra $\operatorname{P}(A_{\hbar})$ of $A_{\hbar}$} as the (resp., unitary) dg algebra over $k$ whose underlying bigraded $k$-module is $A$, 
together with the multiplication $m_{2}$ and the differential $\tilde{m}_{1}^{1}$. 
It is a trivial verification that $\operatorname{P}(A_{\hbar})$ is a (resp., unitary) dg algebra provided with a compatible bigrading of $(s+1)$-th type. 
We may in fact regard $\operatorname{P}$ as a functor from the category of formal bigraded deformations to the category of (resp., unitary) dg algebras. 
Let $A_{\hbar}$ and $B_{\hbar}$ be formal bigraded deformations of the (resp., unitary) $A_{\infty}$-algebras $A$ and $B$ provided with compatible bigradings of $s$-th type, for some $s \in \NN_{0}$, 
and let $(f_{\bullet}, \tilde{f}_{\bullet})_{\bullet \in \NN}$ be a morphism of formal bigraded deformations from $A_{\hbar}$ to $B_{\hbar}$. 
Its image under $\operatorname{P}$ is the morphism of (resp., unitary) dg algebras from $\operatorname{P}(A_{\hbar})$ to $\operatorname{P}(B_{\hbar})$ given by $f_{1}$. 
Note that this functor sends quasi-isomorphisms to isomorphisms. 

On the other hand, define the \emph{translated deformation $\operatorname{T}(A_{\hbar})$ of $A_{\hbar}$} as the (resp., unitary) $A_{\infty}$-algebra over $k[\hbar]$ whose underlying bigraded $k[\hbar]$-module coincides with that of $A_{\hbar}$ but with  structure maps given by $\tilde{m}_{i}^{\mathrm{T}} = \tilde{m}_{i} \hbar^{i-2}$. 
Note that $\tilde{m}_{1}^{\mathrm{T}}$ is well-defined because $m_{1}^{0} = 0$. 
Furthermore, by the homogeneity of the definition on the structure maps, the Stasheff identity $\mathrm{SI}(n)$ of $\operatorname{T}(A_{\hbar})$ is equivalent to the Stasheff identity 
$\mathrm{SI}(n)$ of $A_{\hbar}$ multiplied by $\hbar^{(n-3)}$. 
If $A$ is assumed to be unitary the identity axioms are straightforward. 
By a simple calculation the underlying (resp., unitary) $A_{\infty}$-algebra over $k$ of $\operatorname{T}(A_{\hbar})$ has a compatible bigrading of $(s+1)$-th type, and it is in fact a formal bigraded deformation of 
the projected (resp., unitary) dg algebra $\operatorname{P}(A_{\hbar})$. 
We can regard $\operatorname{T}$ as a functor from the category whose objects are formal bigraded deformations of (resp., unitary) $A_{\infty}$-algebras provided 
with a morphisms of formal bigraded deformations to itself. 
Indeed, let $A_{\hbar}$ and $B_{\hbar}$ be formal bigraded deformations of the (resp., unitary) $A_{\infty}$-algebras $A$ and $B$ provided with compatible bigradings of $s$-th type, for some $s \in \NN_{0}$, 
and let $(f_{\bullet}, \tilde{f}_{\bullet})_{\bullet \in \NN}$ be a morphism of formal bigraded deformations from $A_{\hbar}$ to $B_{\hbar}$. 
The image of this morphism under $\operatorname{T}$ is given by the pair $(f_{1},\operatorname{T}(\tilde{f}))$ where we define the morphism $\operatorname{T}(\tilde{f})$ from $\operatorname{T}(A)$ 
to $\operatorname{T}(B)$, whose $i$-th component $\operatorname{T}(\tilde{f})_{i}$ is defined as $\tilde{f}_{i} \hbar^{i-1}$, for $i \in \NN$. 
The Morphism identity $\mathrm{MI}(n)$ of $\operatorname{T}(\tilde{f})_{\bullet}$ is equivalent to the Morphism identity 
$\mathrm{MI}(n)$ of $\tilde{f}_{\bullet}$ multiplied by $\hbar^{(n-2)}$. 
If $A$ is assumed to be unitary the identity axioms are straightforward. 
Notice that this functor preserves quasi-isomorphisms.   

Note now that we may combine Theorem \ref{thm:quasiisofor} with the proceeding definitions to produce a collection deformation as follows. 
Suppose $({}^{0}A,{}^{0}m_{\bullet})$ is an (resp., a unitary) $A_{\infty}$-algebra provided with a compatible bigrading of $0$-th type, and let $({}^{0}A_{\hbar},{}^{0}\tilde{m}_{\bullet})$ be a formal bigraded deformation of ${}^{0}A$. 
By the previous theorem, we may consider the quasi-isomorphic deformation given by the (resp., unitary) $A_{\infty}$-algebra $(H^{\bullet}({}^{0}A), {}^{0}\bar{m}_{\bullet})$, also provided with a compatible bigrading of 
$0$-th type, and the formal bigraded deformation deformation $(H^{\bullet}({}^{0}A)_{\hbar}, {}^{0}\tilde{\bar{m}}_{\bullet})$, which satisfies by construction that ${}^{0}\bar{m}_{1}^{0} = 0$, so we may 
consider the new formal bigraded deformation $\operatorname{T}(H^{\bullet}({}^{0}A)_{\hbar})$ of the projected (resp., unitary) dg algebra $\operatorname{P}(H^{\bullet}({}^{0}A)_{\hbar})$, which has a compatible grading of first type. 
Let us call them ${}^{1}A_{\hbar}$ and ${}^{1}A$, respectively.  
By iteration we obtain thus a collection of (resp., unitary) dg algebras $\{ {}^{r}A \}_{r \in \NN}$, such that ${}^{r}A$ is provided with a compatible bigrading of $r$-th type together with a formal bigraded deformation ${}^{r}A_{\hbar}$, for each 
$r \in \NN$, and they satisfy that $\operatorname{T}(H^{\bullet}({}^{r}A)_{\hbar}) = {}^{(r+1)}A_{\hbar}$ and $\operatorname{P}(H^{\bullet}({}^{r}A)_{\hbar}) = {}^{(r+1)}A$, 
where $H^{\bullet}({}^{r}A)_{\hbar}$ is the formal bigraded deformation of $H^{\bullet}({}^{r}A)$ given by Theorem \ref{thm:quasiisofor}. 
We note that the collection of (resp., unitary) dg algebras $\{ {}^{r}A \}_{r \in \NN}$ is in fact a multiplicative spectral sequence. 
It will be called the \emph{multiplicative spectral sequence associated to the formal bigraded deformation ${}^{0}A_{\hbar}$ of ${}^{0}A$}. 
The collection of formal bigraded deformations produced will be called the \emph{associated family of formal bigraded deformations}.  
When a multiplicative spectral sequence is constructed as before, and we regard it equipped with the formal bigraded deformation satisfying the previous procedure, 
we shall say that it is \emph{enhanced with an (resp., a unitary) $A_{\infty}$-algebra structure}, or more simply, that it has an \emph{(resp., unitary) $A_{\infty}$-enhancement}. 

\begin{remark}
\label{remark:laseq}
Suppose $({}^{0}A,{}^{0}d)$ is a complex of $k$-modules, regarded as an $A_{\infty}$-algebra with vanishing multiplications $m_{i}$, for $i \in \NN_{\geq 2}$, 
provided with a compatible bigrading of $0$-th type, and let $({}^{0}A_{\hbar},{}^{0}\tilde{m}_{\bullet})$ be a formal bigraded deformation of ${}^{0}A$. 
By dropping if necessary all the multiplications ${}^{0}\tilde{m}_{i}$, for $i \in \NN_{\geq 2}$, 
we shall consider ${}^{0}A_{\hbar}$ as a complex over $k[\hbar]$, which gives a formal bigraded deformation of the complex ${}^{0}A$. 
By forgetting if necessary at each step of the procedure explained in the previous paragraph all the involved multiplications indexed by $i \in \NN_{\geq 2}$, 
for we consider only deformations of the differentials, one obtains a spectral sequence $\{ {}^{r}A \}_{r \in \NN}$, 
which will be called the \emph{spectral sequence associated to the underlying deformation of complexes of the formal bigraded deformation ${}^{0}A_{\hbar}$ of ${}^{0}A$}. 
By using a direct transference argument, Lapin has showed in \cite{La01}, Prop. 3.2, that in fact any spectral sequence can be obtained in this manner 
(though he does not consider any bigrading, this can be incorporated \textit{verbatim} in his proof). 
By combining this together with Proposition \ref{proposition:fil-defo} we can conclude that any spectral sequence over a field $k$ can be obtained from a filtration on a complex. 
This result does not seem to have been observed so far. 
\end{remark}

\begin{example}
\label{example:example}
Suppose we have a formal bigraded deformation $(A_{\hbar}, \tilde{m}_{\bullet})$ of an $A_{\infty}$-algebra $(A,m_{\bullet})$ provided with a compatible bigrading of $s$-th type, for some $s \in \NN_{0}$. 
We want to show how the first two stages $E_{s+1}$ and $E_{s+2}$ of the spectral sequence starting at $(s+1)$ associated to the underlying deformation of complexes 
of the formal bigraded deformation $A_{\hbar}$ of $A$ look like. 
As usual, we write $\tilde{m}_{1}|_{A} = \sum_{j \in \NN_{0}} m_{1}^{j} \hbar^{j}$. 
By definition, 
\[     E_{s+1} = \frac{\Ker(m_{1}^{0})}{\mathrm{Im(m_{1}^{0})}},     \]
provided with the differential $d_{s+1}$ given by 
\[     x + \mathrm{Im}(m_{1}^{0}) \mapsto m_{1}^{1}(x) + \mathrm{Im}(m_{1}^{0}),     \]
where $x \in \Ker(m_{1}^{0})$, \textit{i.e.} $m_{1}^{0}(x) = 0$. 
Furthermore, by easy computations we have that 
\[     E_{s+2} = \frac{\Ker(m_{1}^{0}) \cap (m_{1}^{1})^{-1}(\Ker(m_{0}^{1}))}{\mathrm{Im}(m_{1}^{1}|_{\Ker(m_{1}^{0})}) + \mathrm{Im}(m_{1}^{0})}.     \]
The differential $d_{s+2}$ is given as follows. 
Let $x \in \Ker(m_{1}^{0}) \cap (m_{1}^{1})^{-1}(\Ker(m_{0}^{1}))$, which is tantamount to the fact that $m_{1}^{0}(x) = 0$ and there exists $y \in A$ such that $m_{1}^{1}(x) = m_{1}^{0}(y)$. 
Then, by straightforward computations we see that $d_{s+2}$ is given by 
\[     x + \Big(\mathrm{Im}(m_{1}^{1}|_{\Ker(m_{1}^{0})}) + \mathrm{Im}(m_{1}^{0})\Big) \mapsto (m_{1}^{2}(x) - m_{1}^{0}(y)) + \Big(\mathrm{Im}(m_{1}^{1}|_{\Ker(m_{1}^{0})}) + \mathrm{Im}(m_{1}^{0})\Big).     \]
\end{example}

We want to observe that the procedure giving an $A_{\infty}$-enhancement is rather complicated to deal with, for the operations of taking cohomology at each step are in general very complicated to manage. 
We shall provide an equivalent but in our opinion simpler manner to handle with it. 
Suppose $(A,m_{\bullet})$ is an $A_{\infty}$ algebra over $k$ with a compatible bigrading of $s$-th type, for some $s \in \NN_{0}$, and let $(A_{\hbar},\tilde{m}_{\bullet})$ be a formal bigraded deformation. 
We define in this case the $A_{\infty}$-algebra $\operatorname{D}(A_{\hbar})$ over $k[\hbar]$ whose underlying bigraded $k$-module is 
\[     \{ a \in A_{\hbar} : \text{there exists $b \in A_{\hbar}$ such that } \tilde{m}_{1}(a) = \hbar . b \},     \] 
together with the structure maps $\tilde{m}_{i}^{\textrm{D}} = \hbar^{i-2} \tilde{m}_{i}|_{\operatorname{D}(A_{\hbar})}$, for $i \in \NN$. 
Note that the structure maps are well-defined, by definition of $\operatorname{D}(A_{\hbar})$, \textit{i.e.} the image of $\tilde{m}_{i}^{\textrm{D}}$ lies inside $\operatorname{D}(A_{\hbar})$.   
Following the arguments given at the second paragraph of this section for proving that $\operatorname{T}(A_{\hbar})$ is an $A_{\infty}$-algebra over $k[\hbar]$, 
we see that $\operatorname{D}(A_{\hbar})$ is also. 
Furthermore, $\operatorname{D}(A_{\hbar})$ is $\hbar$-torsion-free and bigraded complete, as one may conclude by a simple convergence argument (because $A_{\hbar}$ is bigraded complete and $\hbar . A_{\hbar} \subseteq D(A_{\hbar})$ by definition), so it is an $\hbar$-topologically free bigraded $k[\hbar]$-module, 
by the comments at the second paragraph before Remark \ref{remark:defofor}. 
By the explanation in the latter remark we see that $\operatorname{D}(A_{\hbar})$ is in fact a formal bigraded deformation. 
In fact, it is easily seen that is a formal bigraded deformation of an $A_{\infty}$-algebra provided with a compatible bigrading of $(s+1)$-th type. 
Moreover, if $A$ is minimal, so $m_{1}$ vanishes, $\operatorname{D}(A_{\hbar}) = \operatorname{T}(A_{\hbar})$. 
As before it is direct to see that $\operatorname{D}$ defines a functor from the category of formal bigraded deformations of $A_{\infty}$-algebras provided with a compatible bigrading of $s$-th type, 
for some $s \in \NN_{0}$, to itself. 
Indeed, let $A_{\hbar}$ and $B_{\hbar}$ be formal bigraded deformations of the (resp., unitary) $A_{\infty}$-algebras $A$ and $B$ provided with compatible bigradings of $s$-th type, for some $s \in \NN_{0}$, 
and let $(f_{\bullet}, \tilde{f}_{\bullet})_{\bullet \in \NN}$ be a morphism of formal bigraded deformations from $A_{\hbar}$ to $B_{\hbar}$. 
The image of this morphism under $\operatorname{D}$ is the pair $(f_{1},\operatorname{D}(\tilde{f}))$ with the morphism $\operatorname{D}(\tilde{f})$ from $\operatorname{T}(A)$ to $\operatorname{T}(B)$ 
whose $i$-th component $\operatorname{D}(\tilde{f})_{i}$ is given by $\hbar^{i-1} \tilde{f}_{i}|_{\operatorname{D}(A_{\hbar})}$, for $i \in \NN$. 

We claim that the functor $\operatorname{D}$ preserves quasi-isomorphisms. 
This is a direct computation, but we provide it nonetheless for clarity. 
Suppose that $(f_{\bullet}, \tilde{f}_{\bullet})_{\bullet \in \NN}$ is a quasi-isomorphism of formal bigraded deformations from $(A_{\hbar},\tilde{m}_{\bullet})$ to $(A'_{\hbar},\tilde{m}'_{\bullet})$. 
To prove that $\operatorname{D}(\tilde{f})_{1}$ induces an injective map in cohomology 
means that given $a \in A_{\hbar}$ such that $\tilde{m}_{1}(a) = 0$ (or, equivalently, $a \in \operatorname{D}(A_{\hbar})$ such that $\tilde{m}_{1}^{\textrm{D}}(a) = 0$), 
and given $a' \in D(A'_{\hbar})$ such that $\tilde{f}_{1}(a) = \hbar^{-1} \tilde{m}'_{1}(a')$, there exists $\alpha \in \operatorname{D}(A_{\hbar})$ such that $a = \hbar^{-1} \tilde{m}_{1}(\alpha)$. 
The equality $\tilde{f}_{1}(a) = \hbar^{-1} \tilde{m}'_{1}(a')$ can be rewritten as 
\[     \tilde{m}'_{1}(a') = \hbar \tilde{f}_{1}(a) = \tilde{f}_{1}(\hbar a).     \]
Now, since $\tilde{f}_{1}$ is a quasi-isomorphism from $A_{\hbar}$ to $A'_{\hbar}$, we see that 
there exists $\alpha \in A_{\hbar}$ such that  $\hbar a = \tilde{m}_{1}(\alpha)$. 
This in particular implies that $\alpha \in \operatorname{D}(A_{\hbar})$ and the required equality, which proves the previously claimed injectivity. 
Let us now show that $\operatorname{D}(\tilde{f})_{1}$ induces a surjective map in cohomology. 
This is equivalent to prove that given $a' \in A'_{\hbar}$ such that $\tilde{m}'_{1}(a')$  vanishes (or, equivalently, $a' \in \operatorname{D}(A'_{\hbar})$ such that $(\tilde{m}'_{1})^{\textrm{D}}(a') = 0$), 
there exists $a \in A_{\hbar}$ such that $\tilde{m}_{1}(a) = 0$ (or, equivalently, an element of $a \in \operatorname{D}(A_{\hbar})$ satisfying that $\tilde{m}_{1}^{\textrm{D}}(a) = 0$) and an element $\alpha' \in \operatorname{D}(A'_{\hbar})$
such that $\tilde{f}_{1}(a) = a' + \hbar^{-1} \tilde{m}'_{1}(\alpha')$. 
However, since $\tilde{f}_{1}$ is a quasi-isomorphism from $A_{\hbar}$ to $A'_{\hbar}$, we get that, given $a'$ as before, there exists $a \in A_{\hbar}$ such that $\tilde{m}_{1}(a) = 0$ (so \textit{a fortiori} an element of $\operatorname{D}(A_{\hbar})$) and $\alpha'' \in A'_{\hbar}$ such that $\tilde{f}_{1}(a) = a' + \tilde{m}'_{1}(\alpha'')$. 
By taking $\alpha' = \hbar \alpha''$, which clearly belongs to $\operatorname{D}(A'_{\hbar})$, we get the desired identity which shows the claimed surjectivity.      
The quasi-isomorphism property of $\operatorname{D}(f)_{1}^{0}$ follows directly (see Remark \ref{remark:quasi-isodefbis}). 

As a consequence, we get that the formal bigraded deformation $\operatorname{D}(A_{\hbar})$ and $\operatorname{T}(H^{\bullet}(A)_{\hbar})$ are in fact quasi-isomorphic. 
Indeed, by Theorem \ref{thm:quasiisofor}, there exists a formal bigraded deformation $H^{\bullet}(A)_{\hbar}$, which is quasi-isomorphic to $A_{\hbar}$. 
Then the formal bigraded deformations $\operatorname{D}(H^{\bullet}(A)_{\hbar})$ and $\operatorname{D}(A_{\hbar})$ are also quasi-isomorphic. 
In view of the fact that $\operatorname{D}(H^{\bullet}(A)_{\hbar})$ coincides with $\operatorname{T}(H^{\bullet}(A)_{\hbar})$, for $H^{\bullet}(A)_{\hbar}$ is a deformation of a minimal $A_{\infty}$-algebra, the claim follows. 
This implies that we may build the associated family of formal bigraded deformations (up to quasi-isomorphism) in a more simple fashion, by defining 
$({}^{r}A'_{\hbar},{}^{r}\tilde{m}'_{\bullet})$, for $r \in \NN_{0}$, as the application of the composition functor $\operatorname{D}^{r}$ to $A_{\hbar}$, 
where $\operatorname{D}^{0}$ is the identity functor. 
We shall call this collection of formal bigraded deformations the \emph{reduced associated family of formal bigraded deformations}.
We note furthermore that our construction is easier to handle, due to the fact that, for instance $\operatorname{D}$ produces dg algebra deformation out of dg algebra deformations, 
unlike the usual associated family of formal bigraded deformations. 

Suppose $A$ is an $A_{\infty}$-algebra provided with a compatible bigrading of $s'$-th type, for some $s' \in \NN_{0}$, together with a formal bigraded deformation $A_{\hbar}$.  
We note it determines a collection of exact couples of $(0,s+s')$-th type, for each $s \in \NN_{0}$, as follows. 
The procedure is exactly the same as for exact couples associated to filtrations. 
Let $s \in \NN_{0}$, and consider the collection of short exact sequences 
\begin{equation}
\label{eq:sesdefo}
     0 \rightarrow \operatorname{D}^{s}(A_{\hbar}) \overset{{}^{s}i}{\rightarrow} \operatorname{D}^{s}(A_{\hbar}) \overset{{}^{s}p}{\rightarrow} \operatorname{D}^{s}(A_{\hbar})/(\hbar) \rightarrow 0     
\end{equation}
of dg modules over $k[\hbar]$ (or of complexes of $k$-modules), which we are going to denote by ${}^{s}\mathcal{C}$, 
where the map ${}^{s}i$ is the homogeneous morphism of bidegree $(-1,1)$ given by multiplication by $\hbar$, 
and ${}^{s}p$ is the canonical projection, which is a homogeneous morphism of bidegree $(0,0)$. 
By taking cohomology one gets an exact couple of $(0,s+s')$-th type 
\begin{equation}
\label{eq:diag}
\xymatrix
{
H^{\bullet}(\operatorname{D}^{s}(A_{\hbar})) 
\ar[rr]^{H^{\bullet}({}^{s}i)}
&
&
H^{\bullet}(\operatorname{D}^{s}(A_{\hbar}))
\ar[dl]^{H^{\bullet}({}^{s}p)}
\\
&
H^{\bullet}(\operatorname{D}^{s}(A_{\hbar})/(\hbar))
\ar[ul]^{{}^{s}\delta^{\bullet}}
&
}
\end{equation}
where we recall that ${}^{s}\delta^{\bullet}$ is the delta morphism obtained by means of the Snake lemma. 
The spectral sequence starting at $(s+s'+1)$ it defines is also multiplicative, with the product induced by that of $\operatorname{D}^{s}(A_{\hbar})$. 

\begin{remark}
\label{remark:rees}
If $B$ is a filtered (graded complete) $A_{\infty}$-algebra, and $A_{\hbar} = \mathrm{Re}_{F^{\bullet}B}(B)$ is the Rees $A_{\infty}$-algebra associated to $B$, then, given $s \in \NN_{0}$, 
the underlying complex of $D^{s}(A_{\hbar})$ coincides with the dg $k[\hbar]$-module provided with a compatible bigrading given by \eqref{eq:reescomplex}, and the underlying complex of $D^{s}(A_{\hbar})/(\hbar)$ 
coincides with \eqref{eq:gr}. 
This tells us in particular that, for a formal deformation given by the Rees $A_{\infty}$-algebra of a filtered (graded complete) $A_{\infty}$-algebra, 
the collection of multiplicative spectral sequences associated to the short exact sequences \eqref{eq:sesdefo} is compatible (by Remark \ref{remark:impa_inf}).  
\end{remark}

\begin{example}
\label{example:examplebis}
Suppose we have a formal bigraded deformation $(A_{\hbar}, \tilde{m}_{\bullet})$ of an $A_{\infty}$-algebra $(A,m_{\bullet})$ provided with a compatible bigrading of $s$-th type, for some $s \in \NN_{0}$. 
The maps of the exact couple of $(0,s)$-th type together with its derived exact couple can be more precisely described as follows. 
As usual, a typical element of $A_{\hbar}$ will be denoted by $\sum_{j \in \NN_{0}} a_{j} \hbar^{j}$, for $a_{j} \in A$. 
First, the map $H^{\bullet}({}^{s}i)$ is induced by the multiplication by $\hbar$, so it can be written out explicitly as
\[     \sum_{j \in \NN_{0}} a_{j} \hbar^{j} + \mathrm{Im}(\tilde{m}_{1}) \mapsto \sum_{j \in \NN} a_{j-1} \hbar^{j} + \mathrm{Im}(\tilde{m}_{1}),     \]
where $\sum_{j \in \NN_{0}} a_{j} \hbar^{j} \in \Ker(\tilde{m}_{1})$, which is equivalent to the family of identities 
\begin{equation}
\label{eq:ec}
     \sum_{i=0}^{j} m_{1}^{i}(a_{j-i}) = 0,     
\end{equation}
for all $j \in \NN_{0}$.
The map $H^{\bullet}({}^{s}j)$ is induced by the canonical quotient $A_{\hbar} \rightarrow A$, so it is just 
\[     \sum_{j \in \NN_{0}} a_{j} \hbar^{j} + \mathrm{Im}(\tilde{m}_{1}) \mapsto a_{0} + \mathrm{Im}(m_{1}^{0}),     \]
for $\sum_{j \in \NN_{0}} a_{j} \hbar^{j} \in \Ker(\tilde{m}_{1})$. 
The map ${}^{s}\delta^{\bullet}$ is constructed by means of Snake lemma and it gives 
\[     a + \mathrm{Im}(m_{1}) \mapsto \sum_{j \in \NN} m_{1}^{j+1}(a) \hbar^{j} + \mathrm{Im}(\tilde{m}_{1}),     \]
for $a \in \Ker(m_{1})$. 
The previous description implies in particular that the $(s+1)$-th term of the spectral sequence starting at $(s+1)$, which we shall denote by $(E'_{s+1},d'_{s+1})$, associated to this exact couple 
we described in the paragraph before the previous remark in fact coincides with the $(s+1)$-th term of the spectral sequence $(E_{s+1},d_{s+1})$ starting at $(s+1)$ constructed in Example \ref{example:example}.

The derived couple of \eqref{eq:diag} is given thus by 
\[
\xymatrix
{
H^{\bullet}(\operatorname{D}^{s}(A_{\hbar}))' 
\ar[rr]^{H^{\bullet}({}^{s}i)'}
&
&
H^{\bullet}(\operatorname{D}^{s}(A_{\hbar}))'
\ar[dl]^{H^{\bullet}({}^{s}p)'}
\\
&
H^{\bullet}(\operatorname{D}^{s}(A_{\hbar})/(\hbar))'
\ar[ul]^{({}^{s}\delta^{\bullet})'}
&
}
\]
We shall describe briefly how it looks like. 
By the comments at the last sentence of the previous paragraph we see that the bottom bigraded $k$-module, which we shall denote by $E'_{s+2}$, in fact coincides with the underlying space $E_{s+2}$ 
of the $(s+2)$-th term of the spectral sequence starting at $(s+1)$ constructed in Example \ref{example:example}. 
We shall prove that the corresponding differentials coincide. 
In order to do so, we only need to describe the morphisms of the derived exact couple. 
First, note the easy fact that a typical element of $H^{\bullet}(\operatorname{D}^{s}(A_{\hbar}))'$ can be written as 
\[     \sum_{j \in \NN} a_{j-1} \hbar^{j} + \mathrm{Im}(\tilde{m}_{1}),     \]
where the set of elements $a_{j} \in A$ satisfy the same identities as in \eqref{eq:ec}. 
Now, the map $H^{\bullet}({}^{s}i)'$ is just as before, \textit{i.e.}
\[     \sum_{j \in \NN} a_{j} \hbar^{j} + \mathrm{Im}(\tilde{m}_{1}) \mapsto \sum_{j \in \NN_{\geq 2}} a_{j-1} \hbar^{j} + \mathrm{Im}(\tilde{m}_{1}),     \]
where $\sum_{j \in \NN} a_{j} \hbar^{j} \in \mathrm{Ker}(\tilde{m}_{1})$ is an element of the image of $H^{\bullet}({}^{s}i)$. 
The map $H^{\bullet}({}^{s}j)'$ is given by
\[     \sum_{j \in \NN} a_{j} \hbar^{j} + \mathrm{Im}(\tilde{m}_{1}) \mapsto a_{0} + \Big(\mathrm{Im}(m_{1}^{1}|_{\Ker(m_{1}^{0})}) + \mathrm{Im}(m_{1}^{0})\Big),     \]
for $\sum_{j \in \NN} a_{j} \hbar^{j} \in \Ker(\tilde{m}_{1})$ is also an element of the image of $H^{\bullet}({}^{s}i)$. 
The map $({}^{s}\delta^{\bullet})'$ has the form  
\[     a + \Big(\mathrm{Im}(m_{1}^{1}|_{\Ker(m_{1}^{0})}) + \mathrm{Im}(m_{1}^{0})\Big) \mapsto \sum_{j \in \NN} m_{1}^{j+1}(a) \hbar^{j} + \mathrm{Im}(\tilde{m}_{1}),     \]
for $a \in \Ker(m_{1}^{0}) \cap (m_{1}^{1})^{-1}(\Ker(m_{1}^{0}))$, which means that $m_{1}^{0}(a)=0$ and there exists $a' \in A$ such that $m_{1}^{1}(a)=m_{1}^{0}(a')$. 
However, since $\tilde{m}_{1}(a')$ goes to zero in $E'_{s+2}$ by definition, $({}^{s}\delta^{\bullet})'$ can also be given as 
\[     a + \Big(\mathrm{Im}(m_{1}^{1}|_{\Ker(m_{1}^{0})}) + \mathrm{Im}(m_{1}^{0})\Big) \mapsto \sum_{j \in \NN} (m_{1}^{j+1}(a) - m_{1}^{j}(a')) \hbar^{j} + \mathrm{Im}(\tilde{m}_{1}),     \]
This implies in particular that $H^{\bullet}({}^{s}j)' \circ ({}^{s}\delta^{\bullet})'$ coincides with the map $d_{s+2}$ given in Example \ref{example:example}. 
Hence, the first two terms of the spectral sequences starting at $(s+1)$ associated to the exact couple \eqref{eq:diag} and of the one considered in the mentioned remark coincide.   
\end{example}

The conclusion of the previous two remarks can be strengthen as follows. 
Let $A$ be an $A_{\infty}$-algebra provided with a compatible bigrading of $0$-th type, together with a formal bigraded deformation $A_{\hbar}$. 
We first claim that the collection of multiplicative spectral sequences constructed from the family of short exact sequences given by \eqref{eq:sesdefo} are always compatible, 
regardless of the chosen bigraded formal deformation. 
This follows from the fact that any bigraded formal deformation can be regarded as the Rees $A_{\infty}$-algebra of a filtered (graded complete) $A_{\infty}$-algebra, by Proposition \eqref{proposition:fil-defo}, 
so the claim follows from Remark \ref{remark:rees}. 
Now, from the previous compatibility condition,  we see that the multiplicative spectral sequence associated to the exact couple with $s=0$ given by \eqref{eq:diag} is isomorphic to the 
multiplicative spectral sequence associated to the underlying deformation of complexes of the formal bigraded deformation $A_{\hbar}$ of $A$. 
This follows simply by the quasi-isomorphism of formal bigraded deformations between $\operatorname{D}^{s}(A_{\hbar})$ and ${}^{s}A_{\hbar}$, where the latter were defined at the paragraph before 
Remark \ref{remark:laseq}, for this induces isomorphic multiplicative spectral sequences. 
We have thus proved the following result. 
\begin{proposition}
\label{proposition:compatres}
Let $A$ be an $A_{\infty}$-algebra provided with a compatible bigrading of $0$-th type, together with a formal bigraded deformation $A_{\hbar}$. 
Then, the multiplicative spectral sequence associated to the exact couple with $s=0$ given by \eqref{eq:diag} is isomorphic to the multiplicative spectral sequence associated 
to the formal bigraded deformation $A_{\hbar}$ of $A$ defined at the paragraph before Remark \ref{remark:laseq}. 
\end{proposition}

As another consequence of the arguments at the previous paragraphs, we have the following immediate result, which also justifies the interest in the $A_{\infty}$-algebra enhancements of multiplicative 
spectral sequences. 
\begin{proposition}
\label{proposition:final}
Given a filtered $A_{\infty}$-algebra $(A, m_{\bullet})$, with filtration $\{ F^{\bullet}A \}_{\bullet \in \ZZ}$, provided with a compatible bigrading of $0$-th type, then the multiplicative spectral sequence associated to the filtration is enhanced with an 
$A_{\infty}$-algebra structure. 
\end{proposition}
\begin{proof}
Consider the formal bigraded deformation of $\mathrm{Gr}_{F^{\bullet}A}(A)$ given by the Rees $A_{\infty}$-algebra $\mathrm{Re}_{F^{\bullet}A}(A)$, as explained in Section \ref{section:fildefo}. 
By the previous proposition we see that the multiplicative spectral sequence associated to this filtration is isomorphic to the multiplicative spectral sequence associated 
to the formal bigraded deformation $A_{\hbar}$ of $A$, so the multiplicative spectral sequence associated to the filtration is enhanced with an 
$A_{\infty}$-algebra structure, and the proposition is proved.  
\end{proof}

\bibliographystyle{model1-num-names}

\begin{bibdiv}
\begin{biblist}

\bib{BP}{article}{
   author={Basu, Saugata},
   author={Parida, Laxmi},
   title={Spectral sequences, exact couples, and persistent homology of filtrations},
   date={2013}, 
   eprint={http://arxiv.org/abs/1308.0801},
}

\bib{BM}{article}{
   author = {Belch\'{\i}, Francisco},
   author={Murillo, Aniceto},
   title={$A_{\infty}$-persistence}, 
   date={2014},
   eprint={http://arxiv.org/abs/1403.2395},
}

\bib{Ben}{book}{
   author={Benson, D. J.},
   title={Representations and cohomology. II},
   series={Cambridge Studies in Advanced Mathematics},
   volume={31},
   edition={2},
   note={Cohomology of groups and modules},
   publisher={Cambridge University Press, Cambridge},
   date={1998},
   pages={xii+279},
}

\bib{BGMP}{article}{
   author={Blumberg, Andrew J.},
   author={Gal, Itamar},
   author={Mandell, Michael A.},
   author={Pancia, Matthew},
   title={Robust statistics, hypothesis testing, and confidence intervals
   for persistent homology on metric measure spaces},
   journal={Found. Comput. Math.},
   volume={14},
   date={2014},
   number={4},
   pages={745--789},
}

\bib{Bour2}{book}{
   author={Bourbaki, N.},
   title={\'El\'ements de math\'ematique. Fascicule XXVIII. Alg\`ebre
   commutative. Chapitre 3: Graduations, filtra- tions et topologies.
   Chapitre 4: Id\'eaux premiers associ\'es et d\'ecomposition primaire},
   language={French},
   series={Actualit\'es Scientifiques et Industrielles, No. 1293},
   publisher={Hermann, Paris},
   date={1961},
   pages={183},
}

\bib{ELZ}{article}{
   author={Edelsbrunner, Herbert},
   author={Letscher, David},
   author={Zomorodian, Afra},
   title={Topological persistence and simplification},
   note={Discrete and computational geometry and graph drawing (Columbia,
   SC, 2001)},
   journal={Discrete Comput. Geom.},
   volume={28},
   date={2002},
   number={4},
   pages={511--533},
}

\bib{FP}{article}{
   author={Fialowski, Alice},
   author={Penkava, Michael},
   title={Deformation theory of infinity algebras},
   journal={J. Algebra},
   volume={255},
   date={2002},
   number={1},
   pages={59--88},
}

\bib{Ger64}{article}{
   author={Gerstenhaber, Murray},
   title={On the deformation of rings and algebras},
   journal={Ann. of Math. (2)},
   volume={79},
   date={1964},
   pages={59--103},
}

\bib{GS92}{article}{
   author={Gerstenhaber, Murray},
   author={Schack, Samuel D.},
   title={Algebras, bialgebras, quantum groups, and algebraic deformations},
   conference={
      title={Deformation theory and quantum groups with applications to
      mathematical physics},
      address={Amherst, MA},
      date={1990},
   },
   book={
      series={Contemp. Math.},
      volume={134},
      publisher={Amer. Math. Soc., Providence, RI},
   },
   date={1992},
   pages={51--92},
}

\bib{GW96}{article}{
   author={Gerstenhaber, Murray},
   author={Wilkerson, Clarence W.},
   title={On the deformation of rings and algebras. V. Deformation of
   differential graded algebras},
   conference={
      title={Higher homotopy structures in topology and mathematical physics
      },
      address={Poughkeepsie, NY},
      date={1996},
   },
   book={
      series={Contemp. Math.},
      volume={227},
      publisher={Amer. Math. Soc., Providence, RI},
   },
   date={1999},
   pages={89--101},
}

\bib{GLS}{article}{
   author={Gugenheim, V. K. A. M.},
   author={Lambe, L. A.},
   author={Stasheff, J. D.},
   title={Perturbation theory in differential homological algebra. II},
   journal={Illinois J. Math.},
   volume={35},
   date={1991},
   number={3},
   pages={357--373},
}

\bib{Herrec}{article}{
   author={Herscovich, Estanislao},
   title={Hochschild (co)homology and Koszul duality},
   eprint={http://arxiv.org/abs/1405.2247},
}

\bib{HR}{article}{
   author={Heyneman, Robert G.},
   author={Radford, David E.},
   title={Reflexivity and coalgebras of finite type},
   journal={J. Algebra},
   volume={28},
   date={1974},
   pages={215--246},
}

\bib{K80}{article}{
   author={Kadei{\v{s}}vili, T. V.},
   title={On the theory of homology of fiber spaces},
   language={Russian},
   note={International Topology Conference (Moscow State Univ., Moscow,
   1979)},
   journal={Uspekhi Mat. Nauk},
   volume={35},
   date={1980},
   number={3(213)},
   pages={183--188},
}

\bib{K82}{article}{
   author={Kadeishvili, T. V.},
   title={The algebraic structure in the homology of an $A(\infty)$-algebra},
   language={Russian, with English and Georgian summaries},
   journal={Soobshch. Akad. Nauk Gruzin. SSR},
   volume={108},
   date={1982},
   number={2},
   pages={249--252 (1983)},
}

\bib{Ka}{book}{
   author={Kassel, Christian},
   title={Quantum groups},
   series={Graduate Texts in Mathematics},
   volume={155},
   publisher={Springer-Verlag, New York},
   date={1995},
   pages={xii+531},
}

\bib{Kos47}{article}{
   author={Koszul, Jean-Louis},
   title={Sur les op\'erateurs de d\'erivation dans un anneau},
   language={French},
   journal={C. R. Acad. Sci. Paris},
   volume={225},
   date={1947},
   pages={217--219},
}

\bib{Kos47b}{article}{
   author={Koszul, Jean-Louis},
   title={Sur l'homologie des espaces homog\`enes},
   language={French},
   journal={C. R. Acad. Sci. Paris},
   volume={225},
   date={1947},
   pages={477--479},
}

\bib{La01}{article}{
   author={Lapin, S. V.},
   title={Differential perturbations and $D_\infty$-differential modules},
   language={Russian, with Russian summary},
   journal={Mat. Sb.},
   volume={192},
   date={2001},
   number={11},
   pages={55--76},
   issn={0368-8666},
   translation={
      journal={Sb. Math.},
      volume={192},
      date={2001},
      number={11-12},
      pages={1639--1659},
      issn={1064-5616},
   },
}

\bib{La02a}{article}{
   author={Lapin, S. V.},
   title={$D_\infty$-differential $A_\infty$-algebras and spectral
   sequences},
   language={Russian, with Russian summary},
   journal={Mat. Sb.},
   volume={193},
   date={2002},
   number={1},
   pages={119--142},
   translation={
      journal={Sb. Math.},
      volume={193},
      date={2002},
      number={1-2},
      pages={119--142},
   },
}

\bib{La02b}{article}{
   author={Lapin, S. V.},
   title={$(DA)_\infty$-modules over $(DA)_\infty$-algebras, and
   spectral sequences},
   language={Russian, with Russian summary},
   journal={Izv. Ross. Akad. Nauk Ser. Mat.},
   volume={66},
   date={2002},
   number={3},
   pages={103--130},
   translation={
      journal={Izv. Math.},
      volume={66},
      date={2002},
      number={3},
      pages={543--568},
   },
}

\bib{La08}{article}{
   author={Lapin, S. V.},
   title={Multiplicative $A_\infty$-structure in the terms of spectral
   sequences of fibrations},
   language={Russian, with English and Russian summaries},
   journal={Fundam. Prikl. Mat.},
   volume={14},
   date={2008},
   number={6},
   pages={141--175},
   translation={
      journal={J. Math. Sci. (N. Y.)},
      volume={164},
      date={2010},
      number={1},
      pages={95--118},
   },
}

\bib{LH}{thesis}{
   author={Lef\`evre-Hasegawa, Kenji},
   title={Sur les $A_{\infty}$-cat\'egories},
   language={French},
   type={Thesis (Ph.D.)--Universit\'e Paris 7},
   place={Paris, France},
   date={2003},
   eprint={http://arxiv.org/abs/math/0310337},
   note={Corrections at \texttt{http://www.math.jussieu.fr/~keller/lefevre/TheseFinale/corrainf.pdf}},
}

\bib{Ler46}{article}{
   author={Leray, Jean},
   title={Structure de l'anneau d'homologie d'une repr\'esentation},
   language={French},
   journal={C. R. Acad. Sci. Paris},
   volume={222},
   date={1946},
   pages={1419--1422},
}

\bib{LPWZ04}{article}{
   author={Lu, D. M.},
   author={Palmieri, J. H.},
   author={Wu, Q. S.},
   author={Zhang, J. J.},
   title={$A_\infty$-algebras for ring theorists},
   booktitle={Proceedings of the International Conference on Algebra},
   journal={Algebra Colloq.},
   volume={11},
   date={2004},
   number={1},
   pages={91--128},
}

\bib{LPWZ09}{article}{
   author={Lu, D.-M.},
  author={Palmieri, J. H.},
   author={Wu, Q.-S.},
   author={Zhang, J. J.},
   title={$A$-infinity structure on Ext-algebras},
   journal={J. Pure Appl. Algebra},
   volume={213},
   date={2009},
   number={11},
   pages={2017--2037},
}

\bib{Lun}{article}{
   author={Lunts, Valery A.},
   title={Formality of DG algebras (after Kaledin)},
   journal={J. Algebra},
   volume={323},
   date={2010},
   number={4},
   pages={878--898},
}

\bib{Ma52}{article}{
   author={Massey, W. S.},
   title={Exact couples in algebraic topology. I, II},
   journal={Ann. of Math. (2)},
   volume={56},
   date={1952},
   pages={363--396},
}

\bib{Ma53}{article}{
   author={Massey, W. S.},
   title={Exact couples in algebraic topology. III, IV, V},
   journal={Ann. of Math. (2)},
   volume={57},
   date={1953},
   pages={248--286},
}

\bib{Ma54}{article}{
   author={Massey, W. S.},
   title={Products in exact couples},
   journal={Ann. of Math. (2)},
   volume={59},
   date={1954},
   pages={558--569},
}

\bib{McC99}{article}{
   author={McCleary, John},
   title={A history of spectral sequences: origins to 1953},
   conference={
      title={History of topology},
   },
   book={
      publisher={North-Holland, Amsterdam},
   },
   date={1999},
   pages={631--663},
}

\bib{McC01}{book}{
   author={McCleary, John},
   title={A user's guide to spectral sequences},
   series={Cambridge Studies in Advanced Mathematics},
   volume={58},
   edition={2},
   publisher={Cambridge University Press, Cambridge},
   date={2001},
   pages={xvi+561},
}

\bib{Mer}{article}{
   author={Merkulov, S. A.},
   title={Strong homotopy algebras of a K\"ahler manifold},
   journal={Internat. Math. Res. Notices},
   date={1999},
   number={3},
   pages={153--164},
}

\bib{Prou}{article}{
   author={Prout{\'e}, Alain},
   title={$A_\infty$-structures. Mod\`eles minimaux de Baues-Lemaire et
   Kadeishvili et homologie des fibrations},
   language={French},
   note={Reprint of the 1986 original;
   With a preface to the reprint by Jean-Louis Loday},
   journal={Repr. Theory Appl. Categ.},
   number={21},
   date={2011},
   pages={1--99},
}

\bib{Se50I}{article}{
   author={Serre, Jean-Pierre},
   title={Homologie singuli\`ere des espaces fibr\'es. I. La suite
   spectrale},
   language={French},
   journal={C. R. Acad. Sci. Paris},
   volume={231},
   date={1950},
   pages={1408--1410},
}

\bib{Se50II}{article}{
   author={Serre, Jean-Pierre},
   title={Homologie singuli\`ere des espaces fibr\'es. II. Les espaces de
   lacets},
   language={French},
   journal={C. R. Acad. Sci. Paris},
   volume={232},
   date={1951},
   pages={31--33},
}
		
\bib{Se50III}{article}{
   author={Serre, Jean-Pierre},
   title={Homologie singuli\`ere des espaces fibr\'es. III. Applications
   homotopiques},
   language={French},
   journal={C. R. Acad. Sci. Paris},
   volume={232},
   date={1951},
   pages={142--144},
}

\bib{Se50a}{article}{
   author={Serre, Jean-Pierre},
   title={Homologie singuli\`ere des espaces fibr\'es. Applications},
   language={French},
   journal={Ann. of Math. (2)},
   volume={54},
   date={1951},
   pages={425--505},
}

\bib{Sta}{article}{
   author={Stasheff, James Dillon},
   title={Homotopy associativity of $H$-spaces. I, II},
   journal={Trans. Amer. Math. Soc. 108 (1963), 275-292; ibid.},
   volume={108},
   date={1963},
   pages={293--312},
}

\bib{SS93}{book}{
   author={Shnider, Steven},
   author={Sternberg, Shlomo},
   title={Quantum groups},
   series={Graduate Texts in Mathematical Physics, II},
   note={From coalgebras to Drinfel\cprime d algebras;
   A guided tour},
   publisher={International Press, Cambridge, MA},
   date={1993},
   pages={xxii+496},
}

\bib{W}{book}{
   author={Weibel, Charles A.},
   title={An introduction to homological algebra},
   series={Cambridge Studies in Advanced Mathematics},
   volume={38},
   publisher={Cambridge University Press},
   place={Cambridge},
   date={1994},
   pages={xiv+450},
}

\bib{Wu}{thesis}{
   author={Wu, E.},
   title={Deformation and Hochschild cohomology of $A_{\infty}$-algebras},
   place={Shanghai, China},
   type={Thesis (M.Sc.)--Zhejiang University},
   date={2002},
   eprint={http://www.math.uwo.ca/~ewu22/deformationofA-infinity.dvi},
}

\end{biblist}
\end{bibdiv}
\addcontentsline{toc}{section}{References}



\end{document}